\documentclass[twoside,11pt,hidelinks]{amsart}
\usepackage[utf8]{inputenc}
\usepackage[english]{babel}
\usepackage[fixlanguage]{babelbib}
\usepackage{amssymb, amsfonts, amsmath, amsthm, amscd, mathtools, thmtools, tikz, xcolor, hyperref, a4wide}
\usepackage[numbers,square]{natbib}
\usepackage[normalem]{ulem}

\newcounter{contador}
\numberwithin{contador}{section}

\newtheorem{theorem}[contador]{Theorem}
\newtheorem{prop}[contador]{Proposition}
\newtheorem{lemma}[contador]{Lemma}
\newtheorem{corollary}[contador]{Corollary}

\theoremstyle{definition}
\newtheorem{defi}[contador]{Definition}
\newtheorem{obs}[contador]{Remark}
\newtheorem{exe}[contador]{Example}

\newcommand{\Not}[2][0]{	
	\setcounter{enumi}{#1}
	\renewcommand{\theenumi}{#2\arabic{enumi}}
	\renewcommand{\labelenumi}{(\theenumi)}
	\setlength{\itemindent}{\widthof{#2}}
	\setlength{\itemsep}{4pt}
}

\DeclareRobustCommand{\SkipTocEntry}[5]{}
\makeatletter
    \let\oldsubsection\subsection
    \renewcommand{\subsection}{\@ifstar{\astsection}{\noastsection}}
    
    \newcommand{\astsection}[1]{
        \addtocontents{toc}{\SkipTocEntry}
        \oldsubsection*{#1}
    }
    
    \newcommand{\noastsection}[2][]{
        \ifx\relax#1\relax
            \oldsubsection{#2}
        \else
            \oldsubsection[#1]{#2}
        \fi
    }
\makeatother

\allowdisplaybreaks 

\title{Reflectors and the Globalization Problem for Partial Actions of Semiconstellations}
\author[Haag and Tamusiunas]{Rafael Haag and Thaísa Tamusiunas$^*$}
\address{Instituto de Matem\'{a}tica, Universidade Federal do Rio Grande do Sul,  Av. Bento Gon\c{c}alves, 9500, 91509-900. Porto Alegre-RS, Brazil}
\email{rafaelpetasny@gmail.com}
\email{thaisa.tamusiunas@gmail.com}
\thanks{$^*$ Corresponding author}
\date{} 

\begin{document}

    
    \begin{abstract}
      We introduce semiconstellations and restriction semiconstellations, together with their actions on sets, providing a common setting for partial actions of categories, inverse semigroupoids, restriction semigroups, and semigroupoids. We then investigate the \emph{globalization problem}, which asks whether every partial action can, up to isomorphism, be obtained by restricting a global action. A distinctive feature of our approach is the use of reflectors not only to construct globalizations, but also to study when partial actions are globalizable. We show that every partial action of a restriction semiconstellation admits a reflector in the full subcategory of $S$-algebras consisting of what we call \emph{almost global actions}, and investigate conditions under which this reflector provides a globalization. As an application, we obtain a positive answer to the globalization problem for partial actions of restriction semigroupoids on sets, and consequently for restriction semigroups and restriction categories. In contrast, the answer is negative for restriction semiconstellations in general. Finally, given a semiconstellation $S$ and a suitable equivalence relation $R$ on $S$, we introduce the notion of an \emph{$R$-compatible action of} $S$ and construct a restriction semiconstellation $S^R$. We then show that partial and $R$-compatible actions of $S$ correspond, respectively, to partial and global actions of $S^R$. In this way, the globalization problem for semiconstellations can be reduced to the corresponding problem for restriction semiconstellations.
    \end{abstract}

    \maketitle
    \tableofcontents

    \section{Introduction} \label{sec:1}

    A partial action of an algebraic structure $S$ on a set $X$ is, roughly speaking, a function from $S$ to the set of partial transformations of $X$ satisfying appropriate conditions, whereas global actions are the functions that are compatible with the operations of $S$. The aim of this paper is to discuss when a partial action can be obtained from a global one. One of the first occurrences of partial actions in the literature is due to R. Exel \cite{exel1994} in the study of the structure of $C^\ast$-algebras possessing actions of the circle group. In \cite{exel1998partial}, Exel introduced partial actions of abstract groups on sets, providing a more general environment for addressing questions arising in the theory of $C^\ast$-algebras. Since then, the notion of partial actions on sets has been extended to a wide variety of algebraic structures, including monoids \cite{schroder2004}, inverse semigroups \cite{lawson2006expansions}, weakly left $E$-ample semigroups \cite{gould2009partial}, ordered groupoids \cite{gilbert2005actions}, categories \cite{nystedt2017}, inverse semigroupoids \cite{demeneghi2025}, semigroups \cite{kudryavtseva2023}, and semigroupoids \cite{haag2026a}. Further occurrences and applications of partial actions can be found in the survey by M. Dokuchaev \cite{dokuchaev2018}.

    Examples of partial actions of groups on sets arise naturally by restricting global actions. More precisely, an action of a group $G$ on a set $Y$ is a group homomorphism $\beta$ from $G$ to the group of bijections on $Y$. Given a subset $X$ of $Y$, each bijection $\beta_g$ can be restricted to a bijection $\alpha(g) \colon {_gX} \to X_g$, where
        $$ {_gX} = \{ x \in X \colon \beta_g(x) \in X \} = \beta_g^{-1}(X) \cap X \quad\text{and}\quad X_g = \beta_g({_gX}). $$
    The family $\{ \alpha(g) \}_{g \in G}$ obtained in this way defines a partial action of the group $G$ on the set $X$. The globalization problem asks whether every partial action is, up to isomorphism, obtained by restricting a global action in this manner. In this case, the partial action is said to be \textit{globalizable}. The answer to the globalization problem is affirmative for all of the previously mentioned classes of partial actions on sets, extending the corresponding result for partial actions of groups established independently by F. Abadie \cite{abadie2003} and by J. Kellendonk and M. Lawson \cite{lawson2004actions}. On the other hand, when the set on which the action is defined carries additional algebraic or topological structure, the existence of globalizations is no longer guaranteed, such as for partial actions of monoids on topological spaces \cite[Example 3.2]{schroder2004}, partial actions of groups on unital rings \cite[Theorem 4.5]{dokuchaev2004}, and partial actions of groups on semigroups \cite[Example 4.28]{khrypchenko2018}.

    Notably, proving that restrictions of global actions give rise to partial actions has become increasingly difficult as the algebraic structure of the acting object becomes more involved, particularly when its multiplication is only partially defined. For instance, see \cite[Lemma 2.6]{kudryavtseva2023} for semigroup actions and \cite[Proposition 2.14]{haag2026a} for semigroupoid actions. The converse problem is even more subtle: given a partial action, one would like to determine whether it can be realized as the restriction of a global one. This raises the question of whether these problems can be studied within a common framework, rather than treated separately for each class of algebraic structures. This question naturally suggests a categorical approach to the theory of partial actions, which may provide a unified framework for understanding which partial actions are globalizable.


    At this point, it is useful to clarify that the term \emph{restriction} will appear throughout this paper in two distinct senses. On the one hand, it refers to the process of restricting a global action to a suitable subset in order to obtain a partial action. On the other hand, it refers to the algebraic notion of a \emph{restriction structure}, such as a restriction semigroup or a restriction category.

    Restriction structures, in this latter sense, have appeared in the literature under several names and in a variety of contexts. We refer the reader to C. Hollings' survey \cite{hollings2009PP} for the historical development of restriction semigroups and to the work of J. Cockett and S. Lack \cite{cockett2002restriction} for examples of restriction categories and their applications to theoretical computer science.

    In particular, restriction semigroups appeared under the name of weakly left $E$-ample semigroups in \cite{gould2009partial}, where they provided a common framework for partial actions of monoids and inverse semigroups on sets. This unifying role of restriction structures naturally raises the question of whether a similar setting can encompass partial actions of categories, inverse semigroupoids, and restriction semigroups.

    A natural candidate for such a framework is the notion of restriction semigroupoid introduced in \cite{haag2026c}. There is, however, an obstruction. The treatment of partial actions of semigroups relies on the fact that every semigroup $S$ embeds into the monoid $S^1$, and hence into a restriction semigroup. In contrast, not every semigroupoid embeds into a restriction semigroupoid. This indicates that a more flexible setting is required, leading naturally to the notions of \emph{semiconstellation} and \emph{restriction semiconstellation}.

    Accordingly, in order to unify the various generalizations of partial actions on sets, we introduce semiconstellations, an asymmetric and non-associative generalization of semigroupoids, together with the corresponding notion of restriction semiconstellation. Partial and global actions of restriction semiconstellations on sets naturally extend the existing notions for categories, inverse semigroupoids, and restriction semigroups. A key feature of this framework is that every semiconstellation embeds into a restriction semiconstellation. However, unlike the semigroup case, there is no canonical choice of restriction semiconstellation containing a given semiconstellation. Consequently, we obtain a unique definition for partial action of semiconstellation on sets, generalizing the ones for semigroups and semigroupoids, but the corresponding notion of global action depends on the choice of the ambient restriction semiconstellation.

    More precisely, given a semiconstellation $S$ and a suitable equivalence relation $R$ on $S$, we construct a restriction semiconstellation $S^R$ containing $S$ such that global actions of $S^R$ correspond precisely to partial actions of $S$ whose domains are constant on the equivalence classes of $R$. We call such partial actions of $S$ as \emph{$R$-compatible}. This construction recovers the existing notions of global action for both semigroupoids and semigroups. Indeed, global semigroupoid actions as defined in \cite{haag2026a} are precisely the $R'$-compatible actions for a specific equivalence relation $R'$, while global semigroup actions as defined in \cite{kudryavtseva2023} correspond to non-degenerate $(S \times S)$-compatible actions. Moreover, when $S$ is a semigroup, $R'=S \times S$ and $S^{S\times S}=S^1$. Thus, the theories of partial and global actions of semigroups and semigroupoids are recovered as particular cases of the more general theory developed here.

    Our approach also has a natural connection with the theory of constellations. Constellations were introduced by V. Gould and C. Hollings \cite{gould2009restriction} in order to establish an isomorphism  between the category of restriction semigroups and the category of inductive constellations, and their partial actions on sets were subsequently studied by the same authors in \cite{gould2011actions}. More recently, inductive constellations were generalized to locally inductive constellations in \cite{haag2026b}. The theory developed here meets this earlier line of work in a particular class of examples. Every restriction semiconstellation that is also a constellation provides a particular example of a locally inductive constellation. For such examples, the notion of partial action introduced here coincides with the notion of partial action considered by Gould and Hollings in \cite{gould2011actions}. 

    We now turn to the globalization problem for partial actions on sets. Our approach is based on the notion of a reflector. Recall that a \textit{reflector} for $X$ is a pair $(\varphi,Y)$ where $\varphi$ is a morphism from an object $X$ in a category $\mathcal{C}$ to the object $Y$ in a subcategory $\mathcal{D}$ of $\mathcal{C}$, with the property that any morphism from $X$ to an object in $\mathcal{D}$ factors uniquely through $\varphi$. If every object in $\mathcal{C}$ has a reflector in $\mathcal{D}$, then the $\mathcal{D}$ is called a \textit{reflective subcategory} of $\mathcal{C}$.
    
    It is known that if every partial action has a globalization, then the category of global actions is a reflective subcategory of the category of partial actions, and in this case the globalization of a partial action is its reflector. In this paper, however, reflectors play a much more important role. We use the fact that when $\mathcal{C}$ is a category of algebraic structures and structure-preserving maps, reflectors often carry important algebraic information.
    
    Specifically, a reflector of a partial action in the subcategory of global actions determines whether the partial action can be obtained by restricting a global one, thereby providing a categorical solution to the globalization problem. Moreover, whenever such a globalization exists, the reflector explicitly constructs a global action together with an isomorphism between the original partial action and the corresponding restriction of the global action.

    There are connections between this point of view and previous work of M. Khrypchenko and F. Klock. In the context of partial actions of groups on universal algebras \cite{khrypchenko2018}, and of monoids on semigroups and objects of categories with pullbacks \cite{khrypchenko2024,khrypchenko2025}, reflectors also arise in relation to the existence of globalizations.

    With this approach in mind, our first objective is to construct a reflector from the category of partial actions of a restriction semiconstellation into its subcategory of global actions. Unexpectedly, the universal construction arising from this problem does not naturally belong to the category of global actions. Instead, it belongs to a slightly larger category contained in the category of unary partial $S$-algebras, that is, sets endowed with a family of partial maps indexed by $S$. Nevertheless, this construction retains significant information about the globalization problem. In particular, we prove that every globalizable partial action is isomorphic to the restriction of its reflector. 

    We then study the reflectors obtained above in two special cases. First, we establish a sufficient condition under which the reflector is a global action and provides a globalization of the original partial action. This condition is automatically satisfied for restriction semigroupoids, yielding a positive answer to the globalization problem for partial actions of restriction semigroupoids on sets. Second, we consider the case in which the restriction structure of $S$ is particularly simple. In this case, the reflectors are again global actions, but we are able to construct partial actions that are not globalizable. Thus, unlike the case of restriction semigroupoids, not every partial action of a restriction semiconstellation on a set is globalizable.

    Finally, we return to partial actions of semiconstellations, without assuming any additional restriction structure on $S$, and investigate when such an action is isomorphic to the restriction of an $R$-compatible action. We approach this question categorically, by relating the categories of partial and $R$-compatible actions of $S$ to the categories of partial and global actions of the restriction semiconstellation $S^R$. Using this correspondence, we show that a partial action of $S$ is isomorphic to the restriction of an $R$-compatible action if and only if the corresponding partial action of $S^R$ is isomorphic to the restriction of a global action. Consequently, the corresponding globalization problem has a positive answer for partial actions of semigroupoids on sets, but not, in general, for partial actions of semiconstellations on sets.

    This paper is organized as follows. In Section \ref{sec:2}, we recall some results from category theory, universal algebra, and the theory of equivalence relations that will be used throughout the paper. In Section \ref{sec:3}, we introduce semiconstellations, restriction semiconstellations, and actions of restriction semiconstellations on sets. In Section \ref{sec:4}, we study the globalization problem for partial actions of restriction semiconstellations on sets. We construct reflectors for partial actions and investigate two special cases of this construction, obtaining a positive answer to the globalization problem for restriction semigroupoids and a negative answer for restriction semiconstellations in general. In Section \ref{sec:5}, we construct restriction semiconstellations from certain equivalence relations on semiconstellations and apply this construction to obtain a representation theorem for semiconstellations. In Section \ref{sec:6}, we return to semiconstellations without assuming any additional restriction structure. We introduce partial and $R$-compatible actions of semiconstellations on sets and show that the corresponding globalization problem can be reduced to the globalization problem for actions of restriction semiconstellations. In Section \ref{sec:7}, we give special attention to partial actions of semigroups on sets, since the globalization obtained by our construction differs slightly from that of \cite{kudryavtseva2023}. Finally, in Section \ref{sec:8}, we discuss a selection of open questions arising from the methods developed in this paper.
    
    \section{Preliminaries and conventions} \label{sec:2}
    
    In this section, we recall some definitions and results about category theory, universal algebras, and equivalence relations necessary in this paper. Our references are \cite{maclane} for category theory, \cite{gratzer} for universal algebras, and \cite{plemmons} for relations.
    
    \subsection{Categories} We begin by fixing notation and terminology concerning categories, functors, and isomorphisms. We then recall the notion of a free object and some of its basic properties.  \label{sec:categories}
    
    \begin{defi}
        A \textit{directed graph} consists of a class of \textit{objects} $C_0$, a class of \textit{arrows} $C_1$ and two operations $\mathbf{d},\mathbf{r} \colon C_1 \to C_0$, called \textit{domain} and \textit{range}. For an arrow $f \in C_1$, we denote $f \colon X \to Y$ to indicate that $\mathbf{d}(f) = X$ and $\mathbf{r}(f) = Y$.
    \end{defi}
    
    \begin{defi}
        A \textit{category} is a directed graph with two additional operations: the \textit{identity} operation $1 \colon C_0 \to C_1$, which associates each object $X \in C_0$ with an arrow $1_X \colon X \to X$; and a \textit{partial composition} $\circ$, which associates to each pair of arrows $(g,f)$ satisfying $\mathbf{d}(g) = \mathbf{r}(f)$ an arrow $g \circ f \colon \mathbf{d}(f) \to \mathbf{r}(g)$. Furthermore, these operations are subject to the following axioms:
        \begin{enumerate}
            \item \textit{Associativity.} Whenever $f \colon X \to Y$, $g \colon Y \to Z$, and $h \colon Z \to W$, one has
                $$ (h \circ g) \circ f = h \circ (g \circ f). $$
            \item \textit{Unit law.} For every arrow $f \colon X \to Y$, one has $id_Y \circ f = f$ and $f \circ id_X = f$.
        \end{enumerate}
    \end{defi}
    
    In any category, there is a bijective correspondence between objects and identity arrows. More precisely, if an arrow $f$ satisfies $f \circ g = g$ and $h \circ f = h$, whenever these compositions are defined, then there exists a unique object $X$ such that $f = 1_X$. Thus, objects may be identified with their identity arrows, and we shall regard $C_0=\{1_X : X\in C_0\}\subseteq C_1$. Throughout this text, every category $\mathcal{C}$ will be represented in the form $\mathcal{C} = (C;\mathbf{d},\mathbf{r},\circ)$, where $C=C_1$ and $C_0\subseteq C$ is identified with the set of identity arrows. 

    In the remaining of this subsection, $\mathcal{B}$ and $\mathcal{C}$ are categories.
    
    \begin{defi}
        An arrow $f \colon X \to Y$ is called an \textit{isomorphism} if there is an arrow $g \colon Y \to X$ that satisfies $g \circ f = id_X$ and $f \circ g = id_Y$. In this case, we denote $g = f^{-1}$. We say that two objects $X,Y \in C$ are \textit{isomorphic} and denote $X \simeq Y$ if there is an isomorphism $f \colon X \to Y$.
    \end{defi}
    
    The categories relevant to this paper will be introduced in later sections. Many of them will arise as subcategories of previously defined categories. We recall the corresponding definition.
    
    \begin{defi}
        A subset $S\subseteq C$ is called a \textit{subcategory} of $\mathcal{C}$ if the following conditions hold:
        \begin{enumerate}
        \item for every $f\in S$, one has $\mathbf{d}(f),\mathbf{r}(f)\in S$;
        \item whenever $f,g\in S$ and $g\circ f$ is defined in $\mathcal{C}$, one has
        $g\circ f\in S$.
        \end{enumerate}
        A subcategory $S$ is said to be \textit{full} if, whenever $X,Y\in S$ are objects and $f\colon X\to Y$ is an arrow of $\mathcal{C}$, then $f\in S$.
    \end{defi}
    
    In other words, a full subcategory of $\mathcal{C}$ is a subcategory $S$ containing every arrow of $\mathcal{C}$ whose domain and range belong to $S$. By definition, if $S$ is a subcategory of $\mathcal{C}$, then $\mathcal{S} = (S;\mathbf{d},\mathbf{r},\circ)$ is itself a category.
    
    \begin{defi}
        A \textit{covariant functor} from $\mathcal{C}$ to $\mathcal{B}$, denoted $F \colon \mathcal{C} \to \mathcal{B}$, is a map that associates to each arrow $f \in C$ an arrow $Ff \in B$ such that:
        \begin{enumerate}
            \item $F\mathbf{d}(f) = \mathbf{d}(Ff)$ and $F\mathbf{r}(f) = \mathbf{r}(Ff)$, for all $f \in C$.
            \item $F(g \circ f) = Fg \circ Ff$, whenever $\mathbf{d}(g) = \mathbf{r}(f)$.
        \end{enumerate}
        A covariant functor is called \textit{fully faithful} if, for every pair of objects $X,Y \in \mathcal{C}$, the map $F$ induces a bijection between the arrows $X \to Y$ and the arrows $FX \to FY$.
    \end{defi}
    
    In the literature, there is also the notion of \textit{contravariant functor}, but such functors will not be used in this paper. Therefore, whenever we say that ``$F$ is a functor'', we mean that $F$ is a covariant functor. It is straightforward to verify that the composition of functors $F \colon \mathcal{C} \to \mathcal{B}$ and $G \colon \mathcal{B} \to \mathcal{A}$, given by $(G \circ F)(f) = GFf$, for each arrow $f \in C$, defines a functor $\mathcal{C} \to \mathcal{A}$. Furthermore, for every subcategory $\mathcal{S}$ of $\mathcal{C}$, the inclusion map $I \colon \mathcal{S} \to \mathcal{C}$ is a functor, called the \emph{inclusion functor}. In particular, if $\mathcal{S}$ is a full subcategory of $\mathcal{C}$, then the inclusion functor is fully faithful.
    
    \begin{defi}
        Let $U \colon \mathcal{B} \to \mathcal{C}$ be a functor. Given objects $Y \in B$, $X \in C$ and an arrow $\iota \colon X \to UY$, we say that the pair $(\iota,Y)$ is a \textit{free object} over $X$ if, for every object $Y' \in B$ and every arrow $\varphi \colon X \to UY'$, there is a unique arrow $\Phi \colon Y \to Y'$ such that $U\Phi \circ \iota = \varphi$. This property is equivalent to the commutativity of the following diagram:
        \begin{center}
            \begin{tikzpicture}
                \tikzstyle{every path} = [draw,->];
    
                \node (X) at (0,0) {$X$};
                \node (FY) at (2,0) {$UY$};
                \node (FY') at (2,-2) {$UY'$};
                \node (Y) at (4,0) {$Y$};
                \node (Y') at (4,-2) {$Y'$};
    
                \path (X) to node[above]{$\iota$} (FY);
                \path (X) to node[below left]{$\varphi$} (FY');
                \path[dashed] (Y) to node[right]{$\Phi$} (Y');
                \path (FY) to node[right]{$U\Phi$} (FY');
            \end{tikzpicture}
        \end{center}
        In the special case in which $\mathcal{B}$ is a full subcategory of $\mathcal{C}$ and $U \colon \mathcal{B} \to \mathcal{C}$ is the inclusion functor, the pair $(\iota,Y)$ is called a \textit{reflector} for $X$ in $\mathcal{B}$.
    \end{defi}
    
    The following results are straightforward from the uniqueness of the arrow $\Phi$.
    
    \begin{lemma}
        Let $U \colon \mathcal{B} \to \mathcal{C}$ be a functor and $(\iota,Y)$ be free over $X$. Then $(\iota',Y')$ is free over $X$ if and only if there is an isomorphism $\varphi \colon Y \to Y'$ such that $U\varphi \circ \iota = \iota'$.
    \end{lemma}
    
    The existence of free objects induces a ``free construction'' functor.
    
    \begin{prop} \label{prop:livre}
        Let $U \colon \mathcal{B} \to \mathcal{C}$ be a functor, and let $\mathcal{C}' \subseteq \mathcal{C}$ be the full subcategory consisting of
        the objects $X$ that have a free object $(\iota_X,FX)$. For each arrow $f \colon X \to Y$, where $X,Y \in \mathcal{C}'$, let $Ff \colon FX \to FY$ be the unique arrow in $B$ satisfying $UFf \circ \iota_X = \iota_Y \circ f$. Then:
        \begin{enumerate}
            \item $f \mapsto Ff$ defines a functor $F \colon \mathcal{C}' \to \mathcal{B}$.
            \item If $U$ is fully faithful, then $(id_{UY},Y)$ is free over $UY$ and $Y \simeq FUY$, for every $Y \in \mathcal{B}$.
        \end{enumerate}
    \end{prop}
    
    \subsection{Universal algebras} We recall the definition of partial algebras, subalgebras and relative subalgebras, morphisms between partial algebras, and the category of partial algebras. Lastly, we characterize isomorphisms in the category of partial algebras. \label{sec:algebras}\\
    
    In the following, $G^n$ denotes the cartesian product of $n$ copies of a set $G$, and $I$ denotes a set of indexes. By a \textit{partial $n$-ary operation} on a set $G$ we mean a function $f \colon S \subseteq G^n \to G$. In this case, we denote $f(a_1,\dots,a_n) \neq \emptyset$ to indicate that $(a_1,\dots,a_n) \in S$. Furthermore, we fix a family $\Omega = (n_\gamma)_{\gamma \in I}$ of positive integers, called a \textit{type}.
    
    \begin{defi}
        A \textit{partial $\Omega$-algebra} is a pair $\mathcal{G} = (G;\{g_{\gamma}\}_{\gamma  \in I})$ where $G$ is a set, and each $g_\gamma$ is a partial $n_{\gamma}$-ary operation on $G$. In this case, we say that $\mathcal{G}$ is of type $\Omega$.
    \end{defi}

    For simplicity, we shall simply say that ``$\mathcal{G}$ is an $\Omega$-algebra'', with the understanding that its operations may be partially defined. To avoid unnecessary repetitions, all results in this subsection will be stated for algebras of type $\Omega$. However, our main interest lies in algebras of types $(2)$, $(2,1)$, and $(1)_{\gamma\in I}$. In later sections, we introduce more convenient notation for these particular types of algebras. Throughout the remainder of this subsection, $\mathcal{G} = (G;\{g_\gamma\}_{\gamma \in I})$ and $\mathcal{H} = (H;\{h_\gamma\}_{\gamma \in I})$ denote $\Omega$-algebras.
    
    \begin{defi}
        A \textit{$\Omega$-subalgebra} of $\mathcal{G}$ is a subset $S \subseteq G$ such that, whenever $(a_1,\dots,a_{n_\gamma}) \in S^{n_\gamma}$ and $g_\gamma(a_1,\dots,a_{n_\gamma}) \neq \emptyset$, one has $g_\gamma(a_1,\dots,a_{n_\gamma}) \in S$. In this case, we denote $S \leq \mathcal{G}$.
    \end{defi}
    
    There is another notion of subalgebra that will play an important role in this work.
    
    \begin{defi}
        Let $S \subseteq G$. The \textit{relative $\Omega$-subalgebra} $\mathcal{G}|_S$ of $\mathcal{G}$ is the $\Omega$-algebra $(S;(s_\gamma)_{\gamma \in I})$ where, for each $(a_1,\dots,a_{n_\gamma}) \in S^{n_\gamma}$,
            $$ s_\gamma(a_1,\dots,a_{n_\gamma}) \neq \emptyset \iff [g_\gamma(a_1,\dots,a_{n_\gamma}) \neq \emptyset \quad\text{and}\quad g_\gamma(a_1,\dots,a_{n_\gamma}) \in S], $$
        and in this case $s_\gamma(a_1,\dots,a_{n_\gamma}) = g_\gamma(a_1,\dots,a_{n_\gamma})$.
    \end{defi}

    Note that every $\Omega$-subalgebra and every relative $\Omega$-subalgebra of $\mathcal{R}$ is itself an $\Omega$-algebra. Moreover, a relative $\Omega$-subalgebra $\mathcal{G}|_S$ is an $\Omega$-subalgebra if and only if, whenever $g_\gamma(a_1,\dots,a_{n_\gamma})$ is defined for elements $a_1,\dots,a_{n_\gamma}\in S$, one has $g_\gamma(a_1,\dots,a_{n_\gamma})\in S$.
    
    \begin{defi}
        A \textit{$\Omega$-morphism} from $\mathcal{G}$ to $\mathcal{H}$ is a function $\varphi \colon G \to H$ such that
            $$ \forall \gamma \in I, [g_\gamma(a_1,\dots,a_{n_\gamma}) \neq \emptyset \implies h_\gamma(\varphi(a_1),\dots,\varphi(a_{n_\gamma})) \neq \emptyset], $$
        and in this case $\varphi(g_\gamma(a_1,\dots,a_{n_\gamma})) = h_\gamma(\varphi(a_1),\dots,\varphi(a_{n_\gamma}))$.
    \end{defi}
    
    It is easy to see that if $\mathcal{G}$, $\mathcal{H}$ and $\mathcal{K}$ are $\Omega$-algebra and $\varphi \colon \mathcal{G} \to \mathcal{H}$ and $\psi \colon \mathcal{H} \to \mathcal{K}$ are $\Omega$-morphism, then the composition $\psi \circ \varphi \colon \mathcal{G} \to \mathcal{K}$ and the identity function $id_G \colon \mathcal{G} \to \mathcal{G}$ are also $\Omega$-morphisms.
    
    \begin{lemma}
        The quintuple $Alg(\Omega) = (Alg(\Omega)_0,Alg(\Omega)_1,\mathbf{d},\mathbf{r},\circ)$ is a category, where $Alg(\Omega)_0$ is the class of all $\Omega$-algebras, $Alg(\Omega)_1$ is the class of all $\Omega$-morphism and $\mathbf{d}$, $\mathbf{r}$ and $\circ$ are the usual domain, range, and composition of functions. In this case, the identity operation is given by $1_\mathcal{G} = id_G$, for all $\mathcal{G} \in Alg(\Omega)_0$.
    \end{lemma}
    
    Since $\Omega$-algebras and $\Omega$-morphisms form a category, the notion of isomorphism of $\Omega$-algebras is understood in the categorical sense. That is, $\mathcal{G} \simeq \mathcal{H}$ if and only if there exist $\Omega$-morphisms $\varphi \colon \mathcal{G} \to \mathcal{H}$ and $\varphi^{-1} \colon \mathcal{H} \to \mathcal{G}$ such that $\varphi^{-1} \circ \varphi = 1_{\mathcal{G}}$ and $\varphi \circ \varphi^{-1} = 1_{\mathcal{H}}$. We are particularly interested in isomorphisms between $\mathcal{G}$ and $\Omega$-subalgebras or relative $\Omega$-subalgebras of $\mathcal{H}$.
    
    \begin{defi}
        Let $\varphi \colon \mathcal{G} \to \mathcal{H}$ be a $\Omega$-morphism. We say that $\varphi$ is \textit{full} if for every $a_1,\dots,a_{n_\gamma}\in G$ satisfying
        \begin{align*}
            h_\gamma(\varphi(a_1),\dots,\varphi(a_{n_\gamma})) \neq \emptyset \quad\text{and}\quad h_\gamma(\varphi(a_1),\dots,\varphi(a_{n_\gamma})) \in \varphi(G), \end{align*} there exist $b_1,\dots,b_{n_\gamma}\in G$ such that
           \begin{align*} g_\gamma(b_1,\dots,b_{n_\gamma}) \neq \emptyset \quad\text{and}\quad \varphi(b_i) = \varphi(a_i).
        \end{align*}
        We say that $\varphi$ is \textit{strong} if for every $a_1,\dots,a_{n_\gamma}\in G$, $h_\gamma(\varphi(a_1),\dots,\varphi(a_{n_\gamma})) \neq \emptyset$ implies $g_\gamma(a_1,\dots,a_{n_\gamma}) \neq \emptyset$.
    \end{defi}

    Note that if $\mathcal{H}' \leq \mathcal{H}$, then the inclusion function $\iota \colon \mathcal{H}' \to \mathcal{H}$ is an injective strong $\Omega$-morphism. Similarly, for every subset $S$ of $H$, the inclusion function $\iota \colon \mathcal{H}|_S \to \mathcal{H}$ is an injective full $\Omega$-morphism. For any function $\varphi \colon G \to H$ such that $\varphi(G) \subseteq H' \subseteq H$, we have $\varphi = \iota \circ \varphi'$, where $\iota \colon H' \to H$ is the inclusion function and $\varphi' \colon G \to H'$ is defined to be equal to $\varphi$ point-wise. The function $\varphi'$ is the \textit{corestriction} of $\varphi$ to $H'$. The following results are straightforward.
    
    \begin{prop} \label{prop:universal}
        Let $\varphi \colon \mathcal{G} \to \mathcal{H}$ be a $\Omega$-morphism. Then:
        \begin{enumerate}
            \item $\mathcal{G} \simeq \mathcal{H}$ if and only if $\varphi$ is bijective strong if and only if $\varphi$ is bijective full.
            
            \item $\varphi(G) \leq \mathcal{H}$ and $\mathcal{G} \simeq \varphi(G)$ if and only if $\varphi$ is injective strong.
            
            \item $\mathcal{G} \simeq \mathcal{H}|_{\varphi(G)}$ if and only if $\varphi$ is injective full.
        \end{enumerate}
        The isomorphisms are understood to be given by the corestriction of $\varphi$ to $\varphi(G)$.
    \end{prop}
    
    The existence of a reflector for $\mathcal{G}$ in a subcategory $\mathcal{C}$ of $Alg(\Omega)$ determines when $\mathcal{G}$ can be isomorphic to a (relative) $\Omega$-subalgebra of some $\Omega$-algebra in $\mathcal{C}$.
    
    \begin{prop} \label{prop:universal-2}
        Let $\mathcal{C}$ be a subcategory of $Alg(\Omega)$, $U \colon \mathcal{C} \to Alg(\Omega)$ be the inclusion functor, and $\varphi \colon \mathcal{G} \to U\mathcal{H}$ be a $\Omega$-morphism. Suppose that there exists a reflector $(\iota,F\mathcal{G})$ for $\mathcal{G}$ in $\mathcal{C}$.
        \begin{enumerate}
            \item If $\varphi$ is injective, then $\iota$ is injective.
            \item If $\varphi$ is full, then $\iota$ is full.
            \item If $\varphi$ is strong, then $\iota$ is strong.
        \end{enumerate}
        Therefore, if $\mathcal{G}$ is isomorphic to a (relative) $\Omega$-subalgebra of $U\mathcal{H}$, then it must also be isomorphic to a (relative) $\Omega$-subalgebra of $UF\mathcal{G}$.
    \end{prop}
    
    \subsection{Equivalence relations} We recall the operations in the family of relations of a set and characterize the smallest equivalence relation containing a given relation and is invariant under a given family of partial functions. We also establish a property concerning inverse images under partial functions. \label{sec:relations} \\
    
    Throughout this subsection, $X$ denotes a set. A \textit{relation} on $X$ is a subset $R$ of $X \times X$. Whenever convenient, we write $x R y$ instead of $(x,y) \in R$. We denote by $\mathcal{R}(X)$ the set of relations on $X$. This set can be endowed with a binary operation $\circ$, called \textit{composition}, defined by
        $$ S \circ R = \{ (x,z) \colon \exists y \in X \text{ such that } (x,y) \in S, (y,z) \in R \}, \quad \forall S,R \in \mathcal{R}(X), $$
    and with a unary operation $^{-1}$, called \textit{inverse}, defined by
        $$ R^{-1} = \{ (y,x) \colon (x,y) \in R \}, \quad \forall R \in \mathcal{R}(X). $$
        
    The composition is associative, in the sense that $(S \circ R) \circ T = S \circ (R \circ T)$, for all $S,R,T \in \mathcal{R}(X)$. We note that $\Delta_X = \{ (x,x) \colon x \in X \} \in \mathcal{R}(X)$ is the identity of $(\mathcal{R}(X);\circ)$, that is, it satisfies $\Delta_X \circ R = R = R \circ \Delta_X$, for all $R \in \mathcal{R}(X)$. Therefore, for every relation $R \in \mathcal{R}(X)$ and every integer $n \geq 0$, we may define the powers of $R$ recursively by $R^0 = \Delta_X$ and $R^{n+1} = R^n \circ R$. It is straightforward to verify that
        $$ xR^ny \iff \exists x_0,\dots,x_n \in X \colon x = x_0,\ x_n = y \text{ and } x_iRx_{i+1}, \ \forall i=0,\dots,n-1. $$
    Since relations are subsets of $X\times X$, the union $\bigcup_{i \in I} R_i$ of a family of relations is again a relation. The following result characterizes the equivalence relation on $X$ generated by a given relation $R$.
    
    \begin{lemma} \label{lema:equivalence}
        Let $R$ be a relation on $X$. Then $\mathcal{E}(R) := \bigcup_{n \geq 0} (R \cup R^{-1})^n$ is an equivalence relation. Furthermore, $\mathcal{E}(R)$ is the smallest equivalence relation on $X$ that contains $R$.
    \end{lemma}
    
    Since a relation $R$ is symmetric if and only if $R^{-1} = R$, it follows from the previous result that the equivalence relation generated by a symmetric relation is given by $\mathcal{E}(R) = \bigcup_{n \geq 0} R^n$.
    
    \begin{defi}
        A relation $R$ on $X$ is said to be \textit{invariant} under a family $\mathcal{F}$ of partial functions on $X$ if
            $$ \forall f \in \mathcal{F}, [x,y \in dom(f) \text{ and } (x,y) \in R \implies (f(x),f(y)) \in R]. $$
    \end{defi}
    
    \begin{obs}
        Let $\mathcal{F}$ be a family of partial functions on $X$. Then $\mathcal{X} := (X;\mathcal{F})$ may be regarded as a universal algebra.  With this interpretation, an equivalence relation on $X$ is invariant under $\mathcal{F}$ if and only if it is a \textit{congruence} on $\mathcal{X}$. Since our treatment of equivalence relations is elementary, we shall use the terminology ``invariant over $\mathcal{F}$'' rather than ``congruence on $\mathcal{X}$''.
    \end{obs}
    
    We are interested in equivalence relations that are invariant under a family of partial functions and contain a given relation. Let $R$ be a relation on $X$, and let $\mathcal{F}$ be a family of partial functions on $X$. As a first step, we define the relation $R_{\mathcal{F}}$ on $X$ by
    \begin{align*}
        R_{\mathcal{F}} = \{(f(x),f(y)) \colon f \in \mathcal{F},\ x,y \in dom(f) \text{ and } (x,y)\in R\}.
    \end{align*}
    We will mostly use the following equivalent characterization: $(x,y) \in R_{\mathcal{F}}$ if and only if there are $f \in \mathcal{F}$ and $x',y' \in dom(f)$ such that $x = f(x')$, $y = f(y')$ and $(x',y') \in R$. Note that in general the relation $R_\mathcal{F}$ need not be reflexive, symmetric, or transitive, nor need it contain $R$. However, if $R$ is symmetric, then so is $R_{\mathcal{F}}$. We now define a sequence of equivalence relations by
        $$ R^{(0)} = \mathcal{E}(R) \quad\text{and}\quad R^{(n+1)} = \mathcal{E}(R^{(n)} \cup R^{(n)}_{\mathcal{F}}), \ \forall n \geq 0. $$
    It follows that each $R^{(n)}$ is an equivalence relation and $R \subseteq R^{(n)} \subseteq R^{(n+1)}$, for all $n \geq 0$. Therefore,
        $$ \mathcal{E}_\mathcal{F}(R) = \bigcup_{n \geq 0} R^{(n)} $$
    is an equivalence relation containing $R$.
    
    \begin{prop} \label{prop:congruence}
        Let $R$ be a relation, and $\mathcal{F}$ be a family of partial functions on $X$. Then $\mathcal{E}_\mathcal{F}(R)$ is the smallest equivalence relation on $X$ containing $R$ that is invariant under $\mathcal{F}$.
    \end{prop}
    
    Equivalence relations invariant under a family of partial functions are necessary to induce partial functions on the quotient set. Given an equivalence relation $R$ on $X$, we denote by $\pi \colon X \to X/R$ the canonical projection, given by $\pi(x) = [x] = \{y \in X \colon (x,y) \in R\}$.
    
    \begin{prop} \label{prop:congruence-2}
        Let $R$ be an equivalence relation invariant under $\mathcal{F}$. Then each $f \in \mathcal{F}$ induces a partial function $\overline{f}$ on $X/R$ satisfying $\pi \circ f = \overline{f} \circ \pi|_{dom(f)}$. More precisely, the partial function $\overline{f}$ is define by
            $$ dom(\overline{f}) = \{ [x] \colon \exists y \in [x] \cap dom(f) \} \quad\text{and}\quad \overline{f}([x]) = [f(y)]. $$
    \end{prop}
    
    Lastly, we establish a property concerning inverse images under partial functions. Let $PT(X)$ denote the set of partial functions on $X$. Recall that, for every subset $A \subseteq X$ and every partial function $f \in PT(X)$, the \textit{inverse image} of $A$ under $f$ is the subset $f^{-1}(A) = \{ x \in dom(f) \colon f(x) \in A \}$. To each partial function $f \in PT(X)$, we associate the relation $R_f = \{ (f(x),x) \colon x \in dom(f) \}$. This assignment defines an injective map $PT(X) \to \mathcal{R}(X)$ which allows us to regard partial functions as particular relations. Moreover, it induces a binary operation $\star$ on $PT(X)$, defined by
        $$ f \star g \colon g^{-1}(dom(f) \cap im(g)) \to X, \quad x \mapsto (f \star g)(x) = f(g(x)). $$
    Naturally, we have $R_{f \star g} = R_f \circ R_g$. Now let $\mathcal{P}(X)$ denote by the family of subsets of $X$. We define functions $\Delta \colon \mathcal{P}(X) \to \mathcal{R}(X)$ and $r \colon \mathcal{R}(X) \to \mathcal{P}(X)$, respectively, by
        $$ \Delta_A = \{ (x,x) \colon x \in A \}, \ \forall A \in \mathcal{P}(X) $$
    and
        $$ r(R) = \{ x \colon \exists y \in X \text{ such that } (x,y) \in R\}, \ \forall R \in \mathcal{R}(X). $$
    Furthermore, consider the unary operation $^+ \colon \mathcal{R}(X) \to \mathcal{R}(X)$, given by
        $$ R^+ = \{ (x,x) \colon \exists y \in X \text{ such that }(x,y) \in R \}, \ \forall R \in \mathcal{R}(X). $$
    
    \begin{lemma} \label{lema:relations}
        The following properties hold:
        \begin{enumerate}
            \item The operation $^+$ satisfies $(S \circ R)^+ = (S \circ R^+)^+$, for all $S,R \in \mathcal{R}(X)$.
            \item The operation $^{-1}$ satisfies $(S \circ R)^{-1} = R^{-1} \circ S^{-1}$, for all $S,R \in \mathcal{R}(X)$.
            \item For all $f \in PT(X)$ and $A \in \mathcal{P}(X)$, we have $f^{-1}(A) = r((R_f^{-1} \circ \Delta_A)^{+})$.
            \item $r(\Delta_A) = A$ and $\Delta_{r(R^+)} = R^+$, for all $A \in \mathcal{P}(X)$ and $R \in \mathcal{R}(X)$.
        \end{enumerate}
    \end{lemma}
    
    We now prove the following property of inverse images of partial functions.
    
    \begin{prop} \label{prop:relations}
        Let $f,g \in PT(X)$ and $A \in \mathcal{P}(X)$. Then $(f \star g)^{-1}(A) = g^{-1}(f^{-1}(A))$.
    
        \begin{proof}
            Since $R_{f \star g} = R_f \circ R_g$, we have
            \begin{align*}
                (R_{f \star g}^{-1} \circ \Delta_A)^+ &= ((R_f \circ R_g)^{-1} \circ \Delta_A)^+ \\
                &= (R_g^{-1} \circ R_f^{-1} \circ \Delta_A)^+ & \ref{lema:relations}(2) \\
                &= (R_g^{-1} \circ (R_f^{-1} \circ 1_A)^+)^+. & \ref{lema:relations}(1)
            \end{align*}
            Using (3) and (4) from Lemma \ref{lema:relations}, we obtain that $\Delta_{f^{-1}(A)} = (R_f^{-1} \circ \Delta_A)^+$. Hence,
            \begin{align*}
                \Delta_{(f \star g)^{-1}(A)} &= (R_{f \star g}^{-1} \circ \Delta_A)^+ \\
                &= (R_{g}^{-1} \circ (R_{f}^{-1} \circ \Delta_A)^+)^+ \\
                &= (R_{f}^{-1} \circ \Delta_{f^{-1}(A)})^+ \\
                &= \Delta_{g^{-1}(f^{-1}(A))}.
            \end{align*}
            From Lemma \ref{lema:relations}(4), we conclude that $(f \star g)^{-1}(A) = g^{-1}(f^{-1}(A))$.
        \end{proof}
    \end{prop}
    
    \section{Actions of restriction semiconstellations} \label{sec:3}
    
    In this section, we introduce semiconstellations, restriction semiconstellations and actions of restriction semiconstellations on sets. We provide examples and prove some properties of those structures that will be used in this paper. Actions of semiconstellations on sets will be introduced in a further section.
    
    \subsection{Semiconstellations} We recall the definitions of semigroupoids and constellations, and introduce the notion of a semiconstellation, which generalizes both structures. We also introduce the notion of disruptive elements and prove that a semiconstellation is a semigroupoid if and only if it has no disruptive elements.
    
    \begin{defi} \cite[Definition 2.1]{exel2011semigroupoid} \label{defi:semigroupoid}
        A \textit{semigroupoid} is a triple $(S,S^{(2)},\cdot)$ such that $S$ is a set, $S^{(2)}$ is a subset of $S \times S$ and $\cdot \colon S^{(2)} \to S$ is a partial operation which is associative in the following sense: if $s,t,r \in S$ are such that either
        \begin{enumerate} \Not{sg}
            \item $(s,t), (t,r) \in S^{(2)}$, or \label{sg1}
            \item $(s,t), (s \cdot t,r) \in S^{(2)}$, or \label{sg2}
            \item $(t,r), (s,t \cdot r) \in S^{(2)}$, \label{sg3}
        \end{enumerate}
        then all of $(s,t)$, $(t,r)$, $(s \cdot t,r)$ and $(s, t \cdot r)$ lie in $S^{(2)}$, and in this case $(s \cdot t) \cdot r = s \cdot (t \cdot r)$.
    \end{defi}
    
    The class of semigroupoids contains all semigroups and all small categories. Examples of semigroupoids that are neither semigroups or categories are the Markov semigroupoids, which can be found in \cite[Definition 6.1]{exel2011semigroupoid}. Semigroupoids in the sense of Definition \ref{defi:semigroupoid} are often referred to as \emph{Exel semigroupoids}. We point out, however, that an older notion of semigroupoid also exists, namely, a directed graph equipped with a partially defined binary operation. These structures were introduced by Tilson \cite{tilson1987categories} and are also known as \emph{Tilson semigroupoids}, \emph{semicategories}, or \emph{categorical semigroupoids}. It was shown in \cite[Theorem 2.15]{cordeiro2023etale} that every Tilson semigroupoid is an Exel semigroupoid, whereas the converse does not hold. Throughout this paper, the term \emph{semigroupoid} always refers to an Exel semigroupoid.
    
    \begin{defi} \cite[Definition 2.1]{gould2009restriction}
        A \textit{right constellation} is a quadruple $(S,S^{(2)},\cdot,\mathbf{d})$ such that $S$ is a set, $S^{(2)}$ is a subset of $S \times S$, $\cdot \colon S^{(2)} \to S$ and $\mathbf{d} \colon S \to S$ are such that:
        \begin{enumerate} \Not{c}
            \item If $(s,t), (s \cdot t, r) \in S^{(2)}$, then $(t,r), (s, t \cdot r) \in S^{(2)}$ and $(s \cdot t) \cdot r = s \cdot (t \cdot r)$.\label{c1}
            \item $(s,t), (s \cdot t,r) \in S^{(2)}$ if and only if $(s,t), (t,r) \in S^{(2)}$.\label{c2}
            \item $(s,\mathbf{d}(t)) \in S^{(2)}$ and $s \cdot \mathbf{d}(t) = s$ if and only if $\mathbf{d}(t) = \mathbf{d}(s)$. \label{const3}
            \item If $(\mathbf{d}(s),t) \in S^{(2)}$, then $\mathbf{d}(s) \cdot t = t$. \label{const4}
        \end{enumerate}
    \end{defi}
    
    The class of right constellations contains all small categories, but not all semigroupoids. Indeed, if $(S,S^{(2)},\cdot,\mathbf{d})$ is a right constellation, then every element of the form $\mathbf{d}(s)$ is idempotent, that is, $\mathbf{d}(s) \cdot \mathbf{d}(s) = \mathbf{d}(s)$. However, a semigroupoid need not contain any idempotent elements. Examples of right constellations that are not categories are given by the function constellations introduced in \cite[p. 269]{gould2009restriction}.

    From the viewpoint of universal algebra, right constellations form a class of $(2,1)$-algebras satisfying axioms \eqref{c1}-\eqref{const4}. Likewise, semigroupoids form a class of $(2)$-algebras satisfying axioms \eqref{sg1}-\eqref{sg3}. If we disregard the unary operation $\mathbf{d}$ and the axioms involving it from the definition of a right constellation, we are naturally led to the following class of $(2)$-algebras.
    
    \begin{defi} \label{defi:semiconst}
        A \textit{right semiconstellation} is a triple $(S,S^{(2)},\cdot)$ such that $S$ is a set, $S^{(2)}$ is a subset of $S \times S$ and $\cdot \colon S^{(2)} \to S$ is such that:
        \begin{enumerate} \Not{sc}
            \item $(s,t), (s \cdot t,r) \in S^{(2)}$ if and only if $(s,t), (t,r) \in S^{(2)}$. \label{sc1}
            \item If $(s,t), (t,r) \in S^{(2)}$, then $(s, t \cdot r) \in S^{(2)}$ and $(s \cdot t) \cdot r = s \cdot (t \cdot r)$. \label{sc2}
        \end{enumerate}
    \end{defi}
    
    The class of right semiconstellations contains both semigroupoids and right constellations.
    
    Before presenting examples of right semiconstellations that are neither semigroupoids nor right constellations, we fix some notation. Firstly, the definition of right semiconstellations is asymmetric. Accordingly, one may define the notion of a \emph{left semiconstellation} dually. In fact, $(S,S^{(2)},\cdot)$ is a right semiconstellation if and only if $(S,(S^{(2)})^{-1}, \cdot_{op})$ is a left semiconstellation, where $s \cdot_{op} t = t \cdot s$. Throughout this paper, the term \emph{semiconstellation} will always mean a right semiconstellation unless explicitly stated otherwise.

    Semiconstellations form a class of $(2)$-algebras satisfying axioms \eqref{sc1} and \eqref{sc2}. Recall that a $(2)$-algebra is a pair $\mathcal{S} = (S;\cdot)$, where $S$ is a set and $\cdot$ is a partially defined binary operation in $S$. We denote the domain of $\cdot$ by $S^{(2)} \subseteq S \times S$ and write $st$ for the value of $\cdot$ at $(s,t)\in S^{(2)}$. Whenever convenient, we write $\exists st$, or simply say that ``$st$ is defined'', instead of $(s,t)\in S^{(2)}$. Therefore, when no confusion is likely to arise, we shall simply write ``$S$ is a semiconstellation'', omitting explicit reference to the domain $S^{(2)}$ and the partial operation $\cdot$.
    
    Next, we introduce a family of elements that characterizes when a semiconstellation is a semigroupoid. Let $S$ be any (2)-algebra. For each $s \in S$, set
        $$ S^s = \{t \in S \colon \exists st\} \quad\text{and}\quad {^sS} = \{t \in S \colon \exists ts\}. $$ Thus, $S^s$ is the set of right composable elements with $s$, while ${}^{s}S$ is the set of left composable elements with $s$.
    
    \begin{prop} \label{prop:semiconst}
    A $(2)$-algebra $S$ is a semiconstellation if and only if, for all $s,t\in S$ such that $\exists\,st$, the following conditions hold:
    \begin{enumerate}
        \item $S^{st}=S^t$;
        \item ${}^{s}S\subseteq{}^{st}S$;
        \item $(rs)t=r(st)$, for every $r\in{}^{s}S$.
    \end{enumerate} In particular, $S$ is a semigroupoid if and only if it is a semiconstellation satisfying ${^sS} = {^{st}S}$ whenever $\exists\,st$.
    
        \begin{proof}
         Suppose that $\exists\,st$. Then axiom \eqref{sc1} states that $\exists\,tr$ if and only if $\exists\,(st)r$, for every $r\in S$, which is equivalent to $S^t=S^{st}$. Likewise, axiom \eqref{sc2} states that, whenever $\exists\,rs$, one has $\exists\,r(st)$ and $(rs)t=r(st)$. This is equivalent to ${}^{s}S\subseteq{}^{st}S$ together with $(rs)t=r(st)$ for all $r\in{}^{s}S$.
            
         Finally, a semiconstellation is a semigroupoid if and only if $\exists\,r(st)$ implies $\exists\,rs$. Since $r\in{}^{st}S$ if and only if $\exists\,r(st)$, this condition is equivalent to ${}^{st}S\subseteq{}^{s}S$. Thus, we obtain ${}^{s}S={}^{st}S$.
        \end{proof}
    \end{prop}
    
        Proposition \ref{prop:semiconst} often provides a more convenient criterion for verifying whether a given $(2)$-algebra is a semiconstellation than Definition \ref{defi:semiconst}. Moreover, since semigroupoids are precisely the semiconstellations satisfying ${^s}S={^{st}S}$ whenever $\exists\,st$, this characterization motivates the following definition.
    
    \begin{defi}
        Let $S$ be a semiconstellation. An element $d\in S$ is called \textit{disruptive} if there exist $t,r\in S$ such that $r(dt)$ is defined but $rd$ is not defined. Equivalently, $d$ is disruptive if there exists $t\in S^d$ for which ${^dS} \neq {^{dt}S}$. We denote by $\mathfrak{D}(S)$ the set of disruptive elements of $S$.
    \end{defi}
    
    \begin{corollary} \label{coro:semiconst}
        A semiconstellation $S$ is a semigroupoid if and only if $\mathfrak{D}(S) = \emptyset$.
    \end{corollary}

    To conclude this subsection, we present some examples of semiconstellations. Firstly, we list all semiconstellations structures on the set $S = \{s,t\}$, highlighting which structures are semigroups, semigroupoids, constellations or semiconstellations. Then, we construct a family of semiconstellations $S_{m}^{n}$, where $m,n \geq 1$, and prove that $S_{m}^{n}$ is a semigroupoid if and only if $m = 1$. Moreover, none of the semiconstellations $S_m^n$ admits the structure of a constellation.

    \begin{exe} \label{exe:smallscs}
        A set with two elements $\{s,t\}$ can be endowed with 13 non-isomorphic semiconstellation structures, being 5 semigroups, 5 semigroupoids, 1 constellation and 2 semiconstellations. The semigroups are:
        \begin{itemize}
            \item The cyclic group $C_2$, with composition given by $s \cdot s = s$, $t \cdot t = s$ and $s \cdot t = t \cdot s = t$.
            
            \item The $\{0,1\}$ monoid, with composition given by $s \cdot s = s$, $t \cdot t = t$ and $s \cdot t = t \cdot s = t$.
            
            \item The left zero semigroup $L_0$, with composition given by $s \cdot t = s \cdot s = s$ and $t \cdot s = t \cdot t = t$.

            \item The right zero semigroup $R_0$ that coincides with the opposite semigroup $L_0^{op}$.

            \item The semigroup with composition given by $s \cdot s = s \cdot t = t \cdot s = t \cdot t = t$.
        \end{itemize}
        The semigroupoids are:
        \begin{itemize}
            \item The set with two elements $I_2$ where $I_2^{(2)} = \emptyset$.
            
            \item The disjoint union $I_1 \sqcup C_1$ of a set with one element and the trivial group. In this case the only composition defined is $s \cdot s = s$.

            \item The disjoint union $C_1 \sqcup C_1$ of two copies of the trivial group. In this case the compositions are given by $s \cdot s = s$ and $t \cdot t = t$.

            \item The semigroupoid $L$, with composition given by $s \cdot s = s$ and $t \cdot s = t$.

            \item The semigroupoid $R$ that coincides with the opposite semigroupoid $L^{op}$.
        \end{itemize}
        The unique constellation is the set $S$ endowed with the composition given by $s \cdot s = s$, $t \cdot t = t$ and $s \cdot t = t$, and the unary operation $\mathbf{d} \colon S \to S$ given by the identity function. Lastly, the two semiconstellations are:
        \begin{itemize}
            \item The semiconstellation $S_1$ with composition $s \cdot t = t$.
            \item The semiconstellation $S_2$ with composition $s \cdot t = t$ and $t \cdot t = t$.
        \end{itemize}
        Note that the semiconstellations $S_i$, $i=1,2$, are not semigroupoids since $s \cdot (s \cdot t)$ is defined but $s \cdot s$ is not, and can not be endowed with a constellation structure since there is no element $\mathbf{d}(s) \in S_i$ such that $s \cdot \mathbf{d}(s)$ is defined and equals to $s$. Furthermore, the semiconstellations $S_i$ can be represented by certain kind of graphs. To fix notations, consider the following graph with vertex $\{1,2\}$ and edges $\{x,y,z\}$. 
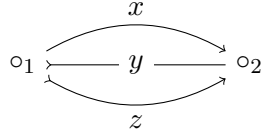
        \begin{center}
            \begin{tikzpicture}
                \tikzstyle{every path}=[draw];

                \node (e1) at (0,0) {$\circ_1$};
                \node (e2) at (3,0) {$\circ_2$};

                \path[bend left, ->] (e1) to node[above]{$x$} (e2);
                \path[>-] (e1) to node[fill=white]{$y$} (e2);
                \path[bend right, >->] (e1) to node[below]{$z$} (e2);
            \end{tikzpicture}
        \end{center}
        In this graph, we say that the edge $x$ is entering vertex $2$, the edge $y$ is leaving vertex $1$, and the edge $z$ both leaving vertex $1$ and entering vertex $2$. With this notation, the semiconstellations $S_1$ and $S_2$ can be represented by the following graphs,
        
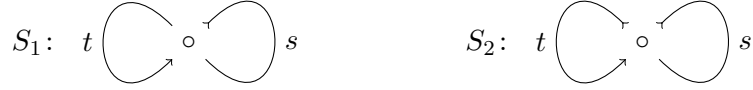
\begin{center}
            \begin{tikzpicture}
                \tikzstyle{every path}=[draw];

                \node (e1) at (0,0) {$\circ$};
                \node at (-2,0) {$S_1 \colon$};

                \path[<-] (e1) edge [out=225,in=135,looseness=10] node[left] {$t$} (e1);
                \path[>-] (e1) edge [out=45,in=315,looseness=10] node[right] {$s$} (e1);

                \node (e2) at (6,0) {$\circ$};
                \node at (4,0) {$S_2 \colon$};

                \path[<-<] (e2) edge [out=225,in=135,looseness=10] node[left] {$t$} (e2);
                \path[>-] (e2) edge [out=45,in=315,looseness=10] node[right] {$s$} (e2);
            \end{tikzpicture}
        \end{center}
        where, for $x,y \in S_i$, the composition $x \cdot y$ is defined if and only if the edges corresponding to $x$ and $y$ meet in a vertex where $y$ is entering and $x$ is leaving. 
    \end{exe}

    To illustrate the idea behind the family of semiconstellations $S^n_m$, we begin describing the simpler case of $S^1_2$. This case will also be used next section to construct examples of actions.

    \begin{exe} \label{exe:S21}
        Consider the sets $S = \{ s_1,s_2,s_3,d_1,d_2 \}$ and $S^{(2)} = \{ (d_1,s_1), (d_2,s_2) \}$. We define an operation $S^{(2)} \to S$ by
            $$ d_1 \cdot s_1 = s_2 \quad\text{and}\quad d_2 \cdot s_2 = s_3. $$
        Then $(S,S^{(2)},\cdot)$ is a semiconstellation. In fact, since the only defined compositions are $d_1 \cdot s_1$ and $d_2 \cdot s_2 = d_2 \cdot (d_1 \cdot s_1)$, conditions \eqref{sc1} and \eqref{sc2} are satisfied by vacuity.

        This semiconstellation is not a semigroupoid since $d_2 \cdot (d_1 \cdot s_1)$ is defined but $d_2 \cdot d_1$ is not, and does not admit a constellation structure since it has no idempotent elements. With the notation from Example \ref{exe:smallscs}, $S$ can be represented by the following graph:
        
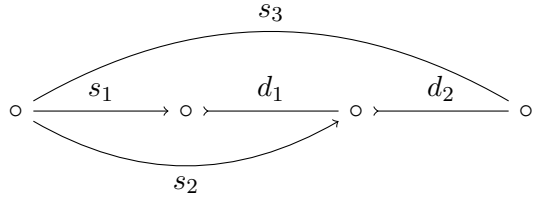
\begin{center}
            \begin{tikzpicture}[scale=0.75]
                \tikzstyle{every path}=[draw];
    
                \node (e0) at (0,0) {$\circ$};
                \node (e1) at (3,0) {$\circ$};
                \node (e2) at (6,0) {$\circ$};
                \node (e3) at (9,0) {$\circ$};
    
                \path[->] (e0) to node[above]{$s_1$} (e1);
                \path[>-] (e1) to node[above]{$d_1$} (e2);
                \path[>-] (e2) to node[above]{$d_2$} (e3);
                \path[->,bend right] (e0) to node[below]{$s_2$} (e2);
                \path[bend left] (e0) to node[above]{$s_3$} (e3);
            \end{tikzpicture}
        \end{center}
    \end{exe}

    Lastly, we describe the construction of the semiconstellation $S^n_m$, for arbitrary $m,n \geq 1$.
    
    \begin{exe} \label{exe:semiconst}
        Let $I_0 = \emptyset$ and, for each integer $n \geq 1$, let $I_n = \{1,\dots,n\}$. Given integers $m,n \geq 1$, consider the sets
            $$ S_{m}^{n} = \{ s_{i}^{j} \colon (i,j) \in I_{m+1} \times I_n \} \cup \{ d_i \colon i \in I_m \}, $$
        and
            $$ (S_{m}^{n})^{(2)} = \{ (d_i,s_{i}^{j}) \colon (i,j) \in I_m \times I_n  \}.$$
        We endow $S_m^n$ with the partially defined binary operation $\cdot \colon (S_{m}^{n})^{(2)} \to S_m^n$ given by $d_i \cdot s_{i}^{j} = s_{i+1}^{j}$. We claim that $S_{m}^{n}$ with this operation is a semiconstellation. In fact, note that
            $$ S^{d_i} = \{s_{i}^{j} \colon j \in I_n\} \quad\text{and}\quad {^{d_i}S} = \emptyset, \quad \forall i \in I_m, $$
            $$ S^{s_{i}^{j}} = \emptyset, \ \forall i \in I_{m+1}, \quad {^{s_{i}^{j}}S} = \{d_i\}, \ \forall i \in I_m, \quad\text{and}\quad {^{s_{m+1}^{j}}S} = \emptyset, \quad \forall j \in I_n. $$
        Therefore, whenever $d_i \cdot s_{i}^{j}$ is defined, we have
            $$ S^{d_i \cdot s_{i}^{j}} = S^{s_{i+1}^{j}} = \emptyset = S^{s_{i}^{j}} \quad\text{and}\quad  {^{d_i}S} = \emptyset \subseteq {^{d_i \cdot s_{i}^j}S}. $$
        Since ${^{d_i}S} = \emptyset$, it follows from Proposition \ref{prop:semiconst} that $S_{m}^{n}$ is a semiconstellation. Furthermore,
            $$ {^{d_i}S} = {^{d_i \cdot s_i^j}S} \iff \emptyset = {^{s_{i+1}^{j}}S} \iff i = m. $$
        Hence, the set of disruptive elements of $S_{m}^{n}$ is $\mathfrak{D}(S_{m}^{n}) = \{d_i \colon i \in I_{m-1}\}$. By Corollary \ref{coro:semiconst}, we conclude that $S_{m}^{n}$ is a semigroupoid if and only if $\mathfrak{D}(S_{m}^{n}) = \emptyset$ if and only if $m=1$. In general, $S_{m}^{n}$ is not a constellation since it has no idempotent elements. 
    \end{exe}
    
        The semiconstellations $S_m^n$ from Example \ref{exe:semiconst} are finite. By replacing the index set $I_n$ and/or $I_m$ with $I_\infty$, the set of positive integers, the same construction yields the infinite semiconstellations $S_m^\infty$, $S_\infty^n$, and $S_\infty^\infty$.

    \begin{obs}
        The notion of a left semiconstellation has recently appeared in \cite[Definition 6.1]{gould2025} under the name \emph{pre-constellation}. The notion is introduced there to characterize $\mathbf{d}$-inverse constellations, a particular class of left constellations, without explicitly referring to the unary operation. Nevertheless, all examples considered in \cite{gould2025} are, in fact, left constellations.
    \end{obs}
    
    \subsection{Restriction semiconstellations} We introduce right restriction semiconstellations, which generalize right restriction semigroupoids and, consequently, right restriction semigroups and right restriction categories.
    
    \begin{defi} \label{defi:restriction}
        A \textit{right restriction semiconstellation} $(S,\ast)$ is a semiconstellation $S$ endowed with a unary operation $\ast \colon S \to S$, $s\mapsto s^\ast$, satisfying the following axioms:
        \begin{enumerate} \Not{r}
            \item For all $s \in S$, $ss^\ast$ is defined and $ss^\ast = s$. \label{r1}
            \item If $s^\ast t^\ast$ is defined, then $t^\ast s^\ast$ is defined and $s^\ast t^\ast = t^\ast s^\ast$. \label{r2}
            \item If $st^\ast$ is defined, then $(st^\ast)^\ast = s^\ast t^\ast$. \label{r3}
            \item If $s^\ast t$ is defined, then $s^\ast t = t(st)^\ast$. \label{r4}
        \end{enumerate}
        We denote by $S^\ast = \{ s^\ast \colon s \in S \}$ the \emph{set of projections} of $(S,\ast)$.
    \end{defi}
    
    \begin{obs} \label{obs:restriction}
        The axioms for right restriction semiconstellations are well defined. Indeed, by Proposition \ref{prop:semiconst} and \eqref{r1}, we obtain that $S^s = S^{ss^\ast} = S^{s^\ast}$, for every $s \in S$. Hence, $st$ is defined if and only if $st^\ast$ is defined. In particular, $st^\ast$ is defined if and only if $s^\ast t^\ast$ is defined in \eqref{r3}, so axiom \eqref{r3} may be applied whenever either side of the equality $(st^\ast)^\ast=s^\ast t^\ast$ is defined. Now suppose that $st$ is defined. Then
            $$ (st)^\ast \overset{\eqref{r1}}{=} (s(tt^\ast))^\ast \overset{\eqref{sc2}}{=} ((st)t^\ast)^\ast \overset{\eqref{r3}}{=} (st)^\ast t^\ast \overset{\eqref{r2}}{=} t^\ast (st)^\ast. $$
        Since $tt^\ast$ is defined by \eqref{r1} and $t^\ast(st)^\ast$ is defined by the previous calculation, it follows from \eqref{sc1} and \eqref{r1} that $(tt^\ast)(st)^\ast = t(st)^\ast$ is defined. Therefore, $s^\ast t$ is defined if and only if $st$ is defined, and in this case $t(st)^\ast$ is defined. Thus, we can use axiom \eqref{r4} whenever any side of the equality $s^\ast t = t(st)^\ast$ is defined.
    \end{obs}
    
    Right restriction semiconstellations constitute a class of (2,1)-algebras that includes all right restriction semigroupoids \cite{haag2026c}. Consequently, this class also encompasses restriction categories \cite{cockett2002restriction}, inverse semigroupoids \cite{cordeiro2023etale}, and right restriction semigroups \cite{hollings2009PP}. For further examples of restriction structures, we refer the reader to \cite{cockett2002restriction}. We now present an example of a right restriction semiconstellation that is not a semigroupoid.
    
    \begin{exe} \label{exe:restriction}
        Let $S_{m}^{n}$ be the semiconstellation from Example \ref{exe:semiconst}. We construct a new semiconstellation by adding idempotent elements to $S_{m}^{n}$. Consider the set
            $$ T_{m}^{n} = S_{m}^{n} \cup \{ s^\ast \} \cup \{ d_i^\ast \colon i \in I_m \}. $$
        endowed with a partial binary operation $\cdot$ defined by
        \begin{gather*}
            d_i \cdot s_{i}^{j} = s_{i+1}^{j}, \quad s_{i}^{j} \cdot s^\ast = s_{i}^{j}, \quad s^\ast \cdot s^\ast = s^\ast,\\
            d_i \cdot d_i^\ast = d_i, \quad d_i^\ast \cdot s_{i}^{j} = s_{i}^{j}, \quad\text{and}\quad d_i^\ast \cdot d_i^\ast = d_i^\ast.
        \end{gather*} Observe that the partial multiplication on $T_{m}^{n}$ extends the multiplication on $S_{m}^{n}$. The unary operation $\ast \colon T_{m}^{n} \to T_{m}^{n}$ is given by
            $$ d_i \mapsto d_i^\ast, \quad d_i^\ast \mapsto d_i^\ast, \quad s_{i}^{j} \mapsto s^\ast \quad\text{and}\quad s^\ast \mapsto s^\ast. $$
        Furthermore, we have $\mathfrak{D}(S_{m}^{n}) = \mathfrak{D}(T_{m}^{n})$. Hence, it follows from Corollary \ref{coro:semiconst} and Example \ref{exe:semiconst} that $T_{m}^{n}$ is a semigroupoid if and only if $m=1$.

        In the special case where $m=2$ and $n=1$, the operation on $T_2^1$ can be represented by the following graph, obtained from the graph representing $S_2^1$ by adding loops to most of its vertex.
        
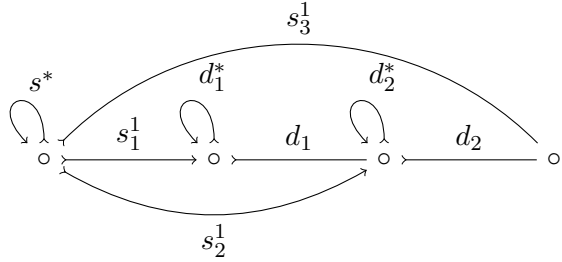
\begin{center}
            \begin{tikzpicture}[scale=0.75]
                \tikzstyle{every path}=[draw];
    
                \node (e0) at (0,0) {$\circ$};
                \node (e1) at (3,0) {$\circ$};
                \node (e2) at (6,0) {$\circ$};
                \node (e3) at (9,0) {$\circ$};
    
                \path[>->] (e0) to node[above]{$s_{1}^{1}$} (e1);
                \path[>-] (e1) to node[above]{$d_1$} (e2);
                \path[>-] (e2) to node[above]{$d_2$} (e3);
                \path[>->,bend right] (e0) to node[below]{$s_{2}^{1}$} (e2);
                \path[>-,bend left=45] (e0) to node[above]{$s_{3}^{1}$} (e3);
                \path[>->] (e0) edge [out=90,in=135,looseness=10] node[above right] {$s^\ast$} (e0);
                \path[>->] (e1) edge [out=90,in=135,looseness=10] node[above right] {$d_1^\ast$} (e1);
                \path[>->] (e2) edge [out=90,in=135,looseness=10] node[above right] {$d_2^\ast$} (e2);
            \end{tikzpicture}
        \end{center}
    \end{exe}
    
    In Section \ref{sec:6}, we develop a method for constructing right restriction semiconstellations by adjoining idempotent elements to a given semiconstellation. The construction in Example \ref{exe:restriction} illustrates one way to obtain a right restriction semiconstellation from the semiconstellation $S_{m}^{n}$, but if $n>1$, then such right restriction semiconstellation is not unique.\\

    Dually, one may define a \emph{left restriction semiconstellation} as a $(2,1)$-algebra whose unary operation satisfies the axioms dual to \eqref{r1}-\eqref{r4}. Observe that $(S,\ast)$ is a right restriction semiconstellation if and only if $(S^{op},\ast)$ is a left restriction semiconstellation. 
    
    Since semiconstellations are inherently one-sided structures, it may seem more natural to consider right restriction structures on right semiconstellations and left restriction structures on left semiconstellations. In particular, for left restriction structures on right semiconstellations, the analogue of Remark \ref{obs:restriction} no longer holds. Nevertheless, the following example shows that right semiconstellations may indeed admit left restriction structures.
    
    \begin{exe}
        Let $S$ be a right semiconstellation. Consider the set $T = S \cup \{1\}$, where $1\notin S$ is a new element. The partial multiplication $\cdot$ on $T$ extends that of $S$, that is, for all $s,t \in S$,  $s\cdot t=st,$ whenever $st$ is defined in $S$, and is further defined by $1\cdot x=x,$ for all $x\in T.$
        
         Define a unary operation $^+ \colon T \to T$ by $x^+ = 1$, for all $x \in T$. Then $(T,+)$ is a left restriction right semiconstellation. Note that ${^{1}T} = \{1\}$ and ${^{1s}T} = {^{s}T} = {^{s}S} \cup \{1\}$, for all $s \in S$. Therefore $1 \in \mathfrak{D}(T)$, except when $S^{(2)} = \emptyset$ in which case $\mathfrak{D}(T) = \emptyset$. By Corollary \ref{coro:semiconst}, with the exception of the case $S^{(2)} = \emptyset$, the semiconstellation $T$ is not a semigroupoid.
    \end{exe}
    
    In this paper we are mainly interested in right restriction structures on right semiconstellations. Accordingly, unless explicitly stated otherwise, the term \emph{restriction} always refers to a right restriction structure. Throughout the remainder of this section, $(S,\ast)$ denotes a restriction semiconstellation.
    
    We now establish some basic properties of restriction semiconstellations. 
    
    \begin{lemma} \label{lema:restriction}
        Let $e \in S^\ast$ and $s,t \in S$ be such that $st$ is defined. Then:
        \begin{enumerate}
            \item $ee$ is defined and $ee = e$.
            \item $e^\ast = e$.
            \item $(st)^\ast = t^\ast (st)^\ast = (s^\ast t)^\ast$.
        \end{enumerate}
    
        \begin{proof}
            (1) If $e \in S^\ast$, then there is $s \in S$ such that $e = s^\ast$. From \eqref{r1} and Remark \ref{obs:restriction}, we obtain that $ss^\ast$ is defined, and thus $s^\ast s^\ast$ is defined. Therefore,
                $$ ee = s^\ast s^\ast \overset{\eqref{r3}}{=} (ss^\ast)^\ast  \overset{\eqref{r1}}{=} s^\ast = e. $$
    
            \noindent(2) Let $e = s^\ast$. Then,
                $$ e = s^\ast \overset{\eqref{r1}}{=} s^\ast (s^\ast)^\ast \overset{\eqref{r2}}{=} (s^\ast)^\ast s^\ast \overset{\eqref{r3}}{=} (s^\ast s^\ast )^\ast \overset{(1)}{=} (s^\ast)^\ast = e^\ast. $$
    
            \noindent(3) In Remark \ref{obs:restriction} we proved that, if $st$ is defined, then $s^\ast t$ is defined and $(st)^\ast = t^\ast (st)^\ast$. For the second equality, we have
                $$ (s^\ast t)^\ast \overset{\eqref{r4}}{=} (t(st)^\ast)^\ast \overset{\eqref{r3}}{=} t^\ast (st)^\ast = (st)^\ast. $$
        \end{proof}
    \end{lemma}
    
    \begin{lemma} \label{lema:order}
        Define a relation on $S$ by $s \leq t$ if and only if $ts^\ast$ is defined and $ts^\ast = s$. Then:
        \begin{enumerate}
            \item $s \leq t$ if and only if exists $e \in S^\ast$ such that $te$ is defined and $te = s$.
            \item The relation $\leq$ is a partial order on $S$.
        \end{enumerate}
    
        \begin{proof}
            (1) If $s \leq t$, then $s^\ast \in S^\ast$, $ts^\ast$ is defined and $ts^\ast = s$. Conversely, suppose that exists $e \in S^\ast$ such that $te$ is defined and $te = s$. Then,
                $$ s^\ast = (te)^\ast \overset{\ref{lema:restriction}(2)}{=} (te^\ast)^\ast \overset{\eqref{r3}}{=} t^\ast e^\ast \overset{\ref{lema:restriction}(2)}{=} t^\ast e. $$
            Since $tt^\ast$ and $t^\ast e = s^\ast$ are defined, it follows from \eqref{sc2} that $t(t^\ast e) = ts^\ast$ is defined, and
                $$ ts^\ast = t(t^\ast e) \overset{\eqref{sc2}}{=} (tt^\ast)e \overset{\eqref{r1}}{=} te = s. $$
    
            \noindent(2) Reflexivity follows from \eqref{r1}. If $s \leq t$ and $t \leq r$, then $ts^\ast$ and $rt^\ast$ are defined. Hence,
                $$ s = ts^\ast = (rt^\ast)s^\ast \overset{\eqref{sc2}}{=} r(t^\ast s^\ast). $$
            That is, $t^\ast s^\ast$ and $r(t^\ast s^\ast)$ are defined, and $r(t^\ast s^\ast) = s$. By  \eqref{r3} we have $t^\ast s^\ast = (ts^\ast)^\ast \in S^\ast$. Therefore, it follows from (1) that $s \leq r$ and the relation is transitive. Lastly, suppose that that $s \leq t$ and $t \leq s$. Then $ts^\ast$ and $st^\ast$ are defined, $ts^\ast = s$ and $st^\ast = t$. Thus,
                $$ s = ts^\ast = (st^\ast)s^\ast \overset{\eqref{sc2}}{=} s(t^\ast s^\ast) \overset{\eqref{r2}}{=} s(s^\ast t^\ast) \overset{\eqref{sc2}}{=} (ss^\ast) t^\ast \overset{\eqref{r1}}{=} st^\ast = t. $$
            This proves that the relation $\leq$ is anti-symmetric, concluding that $\leq$ is a partial order.
        \end{proof}
    \end{lemma}
    
    For the next result, we regard semigroups as semiconstellations satisfying $S^{(2)} = S \times S$.
    
    \begin{prop} \label{prop:restriction}
        Let $\omega = \{ (e,f) \in S^\ast \times S^\ast \colon ef \text{ is defined} \}$. Then:
        \begin{enumerate}
            \item $\omega$ is an equivalence relation on $S^\ast$ and, for each $e \in S^\ast$, the equivalence class $\omega(e)$ of $e$ under $\omega$ is a semigroup with operation induced by $S$.
    
            \item Let $e,f \in S^\ast$. Then $f \in \omega(e)$ if and only if ${^eS} = {^fS}$ if and only if $S^e = S^f$.
        \end{enumerate}
        In particular, $S^\ast$ is a disjoint union of semigroups.
    
        \begin{proof}
            (1) The relation $\omega$ is reflexive by Lemma \ref{lema:restriction}(1) and symmetric by \eqref{r2}. Suppose that $ef$ and $fg$ are defined. From \eqref{sc1} we obtain that $(ef)g$ is defined, and in this case
                $$ (ef)g \overset{\eqref{r2}}{=} (fe)g \overset{\eqref{sc2}}{=} f(eg). $$
            Therefore, $eg$ is defined. This proves that $\omega$ is transitive and, hence, an equivalence relation. Consequently, $S^\ast$ is the disjoint union (as a set) of the equivalence classes of $\omega$. Furthermore, given $e,f \in S^\ast$, it follows that $ef$ is defined if and only if $f \in \omega(e)$, and in this case
                $$ ef \overset{\ref{lema:restriction}(1)}{=} (ee)f \overset{\eqref{sc2}}{=} e(ef). $$
            That is, $e(ef)$ is defined, and hence $ef \in \omega(e)$. Since $\omega$ is an equivalence relation, we conclude that, for $f,f' \in \omega(e) = \omega(f)$, the composition $ff'$ is defined and $ff' \in \omega(f) = \omega(e)$. This proves that each equivalence class of $\omega$ is a semigroup, concluding that $S^\ast$ is a disjoint union of semigroups.\\
    
            \noindent(2) Suppose that $f \in \omega(e)$. From Remark \ref{obs:restriction}, we obtain that
                $$ sf \text{ is defined} \iff s^\ast f \text{ is defined} \iff s^\ast \in \omega(f) = \omega(e). $$
            That is, ${^fS} = \{ s \in S \colon s^\ast \in \omega(e) \}$. Since this is valid for all $f \in \omega(e)$, we obtain in particular that ${^fS} = {^eS}$. Furthermore, since $ef$ is defined, we have
                $$ S^f \overset{\ref{obs:restriction}}{=} S^{ef} \overset{\eqref{r2}}{=} S^{fe} \overset{\ref{obs:restriction}}{=} S^e. $$
            Conversely, it follows from Lemma \ref{lema:restriction}(1) that $ee$ is defined. That is, $e \in {^eS} \cap S^e$. Therefore, if either ${^fS} = {^eS}$ or $S^f = S^e$, then $ef$ or $fe$ are defined and, hence, $f \in \omega(e)$.
        \end{proof}
    \end{prop}
    
    To end this subsection, we characterize the (2,1)-algebras that are, simultaneously, a right constellation and a restriction semiconstellation. Such semiconstellations are essential to Section \ref{sec:6} and appear in the last subsection of Section \ref{sec:4}.
    
    In the following, we say that an element $e \in S$ is an \textit{identity} if $ee$ is defined and, whenever $se$ and $et$ are defined, we have $se = s$ and $et = t$. We denote by $S_0$ the set of identities of $S$.
    
    \begin{prop} \label{prop:reduced}
        Let $S$ be a semiconstellation and $\ast \colon S \to S$ be a unary operation on $S$. Then, the following assertions are equivalent:
        \begin{enumerate}
            \item $s^\ast \in S_0 \cap S^s$, for all $s \in S$.
            \item $(S,\ast)$ is a right constellation and a restriction semiconstellation.
        \end{enumerate}
    
        \begin{proof}
            Assume that $(S,\ast)$ is a right constellation and a restriction semiconstellation. Then $ss^\ast$ is defined by \eqref{r1}, and $s^\ast s^\ast$ is defined by Lemma \ref{lema:restriction}(1). Suppose that $s^\ast t$ is defined. Then $s^\ast t = t$ by \eqref{const4}. On the other hand, if $ts^\ast$ is defined, then
                $$ ts^\ast \overset{\eqref{r1}}{=} (tt^\ast)s^\ast \overset{\eqref{sc2}}{=} t(t^\ast s^\ast) \overset{\eqref{r2}}{=} t(s^\ast t^\ast) \overset{\eqref{const4}}{=} tt^\ast \overset{\eqref{r1}}{=} t. $$
            This concludes that $s^\ast \in S_0 \cap S^s$, for all $s \in S$. Conversely, assume (1). We prove that axioms \eqref{const3}, \eqref{const4} and \eqref{r1}-\eqref{r4} are satisfied.
            
            From $s^\ast \in S^s$, we have that $ss^\ast$ is defined and, since $s^\ast \in S_0$, it must be $ss^\ast = s$. Furthermore, if $s^\ast t$ is defined, then $s^\ast t = t$. This shows \eqref{r1} and \eqref{const4}. Since \eqref{r1} holds, it follows from Remark \ref{obs:restriction} that $S^s = S^{s^\ast}$, for all $s \in S$. In particular, if $s^\ast t$ is defined, then $t(st)^\ast $ is defined and, since $s^\ast, (st)^\ast \in S_0$, we conclude that $s^\ast t = t = t(st)^\ast$, proving \eqref{r4}.
            
            Note that, if $e,f \in S_0$ and $ef$ is defined, then $e = ef = f$. In particular, $s^\ast t^\ast$ is defined if and only if $s^\ast = t^\ast$, in which case $t^\ast s^\ast = s^\ast s^\ast$ is defined and $s^\ast t^\ast = s^\ast = t^\ast s^\ast$. Thus, \eqref{r2} is satisfied. Furthermore, it follows from Remark \ref{obs:restriction} that $st^\ast$ is defined if and only if $s^\ast t^\ast$ is defined if and only if $s^\ast = t^\ast$. In particular, $s^\ast$ is the only element in $S^\ast$ such that $ss^\ast$ is defined and $ss^\ast = s$. Hence, \eqref{const3} holds. Lastly, if $st^\ast$ is defined, then $t^\ast = s^\ast$ and
                $$ (st^\ast)^\ast = (ss^\ast)^\ast = s^\ast = ss^\ast = s^\ast t^\ast. $$
            Therefore, condition \eqref{r3} is satisfied, concluding that $(S,\ast)$ is both a right constellation and a restriction semiconstellation.
        \end{proof}
    \end{prop}
    
    \begin{defi} \label{defi:reduced}
        A restriction semiconstellation $(S,\ast)$ is called  \textit{reduced} if it satisfies $S^\ast \subseteq S_0$. That is, $(S,\ast)$ is reduced if $s^\ast$ is an identity, for all $s \in S$.
    \end{defi}
    
    \begin{exe}
        The restriction semiconstellations $(T_{m}^{n},\ast)$ from Example \ref{exe:restriction} are reduced.
    \end{exe}
    
    In the proof of Proposition \ref{prop:reduced} we used that, if $s^\ast, e \in S_0 \cap S^s$, then $e = s^\ast$. Hence, the set $S_0 \cap S^s$ is empty or contains a single element. This shows that a reduced restriction structure on a semiconstellation, if exists, is uniquely defined. In this case, we must have $S^\ast = S_0$ since, if $e \in S_0$, then $e = ee^\ast = e^\ast$. The following result provides further characterizations of reduced semiconstellations.
    
    \begin{prop} \label{prop:char-reduced}
        The following conditions are equivalent:
        \begin{enumerate}
            \item $(S,\ast)$ is a reduced semiconstellation.
            \item Every element of $S^\ast$ is a right identity; that is, whenever $te$ is defined, one has $te=t$.
            \item The partial order $\leq$ is trivial; that is, $s\leq t$ if and only if $s=t$.
            \item If $e,f \in S^\ast$ and $ef$ is defined, then $e = f$.
            \item If $st$ is defined, then $(st)^\ast = t^\ast$.
        \end{enumerate}
    
        \begin{proof}
            Trivially (1) implies (2) and, by \eqref{r4}, (2) implies (1). Is is straightforward from the definition of $\leq$ that (2) implies (3). If $e,f \in S^\ast$ are such that $ef$ is defined, then $ef \leq e$ and $ef \leq f$. Therefore, (3) implies (4). If $st$ is defined then, by Lemma \ref{lema:restriction}(3), $t^\ast (st)^\ast$ is defined. Hence, (4) implies (5). Lastly, if condition (5) is satisfied, $e \in S^\ast$ and $se$ is defined, then
                $$ se \overset{\eqref{r1}}{=} (ss^\ast)e \overset{\eqref{sc2}}{=} s(s^\ast e) \overset{\eqref{r2}}{=} s(es^\ast) \overset{\eqref{r3}}{=} s(es^\ast)^\ast \overset{(5)}{=} s(s^\ast)^\ast \overset{\ref{lema:restriction}(2)}{=} ss^\ast \overset{\eqref{r1}}{=} s. $$
            Thus, (5) implies (2), which is equivalent to (1). This completes the proof.
        \end{proof}
    \end{prop}
    
    \subsection{Partial actions on sets} We introduce partial and global actions of restriction semiconstellations on sets, generalizing actions of categories, inverse semigroupoids and of restriction semigroups on sets. We also prove that partial actions of restriction semiconstellations on sets can be constructed by restricting global ones.\\

    In this subsection, $(S,\ast)$ denotes a restriction semiconstellation and $X$ denotes a set. An \emph{$S$-algebra} is a pair $(X;\alpha)$, where $\alpha=\{\alpha_s\}_{s\in S}$ is a family of partial unary operations on $X$. Equivalently, an $S$-algebra is a $(1)_{s\in S}$-algebra. For each $s\in S$, we denote the domain and image of $\alpha_s$ by ${}_sX$ and $X_s$, respectively.
    
    \begin{defi} \label{defi:actions}
        A \emph{partial action} of $(S,\ast)$ on $X$ is an $S$-algebra $(X;\alpha)$ satisfying the following conditions:
        \begin{enumerate} \Not{P}
            \item If $st$ is defined, then $\alpha_t^{-1}({_sX} \cap X_t) = {_{st}X} \cap {_tX}$. \label{P1}
            \item If $st$ is defined, then $\alpha_s(\alpha_t(x)) = \alpha_{st}(x)$, for all $x \in {_{st}X} \cap {_tX}$. \label{P2}
            \item $\alpha_e = id_{{_eX}}$, for all $e \in S^\ast$. \label{P3}
            \item ${_sX} \subseteq {_{s^\ast}X}$, for all $s \in S$. \label{P4}
        \end{enumerate}
        A \textit{global action} of $(S,\ast)$ on $X$ is a partial action satisfying the stronger condition:
        \begin{enumerate} \Not{P} \setcounter{enumi}{4}
            \item ${_sX} = {_{s^\ast}X}$, for all $s \in S$. \label{P5}
        \end{enumerate}
    \end{defi}
    
    \begin{obs}
        We now discuss some particular cases of restriction semiconstellations.
        \begin{enumerate}
            \item If $(S,\ast)$ is a restriction semigroup, then our definitions of partial and global actions coincide with those in \cite[Definitions 3.4 and 3.5]{gould2009partial} for restriction semigroups. In that paper, partial and global actions are called \emph{strong partial actions} and \emph{incomplete actions}, respectively.
            
            \item Every inverse semigroupoid $S$ has a natural restriction structure, given by $s^\ast = s^{-1}s$, for all $s \in S$. With this restriction structure, our definitions of partial and global actions generalize those in \cite[Definition 2.4]{demeneghi2025} for inverse semigroupoids.
    
            \item Every category $\mathcal{C}$ has a trivial restriction structure, given by $s^\ast = D(s)$, for all $s \in \mathcal{C}$. With this restriction structure, our definitions of partial and global actions generalize those in \cite[Definition 7]{nystedt2017} for categories.
    
            \item With the restriction structures described in (2) and (3), partial actions of inverse semigroupoids and categories coincide with the partial actions of $(S,\ast)$ satisfying the additional condition $X = \bigcup_{e \in S^\ast} {_eX}$. We have not included this condition in Definition \ref{defi:actions}, since it was not required for partial actions of restriction semigroups.
        \end{enumerate}
    \end{obs}

    We present two examples. In the first, we show that left composition on $S$ does not always define a partial action of $(S,\ast)$ on $S$. In the second, we show that the ``restricted composition'' on $S$ defines a global action of $(S,\ast)$ on $S$.
    
    \begin{exe} \label{exe:partial-action}
        Let $X = S$. For each $s \in S$, define a partial unary operation $\alpha_s$ by
            $$ {_sX} = S^s \quad\text{and}\quad \alpha_s(t) = st, \ \forall t \in {_sX}. $$
        Then $(S;\alpha)$ satisfies \eqref{P1}, \eqref{P2} and \eqref{P5}, but may not satisfy \eqref{P3}. In fact, note that
            $$ \alpha_e = id_{_eX} \iff et = \alpha_e(t) = t, \ \forall t \in S^e. $$
        That is, \eqref{P3} holds if and only if each $e \in S^\ast$ is a left identity in $S$. In particular, this condition is satisfied when $(S,\ast)$ is a reduced semiconstellation.
    \end{exe}
    
    \begin{exe} \label{exe:global}
        Let $Y = S$. For each $s \in S$, define a partial unary operation $\beta_s$ by
            $$ {_sY} = \{ t \in S^s \colon s^\ast t = t \} \quad\text{and}\quad \beta_s(t) = st, \ \forall t \in {_sY}. $$
        Note that ${_sY}$ is well defined since, by Remark \ref{obs:restriction}, we have $S^s = S^{s^\ast}$. We claim that $(Y;\beta)$ is a global action of $(S,\ast)$ on $S$. In fact, condition \eqref{P5} is satisfied due to Remark \ref{obs:restriction}, and condition \eqref{P3} hold by the definition of ${_eY}$. Assume that $st$ is defined and note that
            $$ {_{st}Y} \cap {_tY} = \{ r \in S^t \cap S^{st} \colon t^\ast r = r = (st)^\ast r \}, $$
        and
            $$ \beta_t^{-1}(Y_t \cap {_sY}) = \{ r \in S^t \colon t^\ast r = t, tr \in S^s, s^\ast (tr) = tr \}. $$
        
        Since $st$ is defined, we have that $S^{st} = S^t$ by Proposition \ref{prop:semiconst}. Thus, $r \in S^t$ if and only if $tr \in S^s$, and in this case $s(tr) = (st)r$. Hence, to show \eqref{P1}, it remains to prove that
            $$ t^\ast r = r \implies [(st)^\ast r = r \iff s^\ast (tr) = tr]. $$
        In fact, if $t^\ast r = r$ and $(st)^\ast r = r$, then
            $$ s^\ast (tr) \overset{\eqref{sc2}}{=} (s^\ast t)r \overset{\eqref{r4}}{=} (t(st)^\ast)r \overset{\eqref{sc2}}{=} t((st)^\ast r) = tr. $$
        On the other hand, if $t^\ast r = r$ and $s^\ast (tr) = tr$, then
            $$ (st)^\ast r \overset{\eqref{r4}}{=} r((st)r)^\ast \overset{\eqref{sc2}}{=} r(s(tr))^\ast \overset{\ref{lema:restriction}(3)}{=} r(s^\ast (tr))^\ast = r(tr)^\ast \overset{\ref{lema:restriction}(3)}{=} r(t^\ast r)^\ast = rr^\ast \overset{\eqref{r1}}{=} r. $$
        This proves \eqref{P1}. Condition \eqref{P2} follows from \eqref{sc2}, concluding that $(Y;\beta)$ is a global action.
    \end{exe}
    
    At times, it will be convenient to regard partial actions as certain functions between restriction semiconstellations.
    
    \begin{defi}
        A function $\varphi \colon S \to T$ between restriction semiconstellations $(S,\ast)$ and $(T,\ast)$ is called a \textit{strong $\wedge$-premorphism} if it satisfies the following conditions:
        \begin{enumerate} \Not{$\wedge$}
            \item If $st$ is defined, then $\varphi(s)\varphi(t)$ and $\varphi(st)\varphi(t)^\ast$ are defined and coincide. \label{pm1}
            \item $\varphi(s)^\ast = \varphi(s^\ast) \varphi(s)^\ast$, for all $s \in S$, or equivalently, $\varphi(s)^\ast \leq \varphi(s^\ast)$. \label{pm2}
        \end{enumerate}
    \end{defi}
    
    By Lemma \ref{lema:order}(1), condition \eqref{pm1} implies that $\varphi(s)\varphi(t) \leq \varphi(st)$ whenever $st$ is defined. In \cite{mcalister1977}, a function satisfying this inequality is called a \emph{$\wedge$-premorphism}, whereas one satisfying the reverse inequality is called a \emph{$\vee$-premorphism}. On the other hand, \cite{hollings2010extending} shows that not every $\wedge$-premorphism is a strong one.  Since only strong $\wedge$-premorphisms are needed in this paper, we shall simply refer to them as \emph{premorphisms}.
    
    \begin{lemma} \label{lema:premorphism}
        Let $\varphi \colon (S,\ast) \to (T,\ast)$ be a premorphism. Then:
        \begin{enumerate}
            \item If $e \in S^\ast$, then $\varphi(e) \in T^\ast$.
            \item If $s \leq t$, then $\varphi(s) \leq \varphi(t)$.
        \end{enumerate}
    
        \begin{proof}
            \noindent(1) Let $e \in S^\ast$. Then
                $$ \varphi(e) \overset{\eqref{r1}}{=} \varphi(e)\varphi(e)^\ast \overset{\ref{lema:restriction}(2)}{=} \varphi(e^\ast)\varphi(e)^\ast \overset{\eqref{pm2}}{=} \varphi(e)^\ast \in T^\ast. $$
    
            \noindent(2) If $s \leq t$, then $s = ts^\ast$. From \eqref{r1} and \eqref{pm1}, we obtain that $\varphi(s)\varphi(s^\ast)$ is defined. Hence,
            \begin{align*}
                \varphi(s) &= \varphi(s)\varphi(s)^\ast & \eqref{r1} \\
                &= \varphi(s) {\Big(} \varphi(s^\ast)\varphi(s)^\ast {\Big)} & \eqref{pm2} \\
                &= {\Big(} \varphi(s)\varphi(s^\ast) {\Big)} \varphi(s)^\ast & \eqref{sc2} \\
                &= {\Big(} \varphi(ts^\ast)\varphi(s^\ast)^\ast {\Big)} \varphi(s)^\ast & (1), \ref{lema:restriction}(2) \\
                &= {\Big(} \varphi(t) \varphi(s^\ast) {\Big)} \varphi(s)^\ast & \eqref{pm1} \\
                &= \varphi(t) {\Big(} \varphi(s^\ast) \varphi(s)^\ast{\Big)} & \eqref{sc2} \\
                &= \varphi(t) \varphi(s)^\ast. & \eqref{pm2}
            \end{align*}
            This proves that $\varphi(s) \leq \varphi(t)$.
        \end{proof}
    \end{lemma}
    
    Recall that $PT(X)$, the set of partial functions on $X$, is a restriction semigroup where
        $$ f \star g \colon g^{-1}({_fX} \cap X_g) \to f({_fX} \cap X_g), \quad (f \star g)(x) = f(g(x)), \ \forall x \in g^{-1}({_fX} \cap X_g), $$
    and $f^\ast = id_{_fX}$, for all $f \in PT(X)$.
    
    \begin{prop} \label{prop:actions}
        Let $(X;\alpha)$ be a $S$-algebra and denote by $\alpha \colon (S,\ast) \to PT(X)$ the function given by $\alpha(s) = \alpha_s$, for all $s \in S$. Then:
        \begin{enumerate}
            \item $(X;\alpha)$ is a partial action if and only if $\alpha \colon (S,\ast) \to PT(X)$ is a premorphism.
            \item $(X;\alpha)$ is a global action if and only if $\alpha \colon (S,\ast) \to PT(X)$ is a (2,1)-morphism.
        \end{enumerate}
    
        \begin{proof}
            \noindent(1) Suppose that $st$ is defined. Then
                $$ dom(\alpha_s \star \alpha_t) = \alpha_t^{-1}({_sX} \cap X_t) \quad\text{and}\quad dom(\alpha_{st} \star \alpha_t^\ast) = {_{st}X} \cap {_tX}. $$
            Hence, \eqref{P1} is equivalent to $dom(\alpha_s \star \alpha_t) = dom(\alpha_{st} \star \alpha_t^\ast)$, and in this case
                $$ (\alpha_{st} \star \alpha_t^\ast)(x) = \alpha_{st}(x) \quad\text{and}\quad (\alpha_s \star \alpha_t)(x) = \alpha_s(\alpha_t(x)), \ \quad \forall x \in {_{st}X} \cap {_tX}. $$
            Therefore, $(X;\alpha)$ satisfies \eqref{P1} and \eqref{P2} if and only if $\alpha$ satisfies \eqref{pm1}. Now, suppose that $(X;\alpha)$ satisfies \eqref{P3} and \eqref{P4}. Then
                $$ dom(\alpha_{s^\ast} \star \alpha_s^\ast) = {_{s^\ast}X} \cap {_sX} \overset{\eqref{P4}}{=} {_sX} \quad\text{and}\quad (\alpha_{s^\ast} \star \alpha_s^\ast)(x) = \alpha_{s^\ast}(x) \overset{\eqref{P3}}{=} x, \ \forall x \in {_sX}. $$
            This shows that $\alpha_{s^\ast} \star \alpha_s^\ast = id_{_sX} = \alpha_s^\ast$, proving \eqref{pm2}. Conversely, suppose that $\alpha$  satisfies \eqref{pm2}. Then \eqref{P3} follows from Lemma \ref{lema:premorphism}(1), and
                $$ {_sX} = dom(\alpha_s^\ast) \overset{\eqref{pm2}}{=} dom(\alpha_{s^\ast} \star \alpha_s^\ast) = {_{s^\ast}X} \cap {_sX}. $$
             Hence, condition \eqref{P4} holds, concluding that $(X;\alpha)$ is a partial action if and only if the function $\alpha \colon (S,\ast) \to PT(X)$ is a premorphism.\\
    
             \noindent(2) Suppose that $(X;\alpha)$ is a global action. Then, from \eqref{P3} and \eqref{P5}, we conclude that $\alpha_{s^\ast} = \alpha_s^\ast$. That is, $\alpha$ is compatible with the unary operation $^\ast$. Since every global action is a partial action, we know from (1) that $\alpha$ is a premorphism. On the other, if $st$ is defined, then $(st)^\ast \leq t^\ast$ by Lemma \ref{lema:restriction}(3). From Lemma \ref{lema:premorphism}(2) we obtain that $\alpha_{(st)^\ast} \leq \alpha_{t^\ast}$ and, thus
             \begin{align*}
                 \alpha_s \star \alpha_t &= \alpha_{st} \star \alpha_t^\ast & \eqref{pm1} \\
                 &= \alpha_{st} \star \alpha_{st}^\ast \star \alpha_t^\ast & \eqref{r1} \\
                 &= \alpha_{st} \star \alpha_{(st)^\ast} \star \alpha_{t^\ast} \\
                 &= \alpha_{st} \star \alpha_{(st)^\ast} & \ref{lema:premorphism}(2) \\
                 &= \alpha_{st} \star \alpha_{st}^\ast \\
                 &= \alpha_{st}. & \eqref{r1}
             \end{align*}
             This proves that $\alpha$ is compatible with composition and, hence, is a (2,1)-morphism. Conversely, suppose that $\alpha$ is a (2,1)-morphism. Then, for all $s \in S$, we have
                $$ {_{s^\ast}X} = dom(\alpha_{s^\ast}) = dom(\alpha_s^\ast) = {_sX}. $$
            Hence, $(X;\alpha)$ satisfies \eqref{P5}. Since every (2,1)-morphism is a premorphism, we conclude from (1) and the previous equality that $(X;\alpha)$ is a global action.
        \end{proof}
    \end{prop}
    
    \begin{corollary} \label{coro:actions}
        Let $(X;\alpha)$ be a partial action of $(S,\ast)$. Then:
        \begin{enumerate}
            \item If $s \leq t$, then ${_sX} \subseteq {_tX}$ and $\alpha_t|_{_sX} = \alpha_s$.
            \item If $e,f \in S^\ast$ and $ef$ is defined, then ${_{ef}X} = {_eX} \cap {_fX}$.
            \item If $(X;\alpha)$ is a global action, then ${_{st}X} \cap {_tX} = {_{st}X}$, whenever $st$ is defined.
        \end{enumerate}
    
        \begin{proof}
            \noindent(1) By Proposition \ref{prop:actions}, $\alpha \colon (S,\ast) \to PT(X)$ is a premorphism. Hence, it follows from Lemma \ref{lema:premorphism}(2) that $s \leq t$ implies $\alpha_s \leq \alpha_t$. That is, ${_sX} \subseteq {_tX}$ and $\alpha_t|_{_sX} = \alpha_s$.\\
    
            \noindent(2) By \eqref{r3} and Lemma \ref{lema:restriction}(2), we obtain that $S^\ast$ is a (2,1)-subalgebra of $(S,\ast)$. Thus, the partial action $(X;\alpha)$ induces a global action $\{\alpha_e\}_{e \in S^\ast}$ of $S^\ast$ on $X$. By Proposition \ref{prop:actions}, $\alpha|_{S^\ast} \colon (S^\ast,\ast) \to PT(X)$ is a (2,1)-morphism. Therefore, if $ef$ is defined, then
                $$ {_{ef}X} = dom(\alpha_{ef}) = dom(\alpha_e \star \alpha_f) \overset{\eqref{P3}}{=} dom(id_{_eX} \star id_{_fX}) = {_eX} \cap {_fX}. $$
    
            \noindent(3) Let $(X;\alpha)$ be a global action and suppose that $st$ is defined. Then
                $$ {_{st}X} \overset{\eqref{P5}}{=} {_{(st)^\ast}X} \overset{\ref{lema:restriction}(3)}{=} {_{t^\ast (st)^\ast}X} \overset{(2)}{=} {_{t^\ast}X} \cap {_{(st)^\ast}X} \overset{\eqref{P5}}{=} {_tX} \cap {_{st}X}. $$
        \end{proof}
    \end{corollary}
    
    To end this section, we show that partial actions can be constructed by restricting global ones. More precisely, since a partial action $(Y;\beta)$ is a $S$-algebra, given any subset $X$ of $Y$ we may consider the relative $S$-subalgebra $(Y;\beta)|_X$. In this case, $(Y;\beta)|_X$ is the pair $(X;\alpha)$, where $\alpha = \{\alpha_s\}_{s \in S}$ is the family of partial unary operations defined by
        $$ \alpha_s = \beta_s|_{_sX} \quad\text{where}\quad {_sX} = \{ x \in X \cap {_sY} \colon \beta_s(x) \in X \} = \beta_s^{-1}(X \cap Y_s) \cap X. $$
    The pair $(X;\alpha)$ is usually called the \textit{restriction} of $(Y;\beta)$ to $X$.
    
    \begin{prop} \label{prop:relative}
        The restriction of a global action $(Y;\beta)$ to a subset $X$ is a partial action.
        
        \begin{proof}
            If $e \in S^\ast$, then $\beta_e = id_{_eY}$ by \eqref{P3} and, thus, $\alpha_e = \beta_e|_{eX} = id_{_eY}|_{_eX} = id_{_eX}$. Therefore, $(X;\alpha)$ satisfies \eqref{P3}. Furthermore, for all $e \in S^\ast$, we have
                $$ {_eX} = \{ x \in X \cap {_eY} \colon \beta_e(x) \in X \} = X \cap {_eY}. $$
            Hence,
                $$ {_sX} \subseteq X \cap {_sY} \overset{\eqref{P5}}{=} X \cap {_{s^\ast}Y} = {_{s^\ast}X}. $$
            That is, $(X;\alpha)$ satisfies \eqref{P4}. To prove \eqref{P1} and \eqref{P2}, assume for a moment that, if $st$ is defined, then the following identities are valid:
            \begin{gather*}
                \beta_t^{-1}({_sX} \cap Y_t) \cap X = \alpha_t^{-1}({_sX} \cap X_t), \tag{I}\label{tag:I} \\
                \beta_s^{-1}(X \cap Y_{st}) \cap Y_t = \beta_s^{-1}(X \cap Y_s) \cap Y_t, \tag{II}\label{tag:II}.
            \end{gather*}
            Note that, by Proposition \ref{prop:actions}, $\beta \colon (S,\ast) \to PT(X)$ is a (2,1)-morphism. Hence,
                $$ \beta_{st}^{-1}(A) = (\beta_s \star \beta_t)^{-1}(A) \overset{\ref{prop:relations}}{=} \beta_t^{-1}(\beta_s^{-1}(A)), \quad \forall A \in \mathcal{P}(X). $$
            Furthermore, recall that $\beta_t^{-1}(A \cap B) = \beta_t^{-1}(A) \cap \beta_t^{-1}(B)$, for all $A,B \in \mathcal{P}(X)$. Thus,
            \begin{align*}
                {_{st}X} \cap {_tX} &= \beta_{st}^{-1}(X \cap {Y_{st}}) \cap \beta_t^{-1}(X \cap Y_t) \cap X \\
                &= \beta_{t}^{-1}(\beta_s^{-1}(X \cap {Y_{st}})) \cap \beta_t^{-1}(X \cap Y_t) \cap X \\
                &= \beta_t^{-1}{\Big(} \beta_s^{-1}(X \cap Y_{st}) \cap X \cap Y_t {\Big)} \cap X \\
                &\overset{\eqref{tag:II}}{=} \beta_t^{-1}{\Big(} \beta_s^{-1}(X \cap Y_{s}) \cap X \cap Y_t {\Big)} \cap X \\
                &= \beta_t^{-1}({_sX} \cap Y_t) \cap X \\
                &\overset{\eqref{tag:I}}{=} \alpha_t^{-1}({_sX} \cap X_t).
            \end{align*}
            Since $\alpha_s = \beta_s|_{_sX}$, if $x \in {_{st}X} \cap {_tX} \subseteq {_{st}Y} \cap {_tY}$, then
                $$ \alpha_s(\alpha_t(x)) = \beta_s(\beta_t(x)) \overset{\eqref{P2}}{=} \beta_{st}(x) = \alpha_{st}(x). $$
            This shows \eqref{P1} and \eqref{P2}. It remains to prove the claims \eqref{tag:I} and \eqref{tag:II}.\\
    
            \noindent\eqref{tag:I} If $x \in \beta_t^{-1}({_sX} \cap Y_t) \cap X \subseteq X \cap {_tY}$, then $\beta_t(x) \in {_sX} \cap Y_t \subseteq X$. From the definition of restriction, we obtain $x \in {_tX}$ and, hence, $\beta_t(x) = \alpha_t(x) \in X_t$. Therefore,
                $$ x \in \alpha_t^{-1}(\beta_t(x)) \subseteq \alpha_t^{-1}({_sX} \cap X_t). $$
            This shows the inclusion $\beta_t^{-1}({_sX} \cap Y_t) \cap X) \subseteq \alpha_t^{-1}({_sX} \cap X_t)$. For the reverse inclusion, note that $X_s \subseteq Y_s$. Since $\alpha_t$ is a restriction of $\beta_t$ to a subset of $X$, we obtain
                $$ \alpha_t^{-1}({_sX} \cap X_t) \subseteq \beta_t^{-1}({_sX} \cap X_t) \cap X \subseteq \beta_t^{-1}({_sX} \cap Y_t) \cap X. $$
            This proves that $\beta_t^{-1}({_sX} \cap Y_t) = \alpha_t^{-1}({_sX} \cap X_t)$.\\
    
            \noindent\eqref{tag:II} Again, it follows from Proposition \ref{prop:actions} that $\beta_s \star \beta_t = \beta_{st}$, whenever $st$ is defined. Hence,
            \begin{align*}
                Y_{st} = \beta_{st}({_{st}Y}) = \beta_s(\beta_t({_{st}Y})) = \beta_s(\beta_t(\beta_t^{-1}({_sY} \cap Y_t))) = \beta_s({_sY} \cap Y_t) \subseteq Y_s.
            \end{align*}
            Since ${_sY} \cap Y_t \subseteq \beta_s^{-1}(\beta_s({_sY} \cap Y_t))$ and $\beta_s^{-1}(X \cap {Y_s}) \subseteq {_sY}$, we obtain
            \begin{align*}
                \beta_s^{-1}(X \cap Y_{st}) \cap Y_t &= \beta_s^{-1}(X \cap Y_s \cap Y_{st}) \cap Y_t \\
                &= \beta_s^{-1}(X \cap Y_s) \cap \beta_s^{-1}(Y_{st}) \cap Y_t \\
                &= \beta_s^{-1}(X \cap Y_s) \cap \beta_s^{-1}(\beta_s({_sY} \cap Y_t)) \cap Y_t \\
                &= \beta_s^{-1}(X \cap Y_s) \cap \beta_s^{-1}(\beta_s({_sY} \cap Y_t)) \cap {_sY} \cap Y_t \\
                &= \beta_s^{-1}(X \cap Y_s) \cap {_sY} \cap Y_t \\
                &= \beta_s^{-1}(X \cap Y_s) \cap Y_t.
            \end{align*}
            This proves \eqref{tag:II}, concluding that $(X;\alpha)$ is a partial action.
        \end{proof}
    \end{prop}
    
    \section{The globalization problem for actions of \texorpdfstring{$(S,\ast)$}{}} \label{sec:4}
    
    In this section, we study the globalization problem for partial actions of restriction semiconstellations on sets. Proposition \ref{prop:relative} shows that every global action gives rise to a partial action by restriction. The globalization problem asks whether, conversely, given a partial action, it is possible to construct a global action whose restriction is isomorphic to the given partial action. As we shall see, this is not always the case. 
    
    Our strategy for addressing the globalization problem is the following. By Proposition \ref{prop:universal}, a partial action $(X;\alpha)$ is isomorphic to the restriction of a global action $(Y;\beta)$ if and only if there exists an injective full $S$-morphism $\varphi \colon (X;\alpha) \to (Y;\beta)$. Now suppose that there exists a subcategory $\mathcal{C}$ of $Alg(S)$, the category of $S$-algebras, containing all global actions of $(S,\ast)$ and such that $(X;\alpha)$ has a reflector $(\iota,(FX;\beta))$ in $\mathcal{C}$. Then Proposition \ref{prop:universal-2} implies that, whenever $(X;\alpha)$ is isomorphic to the restriction of a global action, $\iota \colon (X;\alpha) \to (FX;\beta)$ is necessarily an embedding. This reduces the globalization problem to two tasks: constructing reflectors for partial actions in a suitable subcategory of $Alg(S)$ containing all global actions, and characterizing when $\iota$ is an embedding.

    Following the usual terminology for partial actions, we call an injective full $S$-morphism an \emph{embedding}. A global action $(Y;\beta)$ is called a \emph{globalization} of a partial action $(X;\alpha)$ if there exists an embedding $(X;\alpha) \to (Y;\beta)$, and $(X;\alpha)$ is said to be \emph{globalizable} if it admits a globalization. In this terminology, the globalization problem for a given class of partial actions asks whether every partial action in that class is globalizable.
    
    \subsection{Constructing reflectors} We prove that every partial action has a reflector in the full subcategory of the category of $S$-algebras whose objects are \emph{almost global actions}.\\
    
    In this section, $(S,\ast)$ denotes a restriction semiconstellation and $(X;\alpha)$ a partial action of $(S,\ast)$. Our first goal is to construct an $S$-algebra $(FX;\beta)$. The following notation will be used in the remaining of this section. Let $S^1 = S \cup \{1\}$, where $1$ is a symbol not belonging to $S$, and consider the $S$-algebra $(S^1\times X;\lambda)$, where
        $$ dom(\lambda_s) = (S^s \cup \{1\}) \times X, \quad \lambda_s(t,x) = (st,x) \quad\text{and}\quad \lambda_s(1,x) = (s,x). $$
    Also, let $i\colon X\to S^1\times X$ be the map given by $i(x)=(1,x)$. Note that in general $i$ is not an $S$-morphism. We will construct a quotient of a relative $S$-subalgebra of $(S^1\times X;\lambda)$ in the following way. Consider the set
        $$ \overline{X} = \{ (s,x) \colon x \in {_{s^\ast}X} \} \cup \{ (1,x) \colon x \in X \} \subseteq S^1 \times X. $$
    Then, the operations of the relative $S$-subalgebra $(\overline{X};\beta) = (S^1 \times X;\lambda)|_{\overline{X}}$ are defined by
        $$ \beta_s = \lambda_s|_{_s\overline{X}}, \quad\text{where}\quad {_s\overline{X}} = \{ (t,x) \colon t \in S^s, x \in {_{(st)^\ast}X} \} \cup \{ (1,x) \colon x \in {_{s^\ast}X} \}. $$
    Note that ${_s\overline{X}} \subseteq \overline{X}$. Indeed, if $st$ is defined, then $(st)^\ast = t^\ast (st)^\ast \leq t^\ast$ by Lemma \ref{lema:restriction}(3). From Corollary \ref{coro:actions}, we conclude that ${_{(st)^\ast}X} \subseteq {_{t^\ast}X}$. Hence, $(t,x) \in {_s\overline{X}}$ implies $(t,x) \in \overline{X}$.

    Let $\simeq$ be an equivalence relation on $\overline{X}$. By Proposition \ref{prop:congruence-2}, the partial functions on $\overline{X}$ induce partial functions on the quotient set $\overline{X}/\simeq$ if and only if $\simeq$ is is invariant under $\beta$. Let $\pi \colon \overline{X} \to \overline{X}/\simeq$ denote the canonical projection. Then the map $\pi \circ i \colon X \to \overline{X}/\simeq$ is an $S$-morphism if and only if $(s,x) \in \overline{X}$ and $\pi(s,x) = \pi(1,\alpha_s(x))$, whenever $x \in {_sX}$. This motivates the definition of the relation $\sim$ on $\overline{X}$ by 
    \begin{align*}
        (s,x) \sim (1,y) \iff x \in {_sX} \text{ and } \alpha_s(x) = y.
    \end{align*}

    By Proposition \ref{prop:congruence}, there exists a smallest equivalence relation on $\overline{X}$ containing $\sim$ that is invariant under $\beta$, namely $\mathcal{E}_\beta(\sim)$. We denote this relation by $\simeq$ and set $FX = \overline{X}/\simeq$. We write $\pi(a,x)=a\otimes x$ for the equivalence class of $(a,x)$, where $a\in S^1$. The quotient $FX$ inherits a family of partial functions, which we continue to denote by $\beta=\{\beta_s\}_{s\in S}$, given by
        $$ {_sFX} = \{ a \otimes x \colon (a,x) \simeq (b,y) \text{ for some } (b,y) \in {_s\overline{X}} \} \quad\text{and}\quad \beta_s(a \otimes x) = \pi(\beta_s(b,y)). $$
    It follows from \eqref{P4} that $\iota \colon (X;\alpha) \to (FX;\beta)$, defined by $\iota(x) = 1 \otimes x$, is a $S$-morphism.
    
    \begin{obs} \label{obs:construction}
        The choice of the subset $\overline{X}\subseteq S^1\times X$ is not arbitrary. Indeed, assuming that the canonical projection $\pi \colon (\overline{X};\beta) \to (FX;\beta)$ is a strong $S$-morphism and that $(\pi\circ i,(FX;\beta))$ is a reflector of $(X;\alpha)$ in the full subcategory of $Alg(S)$ whose objects satisfy \eqref{P5}, one can prove that $\overline{X}$ is necessarily the subset defined above. Since the proof is rather lengthy and is not needed elsewhere in the paper, we omit it.
    \end{obs}

    In the remainder of this section, we prove that $(\iota,(FX;\beta))$ is a reflector for $(X;\alpha)$ in a suitable subcategory of $Alg(S)$. To this end, we recall the construction of the equivalence relation $\simeq$. Let $\sim$ be the relation on $\overline{X}$ defined above, and let $\beta = \{\beta_s\}_{s \in S}$ be the family of partial functions on $\overline{X}$. We define inductively a sequence of equivalence relations. Let $\simeq^{(0)}$ be the smallest equivalence relation containing $\sim$. For each $n\geq0$, define
        $$ (a,x) \simeq_\beta^{(n)} (b,y) \iff \begin{matrix}
            \exists p \in S, (a',x), (b',y) \in {_p\overline{X}} \colon (a',x) \simeq^{(n)} (b',y), \\
            (a,x) = \beta_p(a',x) \text{ and } (b,y) = \beta_p(b',y).
        \end{matrix} $$
    Then let $\simeq^{(n+1)}$ be the smallest equivalence relation containing both $\simeq^{(n)}$ and $\simeq_\beta^{(n)}$. Finally, the equivalence relation $\simeq$ is the relation $\bigcup_{n \geq 0} \simeq^{(n)}$.
    
    \begin{prop} \label{prop:construction}
        Let $(FX;\beta)$ be the $S$-algebra constructed above. Then:
        \begin{enumerate}
            \item ${_sFX} = {_{s^\ast}FX}$, for all $s \in S$.
            \item $\beta_e = id_{_eFX}$, for all $e \in S^\ast$.
            \item If $st$ is defined, then
                $$ {_{st}FX} \subseteq \beta_t^{-1}(FX_t \cap {_sFX}) \quad\text{and}\quad \beta_{st}(x) = \beta_s(\beta_t(x)), \ \forall x \in {_{st}FX}. $$
        \end{enumerate}
    
        \begin{proof}
            \noindent(1) Note that
                $$ (1,x) \in {_s\overline{X}} \iff x \in {_{s^\ast}X} \overset{\ref{lema:restriction}(2)}{=} {_{(s^\ast)^\ast}X} \iff (1,x) \in {_{s^\ast}\overline{X}}, $$
            and
                $$ (t,x) \in {_s\overline{X}} \iff t \in S^s \overset{\ref{obs:restriction}}{=} S^{s^\ast}, x \in {_{(st)^\ast}X} \overset{\ref{lema:restriction}(3)}{=} {_{(s^\ast t)^\ast}X} \iff (t,x) \in {_{s^\ast}\overline{X}}. $$
            Thus, $a \otimes x \in {_sFX}$ if and only if $(a,x) \simeq (b,y) \in {_s\overline{X}} = {_{s^\ast}\overline{X}}$ if and only if $a \otimes x \in {_{s^\ast}FX}$.\\
    
            \noindent(2) Let $a \otimes x \in {_eFX}$. Then $(a,x) \simeq (b,y)$ for some $(b,y) \in {_e\overline{X}}$. If $b = 1$, then $y \in {_{e^\ast}X}$. But then $(e,y) \in \overline{X}$ and $(e,y) \sim (1,\alpha_e(y)) = (1,y)$ by \eqref{P3}. From Lemma \ref{lema:restriction}(1), we obtain $e \in S^e$ and $y \in {_{e^\ast}X} = {_{(ee)^\ast}X}$. That is, $(a,x) \simeq (1,y) \simeq (e,y)$ where $(e,y) \in {_e\overline{X}}$. Therefore, we can assume that $b \in S$. In this case, we have $b \in S^e$ and $y \in {_{b^\ast}X} \cap {_{(eb)^\ast}X}$. Hence,
                $$ ((eb)^\ast,y) \sim (1,\alpha_{(eb)^\ast}(y)) \overset{\eqref{P3}}{=} (1,y) \quad\text{and}\quad \beta_e(b,y) = (eb,y) \overset{\eqref{r4}}{=} (b(eb)^\ast,y). $$
            Note that $(1,y), ((eb)^\ast,y) \in {_b\overline{X}}$ and, thus
                $$ (b,y) = \beta_b(1,y) \simeq_\beta^{(0)} \beta_b((eb)^\ast,y) = (b(eb)^\ast,y) \overset{\eqref{r4}}{=} (eb,y) = \beta_e(b,y). $$
            This shows $\beta_e(a \otimes x) = \pi(\beta_e(b,y)) = b \otimes y = a \otimes x$, concluding that $\beta_e = id_{_eFX}$.\\
                
            \noindent(3) Suppose that $st$ is defined and let $a \otimes x \in {_{st}FX}$. Then $(a,x) \simeq (b,y)$ for some $(b,y) \in {_{st}\overline{X}}$. If $b \in S$, then $b \in S^{st}$ and $x \in {_{b^\ast}X} \cap {_{((st)b)^\ast}X}$. Note that
                $$ ((st)b)^\ast \overset{\eqref{sc2}}{=} (s(tb))^\ast \overset{\ref{lema:restriction}(3)}{=} (tb)^\ast(s(tb)^\ast)^\ast \leq (tb)^\ast. $$
            Thus, it follows from Corollary \ref{coro:actions}(1) that $y \in {_{((st)b)^\ast}X} \subseteq {_{(tb)^\ast}X}$. By Proposition \ref{prop:semiconst}, we have $b \in S^{st} = S^t$, from where we conclude that
                $$ (a, x) \simeq (b, y) \in {_{st}\overline{X}} \cap {_t\overline{X}} \quad\text{and}\quad \beta_t(a \otimes x) = tb \otimes y. $$
            On the other hand, it follows from \eqref{sc1} that $tb \in S^s$ and $y \in {_{(s(tb))^\ast}X}$. That is, $(tb,y) \in {_s\overline{X}}$. But then $a \otimes x = b \otimes y \in \beta_t^{-1}(FX_t \cap {_sFX})$ and
                $$ \beta_s(\beta_t(a \otimes x)) = \beta_s(tb \otimes y) = s(tb) \otimes y = (st)b \otimes y = \beta_{st}(a \otimes y).  $$
            Analogously, if $b = 1$, then $x \in {_{(st)^\ast}X} \subseteq {_{t^\ast}X}$ implies that $(1,y) \in {_t\overline{X}}$ and $\beta_t(a \otimes x) = t \otimes y$ where $(t,y) \in {_s\overline{X}}$. Hence, $1 \otimes y \in \beta_t^{-1}(FX_t \cap {_sFX})$ and $\beta_s(\beta_t(a \otimes x)) = st \otimes y = \beta_{st}(a \otimes x)$.
        \end{proof}
    \end{prop}
    
    \begin{defi} \label{defi:almost}
        We say that a $S$-algebra $(Y;\gamma)$ is an \textit{almost global action} of $(S,\ast)$ if it satisfies \eqref{P5} and, whenever $st$ is defined, 
            $$ {_{st}Y} \subseteq \gamma_t^{-1}(Y_t \cap {_sY}) \quad\text{and}\quad \gamma_{st}(x) = \gamma_s(\gamma_t(x)), \ \forall x \in {_{st}Y}. $$
        We denote by $\mathcal{A}_a(S,\ast)$ the full subcategory of $Alg(S)$ whose objects are almost global actions.
    \end{defi}

    Note that the global actions of $(S,\ast)$ are precisely the almost global actions satisfying condition \eqref{P3} and the inclusion $\gamma_t^{-1}(Y_t \cap {_sY}) \subseteq {_{st}Y}$, whenever $st$ is defined. Therefore, the full subcategory $\mathcal{A}(S,\ast)$ of $Alg(S)$, whose objects are the global actions of $(S,\ast)$, is contained in the category $\mathcal{A}_a(S,\ast)$. By Proposition \ref{prop:construction}, we have that $(FX;\beta) \in \mathcal{A}_a(S,\ast)$.

    \begin{theorem} \label{teo:construction}
        The pair $(\iota,(FX;\beta))$ is a reflector for $(X;\alpha)$ in the subcategory $\mathcal{A}_a(S,\ast)$.
    
        \begin{proof}
            Let $\varphi \colon (X;\alpha) \to (Y;\gamma)$ be a $S$-morphism, where $(Y;\gamma)$ is an almost global action. We need to prove that there exists a unique $S$-morphism $\Phi \colon (FX;\beta) \to (Y;\gamma)$ satisfying $\Phi \circ \iota = \varphi$. Note that, if $s \otimes x \in FX$, then $1 \otimes x \in {_sFX}$ and $\beta_s(1 \otimes x) = s \otimes x$. Therefore, if $\Phi$ is a $S$-morphism satisfying $\Phi \circ \iota = \varphi$, then
                $$ \Phi(s \otimes x) = \Phi(\beta_s(\iota(x)) = \gamma_s(\Phi(\iota(x)) = \gamma_s(\varphi(x)) \quad\text{and}\quad \Phi(1 \otimes x) = \Phi(\iota(x)) = \varphi(x). $$
            That is, the $S$-morphism $\Phi$ satisfying $\Phi \circ \iota = \varphi$, if exists, is uniquely defined.
    
            We begin showing that $\Phi \colon FX \to Y$ is well defined. Consider the function $\Phi \colon \overline{X} \to Y$ given by $\Phi(s,x) = \gamma_s(\varphi(x))$ and $\Phi(1,x) = \varphi(x)$. The element $\Phi(1,x)$ is always well defined. On the other hand, if $(s,x) \in \overline{X}$, then $x \in {_{s^\ast}X}$. Since $\varphi \colon (X;\alpha) \to (Y;\gamma)$ is a $S$-morphism and $(Y;\gamma)$ satisfies \eqref{P5}, it follows that $\varphi(x) \in {_sY} = {_{s^\ast}Y}$. Hence, the element $\Phi(s,x)$ is well defined. Now we prove that $(a,x) \simeq (b,y)$ implies $\Phi(a,x) = \Phi(b,y)$. Since $\simeq$ is the relation $\bigcup_{n \geq 0} \simeq^{(n)}$, we proceed by induction in $n \geq 0$.
    
            Suppose that $(a,x) \simeq^{(0)} (b,y)$. If $a,b \in S$, then $x \in {_aX}$, $y \in {_bX}$ and $\alpha_a(x) = \alpha_b(y)$. Thus,
                $$ \Phi(a,x) = \gamma_a(\varphi(x)) = \varphi(\beta_a(x)) = \varphi(\alpha_b(y)) = \gamma_b(\varphi(y)) = \Phi(b,y). $$
            The proof is analogous for $a = 1$ or $b = 1$. Now, suppose that $(a',x) \simeq^{(n)} (b',y)$ implies $\Phi(a',x) = \Phi(b',y)$ and let $(a,x) \simeq_\beta^{(n)} (b,y)$. Then there are $p \in S$ and $(a',x), (b',y) \in {_p\overline{X}}$ such that $(a,x) = \beta_p(a',x)$, $(b,y) = \beta_p(b',y)$ and $(a',x) \simeq^{(n)} (b',y)$. Again, suppose that $a',b' \in S$. Then
                $$ (a,x) = (pa',x), \quad (b,y) = (pb',y) \quad\text{and}\quad \gamma_{a'}(\varphi(x)) = \gamma_{b'}(\varphi(y)), $$
            where the last equality follows from the induction hypothesis. From $(a,x) = (pa',x)$, we obtain $\varphi(x) \in {_{pa'}Y}$. Since $(Y;\gamma)$ is an almost global action, it follows that $\gamma_{pa'}(\varphi(x)) = \gamma_p(\gamma_{a'}(\varphi(x))$. Analogously, we conclude that $\gamma_{pb'}(\varphi(y)) = \gamma_p(\gamma_{b'}(\varphi(y))$. Hence,
                $$ \Phi(a,x) = \gamma_{pa'}(\varphi(x)) = \gamma_p(\gamma_{a'}(\varphi(x))) = \gamma_p(\gamma_{b'}(\varphi(y)) = \gamma_{pb'}(\varphi(y)) = \Phi(b,y). $$
            The case where $a' = 1$ or $b' = 1$ is analogous. This shows that, if $(a,x) \simeq^{(n)} (b,y)$ or $(a,x) \simeq_\beta^{(n)} (b,y)$, then $\Phi(a,x) = \Phi(b,y)$. Since $\simeq^{(n+1)}$ is defined by $\mathcal{E}(\simeq^{(n)} \cup \simeq_\beta^{(n)})$, it follows that $(a,x) \simeq^{(n+1)} (b,y)$ implies $\Phi(a,x) = \Phi(b,y)$. By induction, we conclude that $(a,x) \simeq (b,y)$ implies $\Phi(a,x) = \Phi(b,y)$. Therefore, the function $\Phi \colon \overline{X} \to Y$ induces the function $\Phi \colon FX \to Y$ described at the beginning of the proof.
    
            Lastly, we prove that $\Phi$ is a $S$-morphism $(FX;\beta) \to (Y;\gamma)$. In fact, suppose that $a \otimes x \in {_sFX}$. Then $(a,x) \simeq (b,y)$ for some $(b,y) \in {_s\overline{X}}$. Since $\Phi(a \otimes x) = \Phi(b \otimes y)$, we can assume that $(a,x) \in {_s\overline{X}}$. Suppose $a \in S$. Then $a \in S^s$ and $x \in {_{a^\ast}X} \cap {_{(sa)^\ast}X}$. Since $\varphi$ is a $S$-morphism and $(Y;\gamma)$ is an almost global action, it follows that
                $$ \varphi(x) \in {_{a^\ast}Y} \cap {_{(sa)^\ast}Y} = {_aY} \cap {_{sa}Y} \subseteq \gamma_a^{-1}(Y_a \cap {_sY}) \quad\text{and}\quad \gamma_{sa}(\varphi(x)) = \gamma_s(\gamma_a(\varphi(x)). $$
            Therefore, $\Phi(a \otimes x) = \gamma_a(\varphi(x)) \in {_sY}$ and $\gamma_s(\Phi(a \otimes x)) = \Phi(\beta_s(a \otimes x))$. The proof is analogous if $a = 1$. This concludes that $\Phi \colon (FX;\beta) \to (Y;\gamma)$ is a $S$-morphism and, by the first observation, it is the unique $S$-morphism satisfying $\Phi \circ \iota = \varphi$.
        \end{proof}
    \end{theorem}
    
    To end this subsection, we characterize, in terms of the relation $\simeq$, when the partial action $(X;\alpha)$ is isomorphic to the relative $S$-subalgebra $(FX;\beta)|_{\iota(X)}$. We recall that an embedding is an injective full $S$-morphism.
    
    \begin{prop} \label{prop:construction-2}
        The following assertions are equivalent:
        \begin{enumerate}
            \item $(X;\alpha)$ is isomorphic to a relative $S$-subalgebra of $(Y;\gamma) \in \mathcal{A}_a(S,\ast)$.
            \item The $S$-morphism $\iota \colon (X;\alpha) \to (FX;\beta)$ is an embedding.
            \item If $(s,x) \simeq (1,y)$, then $x \in {_sX}$ and $\alpha_s(x) = y$.
        \end{enumerate}
    
        \begin{proof}
            By Proposition \ref{prop:universal}, $(X;\alpha)$ is isomorphic to a relative $S$-subalgebra of a $S$-algebra $(Y;\gamma)$ if and only if there is an embedding $(X;\alpha) \to (Y;\gamma)$. Since $(FX;\beta)$ is an almost global action, (2) implies (1). On the other hand, since $(\iota,(FX;\beta))$ is a reflector for $(X;\alpha)$ in $\mathcal{A}_a(S,\ast)$, it follows from Proposition \ref{prop:universal-2} that (1) implies (2).\\
    
            Suppose that (2) is satisfied and let $(s,x) \simeq (1,y)$. Since $(s,x) \in \overline{X}$ we have $(1,x) \in {_s\overline{X}}$ and $(s,x) = \beta_s(1,x)$. Therefore,
                $$ \iota(x) = 1 \otimes x \in {_sFX} \quad\text{and}\quad \beta_s(\iota(x)) = s \otimes x = 1 \otimes y = \iota(y). $$
            Since $\iota$ is full, we must have $x \in {_sX}$ and, in this case $\iota(y) = \beta_s(\iota(x)) = \iota(\alpha_s(x))$. Since $\iota$ is injective, it follows that $\alpha_s(x) = y$, concluding (3)\\
    
            Suppose that (3) is satisfied. We prove that the $S$-morphism $\iota$ is an embedding. Note that, if $\iota(x) = \iota(y)$, then there is $(s,z) \in \overline{X}$ such that $(1,x) \simeq (s,z) \simeq (1,y)$ and, in this case, $x = \alpha_s(z) = y$. Therefore, $\iota$ is injective. Now, suppose that $\iota(x) \in {_sFX}$ and $\beta_s(\iota(x)) = \iota(y)$, for some $y \in X$. Then $(1,x) \simeq (a,z)$ for some $(a,z) \in {_s\overline{X}}$. If $a = 1$, then $(1,x) \simeq (1,z)$. Since $\iota$ is injective, we have $x = z$. Hence, $(1,x) \in {_s\overline{X}}$ and $(s,x) = \beta_s(1,x) \simeq (1,y)$, which implies $x \in {_sX}$. On the other hand, suppose that $a \in S$. Then $(1,x) \simeq (a,z)$ implies $z \in {_aX}$ and $\alpha_a(z) = x$. Therefore, from $(a,z) \in {_s\overline{X}}$, we obtain that
            \begin{align*}
                1 \otimes y = \iota(y) = \beta_s(\iota(x)) = sa \otimes z &\implies z \in {_{sa}X} \cap {_aX} \overset{\eqref{P1}}{=} \alpha_a^{-1}(X_a \cap {_sX}) \\
                &\implies x = \alpha_a(z) \in {_sX} \overset{\eqref{P4}}{\subseteq} {_{s^\ast}X} \\
                &\implies (1,x) \in {_s\overline{X}}.
            \end{align*}
            Again, we have $(s,x) = \beta_s(1,x) \simeq (1,y)$, which implies $x \in {_sX}$, concluding that $\iota$ is full.
        \end{proof}
    \end{prop}
    
    Since the category of global actions of $(S,\ast)$ is a subcategory of $\mathcal{A}_a(S,\ast)$, Proposition \ref{prop:construction-2} implies that, if $(X;\alpha)$ is globalizable, then the $S$-morphism $\iota \colon (X;\alpha) \to (FX;\beta)$ is an embedding. The converse, however, holds only when $(FX;\beta)$ is itself a global action. This will be the case in the following subsections.
    
    \subsection{Sufficient conditions} We show that, if disruptive elements act globally, then $(FX;\beta)$ is a global action and a globalization for $(X;\alpha)$. Consequently, we obtain a positive answer to the globalization problem for partial actions of restriction semigroupoids on sets.\\
    
    Recall from Lemma \ref{lema:equivalence} that, given an equivalence relation $R$, the smallest equivalence relation containing $R$ is the relation $\mathcal{E}(R)$ defined as follows: $x \mathcal{E}(R) y$ if and only if there exist $x_1,\dots,x_n$ such that $x=x_1$, $x_n=y$, and, for each $i=1,\dots,n-1$, either $x_iRx_{i+1}$ or $x_{i+1}Rx_i$. If $R$ is symmetric, we always have $x_iRx_{i+1}$ for every $i$. The integer $n \geq 1$ is called the \emph{length} of the sequence relating $x$ to $y$. If $n$ is minimal among all such sequences, we say that $x_1,\dots,x_n$ is a \emph{minimal sequence}.

    In this subsection, we study the relation $\simeq^{(1)}:=\mathcal{E}(\simeq^{(0)}\cup\simeq_\beta^{(0)})$. The following simple property of minimal sequences will be used repeatedly.

    \begin{lemma}\label{lem:minimal-sequence}
    Let $(a_1,x_1),\dots,(a_n,x_n) \in \overline{X}$ be a minimal sequence relating $(a,x)$ to $(b,y)$ with respect to $\simeq^{(1)}$. Then $a_2,\dots,a_{n-1} \in S$.
    \end{lemma}

    \begin{proof}
    Suppose that $a_i=1$ for some $i\in\{2,\dots,n-1\}$. Then
        $$ (a_{i-1},x_{i-1}) \simeq^{(0)} (1,x_i) \simeq^{(0)} (a_{i+1},x_{i+1}). $$
    Since $\simeq^{(0)}$ is transitive, it follows that $(a_{i-1},x_{i-1}) \simeq^{(0)} (a_{i+1},x_{i+1})$. Hence, removing the term $(1,x_i)$ from the sequence yields a shorter sequence relating $(a,x)$ to $(b,y)$, contradicting its minimality. \end{proof}
    
    We now prove a technical lemma.
    
    \begin{lemma} \label{lema:technical}
        The following properties hold:
        \begin{enumerate}
            \item If $st$ is defined, then ${_tX} \cap {_{(st)^\ast}X} = \alpha_t^{-1}({_{s^\ast}X} \cap X_t)$.
        
            \item If $x \in {_sX}$ and $s \in S^p$, then $(s,x) \in {_p\overline{X}}$ if and only if $(1,\alpha_s(x)) \in {_p\overline{X}}$.
    
            \item If $(a,x) \simeq (b,y)$ for some $(b,y) \in {_p\overline{X}}$, then exists $(t,y) \in {_p\overline{X}}$ such that
                $$ t \in S^p, \quad (a,x) \simeq (t,y) \quad\text{and}\quad\beta_p(t,y) = \beta_p(b,y). $$
            Moreover, this property is valid for any transitive relation containing $\sim$.
        \end{enumerate}
    
        \begin{proof}
            \noindent(1) If $st$ is defined, then $(st)^\ast = t^\ast (st)^\ast$ by Lemma \ref{lema:restriction}(3). From \eqref{r1} and \eqref{sc2}, we obtain that $t(t^\ast (st)^\ast) = t(st)^\ast$ is defined. On the other hand, it follows from \eqref{P3} that ${_{(st)^\ast} X} = X_{(st)^\ast}$ and $\alpha_{(st)^\ast}^{-1}(A) = A$, for all subsets $A$ of $X_{(st)^\ast}$. Therefore,
            \begin{align*}
                {_tX} \cap {_{(st)^\ast}X} &= {_tX} \cap {_tX} \cap {_{(st)^\ast}X} \\
                &= {_tX} \cap \alpha_{(st)^\ast}^{-1}({_tX} \cap X_{(st)^\ast}) & \eqref{P3} \\
                &= {_tX} \cap {_{t(st)^\ast}X} \cap {_{(st)^\ast}X} & \eqref{P1} \\
                &= {_tX} \cap {_{s^\ast t}X} \cap {_{(st)^\ast}X} & \eqref{r4} \\
                &= {_tX} \cap {_{s^\ast t}X} \cap {_{(s^\ast t)^\ast}X} & \ref{lema:restriction}(3) \\
                &= {_tX} \cap {_{s^\ast t}X} & \eqref{P4} \\
                &= \alpha_t^{-1}({_{s^\ast}X} \cap X_t). & \eqref{P1}
            \end{align*}
    
            \noindent(2) Suppose that $x \in {_sX} \subseteq {_{s^\ast}X}$ and $s \in S^p$. Then
            \begin{align*}
                (s,x) \in {_p\overline{X}} &\iff x \in {_{(ps)^\ast}X} \cap {_sX} \\
                &\iff x \in \alpha_{s}^{-1}({_{p^\ast}X} \cap X_{s}) & (1) \\
                &\iff \alpha_s(x) \in {_{p^\ast}X} \\
                &\iff (1,\alpha_s(x)) \in {_p\overline{X}}.
            \end{align*}
    
            \noindent(3) Let $(a,x) \simeq (b,y)$ where $(b,y) \in {_p\overline{X}}$. If $b \in S$, then $b \in S^p$. Thus, we may take $(t,y) = (b,y)$. Suppose that $b = 1$. Then $(1,y) \in {_{p^\ast}X}$ if and only if $y \in {_{p^\ast}X}$. Since $p^\ast \in S^{p^\ast}$ by Lemma \ref{lema:restriction}, $y = \alpha_{p^\ast}(y)$ by \eqref{P3} and $(1,\alpha_{p^\ast}(y)) \in {_p\overline{X}}$, it follows from (2) that $(p^\ast,y) \in {_p\overline{X}}$. Furthermore, we have $(p^\ast,y) \sim (1,\alpha_{p^\ast}(x)) = (1,y) \simeq (a,x)$ and
                $$ \beta_{p}(1,y) = (p,y) \overset{\eqref{r1}}{=} (pp^\ast,y) = \beta_p(p^\ast,y). $$
            Therefore, we may choose $(t,y) = (p^\ast,y)$. Note that (1) is a property of $(X;\alpha)$, (2) is a property of $(\overline{X};\beta)$ and the proof of (3) only uses that $(p^\ast,y) \sim (1,y) \simeq (a,x)$ implies $(p^\ast,y) \simeq (a,x)$. Therefore, the argument is valid for any transitive relation containing $\sim$.
        \end{proof}
    \end{lemma}
    
    In Proposition \ref{prop:construction-2} we proved that $\iota \colon (X;\alpha) \to (FX;\beta)$ is an embedding if and only if the equivalence relation $\simeq$ is such that $(s,x) \simeq (1,y)$ implies $x \in {_sX}$ and $\alpha_s(x) = y$. The following result shows that $\simeq^{(1)}$ always satisfies this property.
    
    \begin{prop} \label{prop:sim1}
        If $(s,x) \simeq^{(1)} (1,y)$, then $x \in {_sX}$ and $\alpha_s(x) = y$.
    
        \begin{proof}
            Suppose that $(s,x) \simeq^{(1)} (1,y)$ and let $(a_1,x_1), \dots, (a_n,x_n) \in \overline{X}$ be a minimal sequence relating $(s,x)$ to $(1,y)$. In this case, for each $i=1,\dots,n-1$, we have
                $$ (a_i,x_i) \simeq^{(0)} (a_{i+1},x_{i+1}) \quad\text{or}\quad (a_i,x_i) \simeq_\beta^{(0)} (a_{i+1},x_{i+1}). $$
            Since $(s,x) \neq (1,y)$, it must be $n \geq 2$. We proceed the proof with induction in $n$.\\
    
            If $n=2$, then $(s,x) \simeq^{(0)} (1,y)$ implies $x \in {_sX}$ and $\alpha_s(x) = y$. If $n=3$, then we have two possible sequences relating $(s,x)$ to $(1,y)$. The sequences are
                $$ (s,x) \simeq^{(0)} (a_2,x_2) \simeq^{(0)} (1,y) \quad\text{and}\quad (s,x) \simeq_\beta^{(0)} (a_2,x_2) \simeq^{(0)} (1,y). $$
            Since $\simeq^{(0)}$ is transitive, the left-most sequence can be reduced to $(s,x) \simeq^{(0)} (1,y)$ and the result follows from the case $n=2$. For the right-most sequence, $(a_2,x_2) \simeq^{(0)} (1,y)$ implies $x_2 \in {_{a_2}X}$ and $\alpha_{a_2}(x_2) = y$, while $(s,x) \simeq_\beta^{(0)} (a_2,x_2)$ implies that exists $p \in S$ and $(s',x), (a'_2,x_2) \in {_p\overline{X}}$ such that $(s',x) \simeq^{(0)} (a'_2,x_2)$, $(s,x) = \beta_p(s',x)$ and $(a_2,x_2) = \beta_p(a'_2,x_2)$. By Lemma \ref{lema:technical}(3), we can choose representatives with $s',a'_2 \in S^p$. Then,
                $$ (s,x) = (ps',x), \quad (a_2,x_2) = (pa'_2,x_2) \quad\text{and}\quad x_2 \in {_{a'_2}X}, \ x \in {_{s'}X}, \ \alpha_{s'}(x) = \alpha_{a'_2}(x_2). $$
            From \eqref{P1}, we obtain that $x_2 \in {_{a_2}X} \cap {_{a'_2}X} = {_{pa'_2}X} \cap {_{a'_2}X} = \alpha_{a'_2}^{-1}(X_{a'_2} \cap {_pX})$, and in this case
            \begin{align*}
                \alpha_p(\alpha_{a'_2}(x_2)) \overset{\eqref{P2}}{=} \alpha_{pa'_2}(x_2) = \alpha_{a_2}(x_2) = y.
            \end{align*}
            On the other hand, we have $\alpha_{s'}(x) = \alpha_{a'_2}(x_2) \in {_pX} \cap {X_{s'}}$. Therefore, it follows from \eqref{P1} that $x \in \alpha_{s'}^{-1}(X_{s'} \cap {_pX}) = {_{ps'}X} \cap {_{s'}X} = {_sX} \cap {_{s'}X}$ and, thus
                $$ y = \alpha_p(\alpha_{a'_2}(x_2)) = \alpha_p(\alpha_{s'}(x)) \overset{\eqref{P2}}{=} \alpha_{ps'}(x) = \alpha_s(x). $$
            This concludes the case $n=3$.
    
            Lastly, let $n \geq 4$ and assume that, if there is a sequence of length $m < n$ relating $(t,z)$ to $(1,y)$, then $z \in {_tX}$ and $\alpha_t(z) = y$. Since $n \geq 4$, we have a sequence of length $n-1 < n$ relating $(a_2,x_2)$ to $(1,y)$. Thus, $x_2 \in {_{a_2}X}$ and $\alpha_{x_2}(a_2) = y$ by the induction hypothesis. But then $(a_2,x_2) \simeq^{(0)} (1,y)$. Therefore, we can reduce the sequence relating $(s,x)$ to $(1,y)$ to one of the sequences from the case $n=3$, concluding that $x \in {_sX}$ and $\alpha_s(x) = y$.
        \end{proof}
    \end{prop}
    
    \begin{corollary} \label{coro:sim1}
        Suppose that $(a_1,x_1),\dots,(a_n,x_n) \in \overline{X}$ are such that
            $$ (a_i,x_i) \simeq^{(1)} (a_{i+1},x_{i+1}), \ \forall i=1,\dots,n-1. $$
        If $x_i \in {_{a_i}X}$, for some $i=1,\dots,n$, then $(a_1,x_1) \simeq^{(0)} (a_n,x_n)$.
    
        \begin{proof}
            Suppose that $x_i \in {_{a_i}X}$. Then
                $$ (a_1,x_1) \simeq^{(1)} (a_i,x_i) \simeq^{(0)} (1,\alpha_i(x_i)) \quad\text{and}\quad (1,\alpha_i(x_i)) \simeq^{(0)} (a_i,x_i) \simeq^{(1)} (a_n,x_n). $$
            Since the relation $\simeq^{(0)}$ is contained in $\simeq^{(1)}$, it follows from Proposition \ref{prop:sim1} that ${x_1} \in {_{a_1}X}$, ${x_n} \in {_{a_n}X}$ and $\alpha_{a_1}(x_1) = \alpha_{a_i}(x_i) = \alpha_{a_n}(x_n)$. Therefore, $(a_1,x_1) \simeq^{(0)} (1,\alpha_i(x_i)) \simeq^{(0)} (a_n,x_n)$.
        \end{proof}
    \end{corollary}
    
    In the next result, we show that $\simeq^{(1)}$ coincides with a relation usually associated with the construction of globalizations for partial actions. Consider the relation $\rightarrow$ on $\overline{X}$ defined by
        $$ (s,x) \rightarrow (t,y) \iff [\exists u \in S^t \colon s = ty,\ x \in {_uX} \text{ and } \alpha_u(x) = y], $$
    or equivalently, $(tu,x) \rightarrow (t,\alpha_u(x))$, whenever $x \in {_uX}$. Let $\simeq_\ast$ be defined by $\mathcal{E}(\simeq^{(0)} \cup \rightarrow)$.
    
    \begin{prop} \label{prop:sim1-2}
        The equivalence relations $\simeq^{(1)}$ and $\simeq_\ast$ coincide.
    
        \begin{proof}
            Since $\simeq^{(1)}$ is the smallest equivalence relation containing $\simeq^{(0)} \cup \simeq_\beta^{(0)}$ and $\simeq_\ast$ is the smallest equivalence relation containing $\simeq^{(0)} \cup \to$, it is enough to prove that $\simeq_\beta^{(0)}$ is contained in $\simeq_\ast$ and $\to$ is contained in $\simeq^{(1)}$.
    
            Suppose that $(a,x) \simeq_\beta^{(0)} (b,y)$. Then there are $p \in S$ and $(a',x), (b',y) \in {_p\overline{X}}$ such that $(a',x) \simeq^{(0)} (b',y)$, $(a,x) = \beta_p(a',x)$ and $(b,y) = \beta_p(b',y)$. By Lemma \ref{lema:technical}(3), we can choose representatives with $a',b' \in S^p$ and, in this case,
                $$ x \in {_{a'}X}, \ y \in {_{b'}X}, \ \alpha_{a'}(x) = \alpha_{b'}(y), \ (a,x) = (pa',x) \text{ and } (b,y) = (pb',y). $$
            Since $(a',x) \in {_p\overline{X}}$ and $x \in {_{a'}X}$, it follows from Lemma \ref{lema:technical}(2) that $(1,\alpha_{a'}(x)) \in {_p\overline{X}}$. Hence,
                $$ (a,x) = \beta_p(a',x) \simeq_\beta^{(0)} \beta_p(1,\alpha_{a'}(x)) = (p,\alpha_{a'}(x)). $$
            That is, $(a,x) \to (p,\alpha_{a'}(x))$. Analogously, we prove $(b,y) \to (p,\alpha_{b'}(y))$, concluding that
                $$ (a,x) \simeq_\ast (p,\alpha_{a'}(x)) = (p,\alpha_{b'}(y)) \simeq_\ast (b,y). $$
    
            Conversely, suppose that $(a,x) \to (b,y)$. Then there is $u \in S^a$ such that $a = bu$, $x \in {_uX}$ and $\alpha_u(x) = y$. From \eqref{P4}, we obtain that $x \in {_uX} \subseteq {_{u^\ast}X}$. Hence, $(u,x) \in \overline{X}$ and $(u,x) \sim (1,\alpha_u(x)) = (1,y)$. Lastly, it follows from $(bu,x) = (a,x) \in \overline{X}$ and $(b,y) \in \overline{X}$ that
                $$ (u,x), (1,y) \in {_b\overline{X}}, \ (a,x) = \beta_b(u,x), \ (b,y) = \beta_b(1,y) \text{ and } (u,x) \simeq^{(0)} (1,y).  $$
            Therefore, $(a,x) \simeq_\beta^{(0)} (b,y)$, concluding the proof.
        \end{proof}
    \end{prop}

    In the remainder of this section, we show that every partial action in which disruptive elements act globally is globalizable. More precisely, we say that \emph{disruptive elements act globally} on a partial action $(X;\alpha)$ if
        $$ {_dX} = {_{d^\ast}X}, \quad \forall d \in \mathfrak{D}(S) = \{ s \in S \colon \exists t \in S^s \text{ such that } {^sS} \subsetneq {^{st}S} \}. $$
    Note that this condition is automatically satisfied by every partial action of $(S,\ast)$ whenever $\mathfrak{D}(S) \subseteq S^\ast$. In particular, if $S$ is a semigroupoid, then $\mathfrak{D}(S) = \emptyset \subseteq S^\ast$ always holds.

    \begin{lemma} \label{lema:suficient}
        Suppose that disruptive elements act globally on $(X;\alpha)$. If $(s,x) \simeq^{(1)} (t,y)$ and $s,t \in S^p$, then $(s,x) \in {_p\overline{X}}$ if and only if $(t,y) \in {_p\overline{X}}$.
    
        \begin{proof}
            First, we note that if $(s,x) \simeq^{(0)} (t,y)$ and $s,t \in S^p$, then $(s,x) \in {_p\overline{X}}$ if and only if $(t,y) \in {_p\overline{X}}$. In fact, in this case we have $x \in {_sX}$, $y \in {_tX}$ and $\alpha_s(x) = \alpha_t(y)$. Therefore, it follows from Lemma \ref{lema:technical}(2) that
                $$ (s,x) \in {_p\overline{X}} \iff (1,\alpha_s(x)) = (1,\alpha_t(y)) \in {_p\overline{X}} \iff (t,y) \in {_p\overline{X}}. $$
    
            Now, suppose that $(s,x) \simeq^{(1)} (t,y)$. By Proposition \ref{prop:sim1-2}, the equivalence relations $\simeq^{(1)}$ and $\simeq_\ast$ coincide. Therefore, there is a minimal sequence $(a_1,x_1),\dots,(a_n,x_n) \in \overline{X}$ such that $(s,x) = (a_1,x_1)$, $(a_n,x_n) = (t,y)$ and, for each $i=1,\dots,n-1$,
                $$ (a_i,x_i) \simeq^{(0)} (a_{i+1},x_{i+1}), \quad (a_i,x_i) \to (a_{i+1},x_{i+1}) \quad\text{or}\quad (a_{i+1},x_{i+1}) \to (a_i,x_i). $$
            Since the sequence is minimal, we have $a_i \neq 1$. Furthermore, if $x_i \in {_{a_i}X}$, for some $i=1,\dots,n$, then $(s,x) \simeq^{(0)} (t,y)$ by Corollary \ref{coro:sim1} and, in this case, the result follows from the previous observation. Therefore, we can assume that none of the relations in the sequence is given by $\simeq^{(0)}$. We proceed the proof with induction in $n \geq 1$.
    
            If $n=1$, then $(s,x) = (t,y)$ and the result is trivial. If $n=2$, then $(s,x) \to (t,y)$ or $(t,y) \to (s,x)$. Assume first that $(s,x) \to (t,y)$. Then exists $u \in S$ such that $s = tu$, $x \in {_uX}$ and $\alpha_u(x) = y$. From \eqref{P4}, we also have $x \in {_{u^\ast}X}$. Since $pt$ is defined, we conclude that
            \begin{align*}
                (s,x) = (tu,x) \in {_p\overline{X}} &\iff x \in {_{(p(tu))^\ast}X} \\
                &\iff x \in {_{((pt)u)^\ast}X} & \eqref{sc2} \\
                &\iff (u,x) \in {_{pt}\overline{X}} \\
                &\iff (1,y) = (1,\alpha_u(x)) \in {_{pt}\overline{X}} & \ref{lema:technical}(2) \\
                &\iff y \in {_{(pt)^\ast}X} \\
                &\iff (t,y) \in {_p\overline{X}}.
            \end{align*}
            On the other hand, assume that $(t,y) \to (s,x)$. Then exists $v \in S$ such that $t = sv$, $y \in {_vX}$ and $\alpha_v(y) = x$. Again, we have $y \in {_{v^\ast}X}$ and $ps$ is defined. Hence,
                $$ (t,y) = (sv,x) \in {_p\overline{X}} \overset{\eqref{sc2}}{\iff} (v,x) \in {_{ps}\overline{X}} \overset{\ref{lema:technical}(2)}{\iff} (1,x) \in {_{ps}\overline{X}} \iff (s,x) \in {_p\overline{X}}. $$
            This concludes the case $n=2$.
    
            Let $n \geq 3$ and suppose that, if $(s',x') \simeq^{(1)} (t',y')$ are related by a sequence of length $m < n$ and $s',t' \in S^p$, then $(s',x') \in {_p\overline{X}}$ if and only if $(t',y') \in {_p\overline{X}}$. Since $n \geq 3$, the sequence relating $(s,x)$ and $(t,y)$ begins with $(s,x) \to (a_2,x_2)$ or $(a_2,x_2) \to (s,x)$. Suppose that $(a_2,x_2) \to (s,x)$. Then there is $v \in S$ such that $a_2 = sv$, $x_2 \in {_vX}$ and $\alpha_v(x_2) = x$. Since $sv$ and $ps$ are defined, it follows from \eqref{sc2} that $pa_2 = p(sv)$ is defined. In this case, $(s,x) \simeq^{(1)} (a_2,x_2)$ and $(a_2,x_2) \simeq^{(1)} (t,y)$ are related by sequences of length $2$ and $n-1$, respectively, and $s,a_2,t \in S^p$. Therefore, it follows from the case $n=2$ and the induction hypothesis that
                $$ (s,x) \in {_p\overline{X}} \iff (a_2,x_2) \in {_p\overline{X}} \iff (t,y) \in {_p\overline{X}}. $$
            Lastly, suppose that $(s,x) \to (a_2,x_2)$. Then there is $u \in S$ such that $s = a_2u$, $x \in {_uX}$ and $\alpha_u(x) = x_2$. We split the proof in two cases:
            \begin{itemize}
                \item If $a_2 \notin \mathfrak{D}(S)$. Then $s \in S^p$ implies $p \in {^sS} = {^{a_2u}S} = {^{a_2}S}$. Thus, $a_2 \in S^p$ and, as in the previous case, it follows from $n=2$ and the induction hypothesis that
                    $$ (s,x) \in {_p\overline{X}} \iff (a_2,x_2) \in {_p\overline{X}} \iff (t,y) \in {_p\overline{X}}. $$
    
                \item If $a_2 \in \mathfrak{D}(S)$. Since disruptive elements act globally, we have $x_2 \in {_{a_2^\ast}X} = {_{a_2}X}$. Therefore, $(s,x) \simeq^{(1)} (a_2,x_2) \simeq^{(1)} (t,y)$ where $x_2 \in {_{a_2}X}$. From Corollary \ref{coro:sim1}, we obtain $(s,x) \simeq^{(0)} (t,y)$. Hence, $(s,x) \in {_p\overline{X}}$ if and only if $(t,y) \in {_p\overline{X}}$.
            \end{itemize}
            This concludes the proof.
        \end{proof}
    \end{lemma}
    
    \begin{prop} \label{prop:suficient}
        Suppose that disruptive elements act globally on $(X;\alpha)$. Then, the equivalence relations $\simeq$ and $\simeq^{(1)}$ coincide.
    
        \begin{proof}
            It is enough to prove that $\simeq_\beta^{(1)}$ is contained in $\simeq^{(1)}$. Suppose that $(s,x) \simeq_\beta^{(1)} (t,y)$. Then there are $p \in S$ and $(a,x), (b,y) \in {_p\overline{X}}$ such that $(a,x) \simeq^{(1)} (b,y)$, $(s,x) = \beta_p(a,x)$ and $(t,y) = \beta_p(t,y)$. By Lemma \ref{lema:technical}(3), we can choose representatives with $a,b \in S^p$. In this case, it follows from Proposition \ref{prop:sim1-2} that there are $(a_1,x_1),\dots,(a_n,x_n) \in \overline{X}$ such that $(a,x) = (a_1,x_1)$, $(a_n,x_n) = (b,y)$ and, for each $i=1,\dots,n-1$,
                $$ (a_i,x_i) \simeq^{(0)} (a_{i+1},x_{i+1}), \quad (a_i,x_i) \to (a_{i+1},x_{i+1}) \quad\text{or}\quad (a_{i+1},x_{i+1}) \to (a_i,x_i). $$
            Furthermore, if $x_i \in {_{a_i}X}$, for some $i=1,\dots,n$, then $(a,x) \simeq^{(0)} (b,y)$ by Corollary \ref{coro:sim1}. In this case, we have $(s,x) = \beta_p(a,x) \simeq_\beta^{(0)} \beta_p(b,y) = (t,y)$, and thus, $(s,x) \simeq^{(1)} (b,y)$. Therefore, we can assume that none of the relations is given by $\simeq^{(0)}$. We proceed the proof with induction in $n \geq 1$.
    
            If $n=1$, then $(a,x) = (b,y)$. Thus, $(s,x) = (t,y)$ and there is nothing to prove. If $n=2$, then $(a,x) \to (b,y)$ or $(b,y) \to (a,x)$. Suppose that $(a,x) \to (b,y)$. Then there is $u \in S$ such that $a = bu$, $x \in {_uX}$ and $\alpha_u(x) = y$. Since $b \in S^p$, we obtain
            \begin{align*}
                (s,x) = \beta_p(a,x) &= (pa,x) = (p(bu),x) \\&\overset{\eqref{sc2}}{=} ((pb)u,x) \to (pb,\alpha_u(x)) = (pb,y) = \beta_p(b,y) = (t,y).
            \end{align*}
            Hence, $(s,x) \simeq^{(1)} (b,y)$. Since $a \in S^p$, the case $(b,y) \to (a,x)$ is analogous. This concludes the case $n=2$.
    
            Let $n \geq 3$ and suppose that if $(a',x') \simeq^{(1)} (b',y')$ are related by a sequence of length $m < n$ and $(a',x'), (b',y') \in {_p\overline{X}}$, then $\beta_p(a',x') \simeq^{(1)} \beta_p(b',y')$. Since $n \geq 3$, we have $(a,x) \to (a_2,x_2)$ or $(x_2,a_2) \to (a,x)$. Suppose that $(x_2,a_2) \to (a,x)$. Then
                $$ \exists v \in S \colon a_2 = av, \ x_2 \in {_vX} \text{ and } \alpha_v(x_2) = x. $$
            Since $pa$ and $av$ are defined, it follows from \eqref{sc2} that $pa_2 = p(av)$ is defined. By Lemma \ref{lema:suficient}, $(a,x) \in {_p\overline{X}}$ implies $(a_2,x_2) \in {_p\overline{X}}$. In this case, $(a,x) \simeq^{(1)} (a_2,x_2)$ and $(a_2,x_2) \simeq^{(1)} (b,y)$ are related by sequences of length $2$ and $n-1$, respectively, and $(a,x), (a_2,x_2), (b,y) \in {_p\overline{X}}$. Therefore, it follows from the case $n=2$ and the induction hypothesis that
                $$ (s,x) = \beta_p(a,x) \simeq^{(1)} \beta_p(a_2,x_2) \simeq^{(1)} \beta_p(b,y) = (t,y). $$
            That is, $(s,x) \simeq^{(1)} (t,y)$. On the other hand, suppose that $(a,x) \to (a_2,x_2)$. Then there is $u \in S$ such that $a = a_2u$, $x \in {_uX}$ and $\alpha_u(x) = x_2$. We split the proof in two cases:
            \begin{itemize}
                \item If $a_2 \notin \mathfrak{D}(S)$. Then $a \in S^p$ implies $p \in {^aS} = {^{a_2u}S} = {^{a_2}S}$. But then $a_2 \in S^p$. From Lemma \ref{lema:suficient} and $(a,x) \in {_p\overline{X}}$, we obtain $(a_2,x_2) \in {_p\overline{X}}$. Therefore, it follows from the case $n=2$ and the induction hypothesis that
                    $$ (s,x) = \beta_p(a,x) \simeq^{(1)} \beta_p(a_2,x_2) \simeq^{(1)} \beta_p(b,y) = (t,y). $$
    
                \item If $a_2 \in \mathfrak{D}(S)$. Since disruptive elements act globally, we have $x_2 \in {_{a_2^\ast}X} = {_{a_2}X}$. Therefore, $(a,x) \simeq^{(1)} (a_2,x_2) \simeq^{(1)} (b,y)$ where $x_2 \in {_{a_2}X}$. From Corollary \ref{coro:sim1}, we obtain $(a,x) \simeq^{(0)} (b,y)$. Hence,
                    $$ (s,x) = \beta_p(a,x) \simeq_\beta^{(0)} \beta_p(b,y) = (t,y) \implies (s,x) \simeq^{(1)} (t,y). $$
            \end{itemize}
            This concludes the proof.
        \end{proof}
    \end{prop}
    
    \begin{theorem} \label{teo:suficient}
        Let $(X;\alpha)$ be a partial action of a restriction semiconstellation $(S,\ast)$. If disruptive elements act globally in $(X;\alpha)$, then $(FX;\beta)$ is a global action of $(S,\ast)$ and a globalization for $(X;\alpha)$.
    
        \begin{proof}
            Since disruptive elements act globally in $(X;\alpha)$ it follows from Proposition \ref{prop:suficient} that the equivalence relation $\simeq$ coincides with $\simeq^{(1)}$. By Proposition \ref{prop:sim1}, we obtain that $(s,x) \simeq (1,y)$ implies $x \in {_sX}$ and $\alpha_s(x) = y$. Therefore, we conclude from Proposition \ref{prop:construction-2} that $\iota \colon (X;\alpha) \to (FX;\beta)$ is an embedding.
            
            Now, we prove that $(FX;\beta)$ is a global action of $(S,\ast$). By Proposition \ref{prop:construction}, it remains to prove that, if $st$ is defined, then $\beta_t^{-1}(FX_t \cap {_sFX}) \subseteq {_{st}FX}$. In fact, we have
                $$ a \otimes x \in \beta_t^{-1}(FX_t \cap {_sFX}) \iff (a,x) \simeq (b,y) \in {_t\overline{X}} \colon \beta_p(b,y) \simeq (c,z) \in {_s\overline{X}}. $$
            By Lemma \ref{lema:technical}(3), we can choose representatives with $b \in S^t$ and $c \in S^s$. In this case, we have $(tb,y) = \beta_t(b,y) \simeq (c,z) \in {_s\overline{X}}$. Since $st$ and $tb$ are defined, it follows from \eqref{sc2} that $s(tb)$ is defined. Therefore, by Lemma \ref{lema:suficient}, $(c,z) \in {_s\overline{X}}$ implies that $(tb,y) \in {_s\overline{X}}$. But
                $$ (tb,y) \in {_s\overline{X}} \iff y \in {_{(s(tb))^\ast}X} \overset{\eqref{sc2}}{=} {_{((st)b)^\ast}X} \iff (b,y) \in {_{st}\overline{X}}. $$
            Hence, $(a,x) \simeq (b,y) \in {_{st}\overline{X}}$, concluding that $a \otimes x \in {_{st}FX}$.
        \end{proof}
    \end{theorem}
    
    Recall that a semiconstellation $S$ is a semigroupoid if and only if $\mathfrak{D}(S) = \emptyset$. The following result is straightforward from Theorem \ref{teo:suficient}.
    
    \begin{corollary} \label{coro:suficient}
        If $(S,\ast)$ is a restriction semiconstellation such that $\mathfrak{D}(S) \subseteq S^\ast$, then every partial action of $(S,\ast)$ is globalizable . In particular, every partial action of a restriction semigroupoid is globalizable.
    \end{corollary}

    Corollary \ref{coro:suficient} generalizes the globalization theorems for partial actions of inverse semigroupoids \cite[Theorem 3.16]{demeneghi2025}, restriction semigroups \cite[Theorem 6.7]{gould2009partial}, and categories \cite[Theorem 4]{nystedt2017}. Further particular cases can be found in the references cited above.
    
    \begin{obs}
        In Section 6 of \cite{gould2009partial}, in the context of the globalization theorem for partial actions of restriction semigroups on sets, the relation $\simeq$ is defined as the smallest equivalence relation containing the relation
            $$ (s,x) \sim (t,y) \iff \begin{matrix}
                \exists u \in S \colon s = tu, x \in {_uX}, \alpha_u(x) = y \text{ and} \\
                \forall e \in S^\ast, [x \in {_{(e(tu))^\ast}X} \iff \alpha_u(x) \in {_{(et)^\ast}X}].
            \end{matrix} $$
        The first line on the right side is precisely the definition of the relation $\to$, while the second is equivalent to the conclusion of Lemma \ref{lema:suficient}. Since every semigroup satisfies $\mathfrak{D}(S)=\emptyset$, Lemma \ref{lema:suficient} implies that the second condition is automatically satisfied. Hence, it is superfluous in this setting.
    \end{obs}
    
    \subsection{Reduced case and examples} We prove that, if $(S,\ast)$ is reduced, then $(FX;\beta)$ is a global action, characterizing when $(X;\alpha)$ is globalizable. We also provide examples justifying the negative response to the globalization problem.\\
    
    In Theorem \ref{teo:suficient}, we proved that if disruptive elements act globally in $(X;\alpha)$, then $(FX;\beta)$ is a global action. More precisely, the assumption on the action of disruptive elements allows us to prove that $\simeq = \simeq^{(1)}$. In this case, Lemma \ref{lema:suficient} implies that, whenever $st$ is defined and $\beta_t(b,x) \simeq (c,z) \in {_s\overline{X}}$, we have $\beta_t(b,x) \in {_s\overline{X}}$. Consequently, $\beta_t^{-1}(FX_t \cap {_sX}) \subseteq {_{st}FX}$. By Proposition \ref{prop:construction}, this was the final condition needed to conclude that $(FX;\beta)$ is a global action. We now show that when $(S,\ast)$ is a reduced semiconstellation, the same conclusion holds without any assumption on the action of the disruptive elements.
    
    \begin{theorem} \label{teo:reduced}
        Let $(X;\alpha)$ be a partial action of a reduced semiconstellation $(S,\ast)$. Then the pair $(\iota,(FX;\beta))$ is a reflector for $(X;\alpha)$ in the category of global actions $\mathcal{A}(S,\ast)$.
    
        \begin{proof}
            By Theorem \ref{teo:construction}, the pair $(\iota,(FX;\beta))$ is a reflector for $(X;\alpha)$ in the category $\mathcal{A}_a(S,\ast)$, which contains the category $\mathcal{A}(S,\ast)$. Therefore, it is enough to show that $(FX;\beta) \in \mathcal{A}(S,\ast)$. Since $(S,\ast)$ is reduced, if $st$ is defined, then $(st)^\ast = t^\ast$ by Proposition \ref{prop:char-reduced}(5). Therefore,
                $$ {_tFX} \overset{\eqref{P5}}{=} {_{t^\ast}FX} \overset{\ref{prop:char-reduced}(5)}{=} {_{(st)^\ast}FX} \overset{\eqref{P5}}{=} {_{st}FX} \overset{\ref{prop:construction}(3)}{\subseteq} \beta_t^{-1}(FX_t \cap {_sFX}) \subseteq {_tFX}. $$
        That is, $\beta_t^{-1}(FX_t \cap {_sFX}) = {_{st}FX} = {_{st}FX} \cap {_tFX}$. From Proposition \ref{prop:construction} and the last equality, we conclude that $(FX;\beta)$ is a global action of $(S,\ast)$.
        \end{proof}
    \end{theorem}

    The short proof of Theorem \ref{teo:reduced} hides the fact that, when $(S,\ast)$ is reduced, the sets ${_s\overline{X}}$ have a simpler description than in the setting of Lemma \ref{lema:suficient}. Indeed, if $(s,x)\in\overline{X}$, with $s\in S$, then
    $$ (s,x) \in {_p\overline{X}} \iff s \in S^p \text{ and } x \in {_{(ps)^\ast}X} \overset{\ref{prop:char-reduced}(5)}{\iff} s \in S^p \text{ and } x \in {_{s^\ast}X} \iff s \in S^p. $$
    
    Thus, the condition $s\in S^p$ alone is sufficient to ensure that $(s,x)\in {_p\overline{X}}$. In contrast, in Lemma \ref{lema:suficient}, this conclusion requires the additional assumption that disruptive elements act globally, together with the existence of some $(t,y)\in {_p\overline{X}}$ such that $(s,x) \simeq^{(1)} (t,y)$ and $t \in S^p$.\\
    
    The next result is an immediate consequence of Proposition \ref{prop:construction-2} and Theorem \ref{teo:reduced}.
    
    \begin{corollary} \label{coro:reduced}
        Let $(X;\alpha)$ be a partial action of a reduced semiconstellation $(S,\ast)$. Then, the following assertions are equivalent:
        \begin{enumerate}
            \item $(X;\alpha)$ is globalizable.
            \item The $S$-morphism $\iota \colon (X;\alpha) \to (FX;\beta)$ is an embedding.
            \item If $(s,x) \simeq (1,y)$, then $x \in {_sX}$ and $\alpha_s(x) = y$.
        \end{enumerate}
    \end{corollary}
    
    \begin{obs}
        In this section, we have fixed $(X;\alpha)$ a partial action. It is interesting to note that many results hold for more general $S$-algebras, especially if $(S,\ast)$ is reduced.
        \begin{enumerate}
            \item The construction of $(FX;\beta)$ can be done from any $S$-algebra. In this case, $(FX;\beta)$ always satisfy \eqref{P5} and every $S$-morphism $(X;\alpha) \to (Y;\gamma)$, where $(Y;\gamma)$ is an almost global action, factors uniquely through the function $\iota \colon X \to FX$.\\
            \item If $(X;\alpha)$ is a $S$-algebra such that the function $\alpha \colon (S,\ast) \to PT(X)$ is order preserving, then the $S$-algebra $(FX;\beta)$ is an almost global action. In this case, the pair $(\iota,(FX;\beta))$ is a reflector for $(X;\alpha)$ in $\mathcal{A}_a(S,\ast)$ if and only if $\iota \colon (X;\alpha) \to (FX;\beta)$ is a $S$-morphism if and only if the $S$-algebra $(X;\alpha)$ satisfies \eqref{P4}.\\
            \item Suppose that $(S,\ast)$ is reduced. Then, by Proposition \ref{prop:char-reduced}(3), we have $s \leq t$ if and only if $s = t$. In this case, any function $\alpha \colon (S,\ast) \to PT(X)$ is order preserving. Therefore, $(FX;\beta)$ is a global action for every $S$-algebra $(S,\ast)$ satisfying \eqref{P3}, and is a reflector for $(X;\alpha)$ in $\mathcal{A}(S,\ast)$ if and only if the $S$-algebra $(X;\alpha)$ also satisfies \eqref{P4}.
        \end{enumerate}
    \end{obs}
    
    We end this section presenting examples to discuss the globalization problem of partial actions of restriction semiconstellations. In the following examples, we use the reduced semiconstellation $(T_2^1,\ast)$ from Example \ref{exe:restriction}. Recall that
        $$ T_2^1 = \{ s_1^1, s_2^1, s_3^1, s^\ast, d_1, d_1^\ast, d_2, d_2^\ast \}. $$
    For simplicity, we denote $s_i$ instead of $s_i^1$. Hence, composition in $T_2^1$ is given by
        $$ d_1s_1 = s_2, \quad d_2s_2 = s_3, \quad d_id_i^\ast = d_i, \quad d_i^\ast s_i = s_i, \ i=1,2, \quad\text{and}\quad s_js^\ast = s_j, \  j=1,2,3. $$
    The reduced restriction structure is uniquely defined by $d_j \mapsto d_j^\ast$ and $s_i \mapsto s^\ast$, and the set of disruptive elements of $T_2^1$ is $\mathfrak{D}(T_2^1) = \{d_1\}$.\\
    
    The first example shows a partial action $(X;\alpha)$ such that $\iota$ is an embedding, but the relation $\simeq$ is strictly larger than $\simeq^{(1)}$. In particular, it follows from Proposition \ref{prop:suficient} that disruptive elements do not act globally in $(X;\alpha)$. Therefore, this condition is sufficient, but not necessary, to the existence of a globalization.
    
    \begin{exe} \label{exe:reduced-1}
        Let $X = \{x,y\}$ and consider the following data:
        \begin{gather*}
            {_{d_1}X} = {_{d_2}X} = {_{d_2^\ast}X} = {_{s_2}X} = {_{s_3}X} = \emptyset, \quad {_{d_1^\ast}X} = \{x\}, \quad {_{s_1}X} = {_{s^\ast}X} = \{x,y\}, \\
            \alpha_{s_1}(x) = \alpha_{s_1}(y) = x, \quad \alpha_{d_1^\ast} = id_{\{x\}} \quad\text{and}\quad \alpha_{s^\ast} = id_X.
        \end{gather*}
        Then $(X;\alpha)$ is a partial action of $(T_2^1,\ast)$. The set $\overline{X}$ and the domains ${_t\overline{X}}$ are given by
        \begin{align*}
            \overline{X} = \left\{\begin{matrix}
                (1,x), & (s^\ast,x), & (s_1,x), & (s_2,x), & (s_3,x), & (d_1^\ast,x), & (d_1,x), \\
                (1,y), & (s^\ast,y), & (s_1,y), & (s_2,y), & (s_3,y) &  &  \\
            \end{matrix}\right\}
        \end{align*}
        and
        \begin{align*}
            {_{s_1}\overline{X}} = {_{s_2}\overline{X}} = {_{s_3}\overline{X}} = {_{s^\ast}\overline{X}} &= \{ (1,x), (1,y), (s^\ast,x), (s^\ast,y) \}, \\
            {_{d_1}\overline{X}} = {_{d_1^\ast}\overline{X}} &= \{ (1,x), (d_1^\ast,x), (s_1,x), (s_1,y) \}, \\
            {_{d_2}\overline{X}} = {_{d_2^\ast}\overline{X}} &= \{ (s_2,x), (s_2,y) \}.
        \end{align*}
        In the following \textit{graph of relations} in $\overline{X}$, we draw an edge between $(a,z)$ and $(a',z')$ labeled with $n$ if and only if $(a,z) \simeq^{(n)} (a',z')$ where $n$ is the smallest positive integer possible.
        \begin{center}
            \begin{tikzpicture}[yscale=1.25]
                \tikzstyle{every path} = [draw];
                
                \node (1x) at (0,0) {$(1,x)$};
                \node (s*x) at (2,0) {$(s^\ast,x)$};
                \node (s1x) at (4,0) {$(s_1,x)$};
                \node (d1*x) at (6,0) {$(d_1^\ast,x)$};
                \node (s1y) at (8,0) {$(s_1,y)$};
                \node (s3x) at (10,0) {$(s_3,x)$};
                \node (1y) at (0,-1) {$(1,y)$};
                \node (s*y) at (2,-1) {$(s^\ast,y)$};
                \node (s2x) at (4,-1) {$(s_2,x)$};
                \node (d1x) at (6,-1) {$(d_1,x)$};
                \node (s2y) at (8,-1) {$(s_2,y)$};
                \node (s3y) at (10,-1) {$(s_3,y)$};
    
                \path (1x) to node[above]{0} (s*x) to node[above]{0} (s1x) to node[above]{0} (d1*x) to node[above]{0} (s1y);
                \path (1y) to node[above]{0} (s*y);
                \path (s2x) to node[above]{1} (d1x) to node[above]{1} (s2y);
                \path (s3x) to node[left]{2} (s3y);
            \end{tikzpicture}
        \end{center}
        Note that $(s,z)$ is connected to $(1,x)$ or $(1,y)$ if and only if the sequence of relations is given by $\simeq^{(0)}$. Therefore, the condition in Corollary \ref{coro:reduced}(3) holds, concluding that the partial action $(X;\alpha)$ is globalizable. Since the highest label on the graph is $2$, the relation $\simeq$ coincides with $\simeq^{(2)}$ but not with $\simeq^{(1)}$.
    \end{exe}
    
    The second example shows a partial action $(X;\alpha)$ such that $\iota$ is full but not injective. Therefore, $\iota$ is not an embedding and $(X;\alpha)$ is not globalizable.
    
    \begin{exe} \label{exe:reduced-2}
        Let $X = \{x,y\}$ and consider the following data:
        \begin{gather*}
            {_{d_1}X} = {_{d_2}X} = {_{d_2^\ast}X} = {_{s_2}X} = \emptyset, \quad {_{d_1^\ast}X} = \{x\}, \quad {_{s_1}X} = {_{s_3}X} = {_{s^\ast}X} = \{x,y\}, \\
            \alpha_{s_1}(x) = \alpha_{s_1}(y) = x, \quad \alpha_{s_3}(x) = x, \quad \alpha_{s_3}(y) = y, \quad \alpha_{d_1^\ast} = id_{\{x\}} \quad\text{and}\quad \alpha_{s^\ast} = id_X.
        \end{gather*}
        Then $(X;\alpha)$ is a partial action of $(T_2^1,\ast)$. The sets $\overline{X}$ and ${_t\overline{X}}$ coincide with the ones from Example \ref{exe:reduced-1}. In this case, the graph of relations in $\overline{X}$ is
        \begin{center}
            \begin{tikzpicture}[yscale=1.25]
                \tikzstyle{every path} = [draw];
                
                \node (1x) at (0,0) {$(1,x)$};
                \node (s*x) at (2,0) {$(s^\ast,x)$};
                \node (s3x) at (4,0) {$(s_3,x)$};
                \node (s1x) at (6,0) {$(s_1,x)$};
                \node (d1*x) at (8,0) {$(d_1^\ast,x)$};
                \node (s1y) at (10,0) {$(s_1,y)$};
                \node (1y) at (0,-1) {$(1,y)$};
                \node (s*y) at (2,-1) {$(s^\ast,y)$};
                \node (s3y) at (4,-1) {$(s_3,y)$};
                \node (s2x) at (6,-1) {$(s_2,x)$};
                \node (d1x) at (8,-1) {$(d_1,x)$};
                \node (s2y) at (10,-1) {$(s_2,y)$};
    
                \path (1x) to node[above]{0} (s*x) to node[above]{0} (s3x) to node[above]{0} (s1x) to node[above]{0} (d1*x) to node[above]{0} (s1y);
                \path (1y) to node[above]{0} (s*y) to node[above]{0} (s3y);
                \path (s2x) to node[above]{1} (d1x) to node[above]{1} (s2y);
                \path (s3x) to node[left]{2} (s3y);
            \end{tikzpicture}
        \end{center}
        Note that $(1,x)$ is connected to $(1,y)$. Hence, the $S$-morphism $\iota \colon (X;\alpha) \to (FX;\beta)$ is not injective. By Corollary \ref{coro:reduced}, we conclude that $(X;\alpha)$ is not globalizable. On the other hand, if $\beta_t(1,z) = (t,z)$ is connected to $(1,x)$ or $(1,y)$, then $z \in {_tX}$. Therefore, the $S$-morphism $\iota \colon (X;\alpha) \to (FX;\beta)$ is full.
    \end{exe}
    
    The third and final example presents a partial action $(X;\alpha)$ for which $\iota$ is injective but not full. Consequently, $\iota$ is not an embedding and $(X;\alpha)$ cannot be globalizable.
    
    \begin{exe} \label{exe:unipotet-3}
        Let $X = \{x,y\}$ and consider the following data:
        \begin{gather*}
            {_{d_1}X} = {_{d_2}X} = {_{d_2^\ast}X} = {_{s_2}X} = \emptyset, \quad {_{s_1}X} = {_{s^\ast}X} = \{x,y\}, \quad {_{s_3}X} = {_{d_1^\ast}X} = \{x\}, \\
            \alpha_{s_1}(x) = \alpha_{s_1}(y) = \alpha_{s_3}(x) = x, \quad \alpha_{d_1^\ast} = id_{\{x\}} \quad\text{and}\quad \alpha_{s^\ast} = id_X.
        \end{gather*}
        Then $(X;\alpha)$ is a partial action of $(T_2^1,\ast)$. The sets $\overline{X}$ and ${_t\overline{X}}$ are the same as in Example \ref{exe:reduced-1}. In this case, the graph of relations in $\overline{X}$ is
        \begin{center}
            \begin{tikzpicture}[yscale=1.25]
                \tikzstyle{every path} = [draw];
                
                \node (1x) at (0,0) {$(1,x)$};
                \node (s*x) at (2,0) {$(s^\ast,x)$};
                \node (s3x) at (4,0) {$(s_3,x)$};
                \node (s1x) at (6,0) {$(s_1,x)$};
                \node (d1*x) at (8,0) {$(d_1^\ast,x)$};
                \node (s1y) at (10,0) {$(s_1,y)$};
                \node (1y) at (0,-1) {$(1,y)$};
                \node (s*y) at (2,-1) {$(s^\ast,y)$};
                \node (s3y) at (4,-1) {$(s_3,y)$};
                \node (s2x) at (6,-1) {$(s_2,x)$};
                \node (d1x) at (8,-1) {$(d_1,x)$};
                \node (s2y) at (10,-1) {$(s_2,y)$};
    
                \path (1x) to node[above]{0} (s*x) to node[above]{0} (s3x) to node[above]{0} (s1x) to node[above]{0} (d1*x) to node[above]{0} (s1y);
                \path (1y) to node[above]{0} (s*y);
                \path (s2x) to node[above]{1} (d1x) to node[above]{1} (s2y);
                \path (s3x) to node[left]{2} (s3y);
            \end{tikzpicture}
        \end{center}
        Since $(1,x)$ and $(1,y)$ are not connected, the $S$-morphism $\iota \colon (X;\alpha) \to (FX;\beta)$ is injective. However, $\beta_{s_3}(1,y) = (s_3,y)$ is connected to $(1,x)$ but $y \notin {_{s_3}X}$. Therefore, the $S$-morphism $\iota \colon (X;\alpha) \to (FX;\beta)$ is not full. From Corollary \ref{coro:reduced}, we conclude that the partial action $(X;\alpha)$ is not globalizable.
    \end{exe}
    
    \begin{obs}
        On the previous examples:
        \begin{enumerate}
            \item We create a sequence $(s_2,x) \simeq^{(1)} (d_1,x) \simeq^{(1)} (s_2,y)$ where $d_1$ is a disruptive element. This is the situation where, in Lemma \ref{lema:suficient} and Proposition \ref{prop:suficient}, we use the hypothesis ${_dX} = {_{d^\ast}X}$. Without this additional hypothesis, we obtain $(s_3,x) \simeq^{(2)} (s_3,y)$. Then, choosing how $s_3$ acts in $x$ and $y$, we manipulate when $\iota$ is injective or full.\\
    
            \item The fact that $\simeq$ equals to $\simeq^{(2)}$ in all examples is not a coincidence. Replacing $T_2^1$ with $T_n^1$, $n \geq 2$, and extending the actions $(X;\alpha)$ by defining ${_tX} = \emptyset$, for all $t \in T_n^1 \setminus T_2^1$, we conclude that $\simeq$ equals to $\simeq^{(n)}$. Replacing $T_2^1$ with $T_\infty^1$ makes that the relation $\simeq$ does not coincide with any of the relations $\simeq^{(n)}$.
        \end{enumerate}
    \end{obs}
    
    \section{Adequate relations on semiconstellations} \label{sec:5}

    In the second half of this paper, we discuss an analogue to the globalization problem for actions of semiconstellations that are not assumed to be endowed with a restriction structure. Our approach is mainly inspired by the globalization theorem for partial actions of semigroups \cite[Remark 3.8]{kudryavtseva2023} and semigroupoids \cite[Theorem 3.7]{haag2026a}.
    
    The method in \cite{kudryavtseva2023} consists of constructing globalizations for \textit{firm} partial semigroup actions and showing that every partial monoid action is firm. Partial actions of a semigroup $S$ are then extended to partial monoid actions of $S^1$, where $S^1 = S \cup \{1\}$ is the monoid obtained by adjoining an identity element to $S$. The globalization for a partial semigroup action is subsequently recovered from the corresponding globalization for the extended partial monoid action. On the other hand, not every semigroupoid can be extended to a semigroupoid with identities. Nevertheless, \cite{haag2026a} constructs globalizations for partial actions of a semigroupoid $S$ by using the set $S^1$ together with an equivalence relation $R$ on $S$.
    
    In this section, we develop a method to construct reduced semiconstellations from semiconstellations and certain equivalence relations, such as the one used in \cite{haag2026a}, generalizing the construction of the monoid $S^1$. This method is a key component in reducing the globalization problem for actions of semiconstellations to the one of restriction semiconstellations.
    
    \subsection{Characterization} We investigate the relations associated to strong morphisms from a semiconstellation $S$ to a reduced semiconstellation $(T,\ast)$ and, given a such relation on $S$, we explicitly construct a reduced semiconstellations $(T,\ast)$ containing $S$ as a (2)-subalgebra.\\
    
    Throughout this subsection, $S$ denotes a semiconstellation. Given a function $\varphi \colon S \to T$, we consider the relation $R_\varphi = \{ (s,t) \colon \varphi(s) = \varphi(t) \}$ on $S$. Note that $R_\varphi$ is always an equivalence relation. In particular, if $(T,\ast)$ is a restriction semiconstellation, we write $R_{\varphi^\ast} = R_{\ast \circ \varphi}$.
    
    \begin{lemma} \label{lema:adequate}
        Let $\varphi \colon S \to T$ be a (2)-morphism, where $(T,\ast)$ is a restriction semiconstellation.
        \begin{enumerate}
            \item If $\varphi$ is strong, then $R_{\varphi^\ast} \subseteq \{(s,t) \colon S^s = S^t\}$.
            \item If $(T,\ast)$ is reduced, then $\mathcal{E}\{(t,st) \colon \exists st\} \subseteq R_{\varphi^\ast}$.
        \end{enumerate}
    
        \begin{proof}
            \noindent(1) By Remark \ref{obs:restriction}, we have $T^s = T^{s^\ast}$ for every $s \in T$. Thus, if $(s,t) \in R_{\varphi^\ast}$, then $\varphi(s)^\ast = \varphi(t)^\ast$, and consequently $$ T^{\varphi(s)} = T^{\varphi(s)^\ast} = T^{\varphi(t)^\ast} = T^{\varphi(t)}. $$ In particular, $\varphi(s)\varphi(r)$ is defined if and only if $\varphi(t)\varphi(r)$ is defined. Since $\varphi$ is strong, these products are defined if and only if $sr$ and $tr$, respectively, are defined. It follows that $sr$ is defined if and only if $tr$ is defined, or equivalently, $S^s = S^t$.\\
    
            \noindent(2) Suppose that $st$ is defined. Since $\varphi$ is a (2)-morphism, the composition $\varphi(s)\varphi(t)$ is defined and equals to $\varphi(st)$. Therefore, since $(T,\ast)$ is reduced, we obtain that
                $$ \varphi(st)^\ast = (\varphi(s)\varphi(t))^\ast \overset{\ref{prop:char-reduced}(5)}{=} \varphi(t)^\ast. $$
            That is, $(t,st) \in R_{\varphi^\ast}$. Since $R_{\varphi^\ast}$ is an equivalence relation, it follows that $\mathcal{E}\{(t,st) \colon \exists st\}$, the smallest equivalence relation containing the pairs $(t,st)$, is contained in $R_{\varphi^\ast}$.
        \end{proof}
    \end{lemma}
    
    In the sequel, we denote $R_M = \{(s,t) \colon S^s = S^t\}$ and $R_m = \mathcal{E}\{(t,st) \colon \exists st\}$.
    
    \begin{defi}
        An equivalence relation $R$ on $S$ is called \textit{adequate} if $R_m \subseteq R \subseteq R_M$.
    \end{defi}
    
    Note that $R_M$ and $R_m$ are equivalence relations on $S$ satisfying $R_m \subseteq R_M$. Therefore, by definition, both $R_m$ and $R_M$ are adequate relations on $S$. Furthermore, it follows from Lemma \ref{lema:adequate} that if $\varphi \colon S \to T$ is a strong $(2)$-morphism into a reduced semiconstellation $(T,\ast)$, then $R_{\varphi^\ast}$ is an adequate relation on $S$. The next result shows that, conversely, every adequate relation on $S$ arises in this way.
    
    \begin{prop} \label{prop:adequate}
        Let $R$ be an adequate relation on $S$. Then there exists a reduced semiconstellation $(T,\ast)$ and an injective strong (2)-morphism $\iota \colon S \to T$ such that $R_{\iota^\ast} = R$.
    
        \begin{proof}
            Since $R$ is an equivalence relation, we may consider the quotient set $S/R$. Let $S^R = S \cup S/R$. We denote elements in $S/R$ by $\overline{s}$, elements in $S$ by $s$, and a generic element of $S^R$ by $x$. Define a partial binary operation $\circ$ on $S^R$ as follows:
            \begin{itemize}
                \item $s \circ t$ is defined if and only if $st$ is defined, and in this case $s \circ t = st$.
                \item $x \circ \overline{t}$ is defined if and only if $x \in \overline{t} \cup \{\overline{t}\}$, and in this case $x \circ \overline{t} = x$.
                \item $\overline{s} \circ t$ is defined if and only if $t \in S^s$, and in this case $\overline{s} \circ t = t$.
            \end{itemize}
            We first show that $(S^R,\circ)$ is a semiconstellation. To this end, we verify axioms \eqref{sc1} and \eqref{sc2} for triples in $S^R$. There are eight possible types of triples. We give the details for one of them and briefly indicate how the remaining cases follow. Let $(s,\overline{t},r) \in S \times S/R \times S$. Then:
            \begin{enumerate}
                \item $s \circ \overline{t}$ and $\overline{t} \circ r$ are defined if and only if $s \in \overline{t}$ and $tr$ is defined.
                \item $s \circ \overline{t}$ and $(s \circ \overline{t}) \circ r = s \circ r$ are defined if and only if $s \in \overline{t}$ and $sr$ are defined.
                \item $\overline{t} \circ r$ and $s \circ (\overline{t} \circ r) = s \circ r$ are defined if and only if $tr$ and $sr$ are defined.
            \end{enumerate}
            Now, $s\in\overline{t}$ if and only if $(s,t)\in R$. Since $R$ is adequate, this implies that $S^s=S^t$. Hence, whenever $s\in\overline{t}$, the product $sr$ is defined if and only if $tr$ is defined. Therefore, (1) and (2) are equivalent, and each implies (3). Moreover, whenever these conditions hold,
                $$ (s \circ \overline{t}) \circ r = sr = s \circ (\overline{t} \circ r). $$
            The remaining cases follow directly from the defining properties of $R$ and from the semiconstellation axioms in $S$. More precisely, the cases $(s,t,r)$ and $(\overline{s},t,r)$ follow from the fact that $S$ is a semiconstellation; the case $(s,t,\overline{r})$ follows from $(t,st)\in R$ whenever $st$ is defined; the case $(s,\overline{t},\overline{r})$ follows from the fact that $R$ is an equivalence relation; the cases $(\overline{s},t,\overline{r})$ and $(\overline{s},\overline{t},\overline{r})$ are immediate; and the case $(\overline{s},\overline{t},r)$ is analogous to the case $(s,\overline{t},r)$ considered above.

            Thus, $(S^R,\circ)$ is a semiconstellation. Define $\ast\colon S^R\to S^R$ by $s^\ast = \overline{s}$ and $\overline{s}^\ast = \overline{s}$. For every $x\in S^R$, the product $x\circ x^\ast$ is defined and $x^\ast$ is an identity of $S^R$. Hence, by Proposition \ref{prop:reduced}, $(S^R,\ast)$ is a reduced semiconstellation. Finally, let $\iota\colon S\to S^R$ be the inclusion map, that is, $\iota(s)=s$. It is immediate from the definition of $\circ$ that  $\iota$ is an injective strong (2)-morphism. Moreover, $$ R_{\iota^\ast} = \{ (s,t) \in S \times S \colon \iota(s)^\ast = \iota(t)^\ast \} = \{ (s,t) \in S \times S \colon \overline{s} = \overline{t} \} = R. $$
            This completes the proof.
        \end{proof}
    \end{prop}
    
    \begin{obs} \label{obs:adequate}
        In Proposition \ref{prop:adequate}, we may question when the semiconstellation $S^R$ is a semigroupoid. From the proof of the proposition, $S^R$ is a semigroupoid if and only if $S$ is a semigroupoid such that $S^s \cap S^t \neq \emptyset$ implies $S^s = S^t$, and the relation $R$ contains the relation
        \begin{align*}
            R' = \mathcal{E}\left(\{ (s,t) \colon S^s = S^t \neq \emptyset \} \cup \{ (t,st) \colon \exists st \}\right).
        \end{align*}
        The semigroupoids such that $S^s \cap S^t \neq \emptyset$ implies $S^s = S^t$ are called \textit{categorical semigroupoid}. By \cite[Theorem 2.15]{cordeiro2023etale}, categorical semigroupoids are precisely the semigroupoids $S$ that admit a compatible directed graph structure $(S,C_0,\mathbf{d},\mathbf{r})$. In this case, appending identities to each vertex of the graph, we obtain a category that contains $S$ as a (2)-subalgebra. Therefore, Proposition \ref{prop:adequate} shows that, if a semigroupoid $S$ is not categorical, then the process of appending identities to $S$ creates a reduced semiconstellation instead of a category.
    \end{obs}

    \subsection{Applications} In this subsection, we use adequate relations to obtain a representation theorem for semiconstellations and a free construction functor from a suitable category to the category of restriction semiconstellations. This functor will play a central role in the next section.\\

    Firstly, we show that semiconstellations are, up to isomorphism, precisely the (2)-subalgebras of \textit{function constellations}. More precisely, given a set $X$, we consider the set $C(X) = PT(X)$ endowed with the partial binary operation
        $$ \exists f \circ g \iff im(g) \subseteq dom(f), \text{ and in this case } f \circ g = f \star g, $$
    where $f \star g$ denotes the composition in the semigroup of partial functions $PT(X)$, and with a unary operation $\mathbf{d} \colon C(X) \to C(X)$, given by $\mathbf{d}(f) = id_{dom(f)}$, for all $f \in C(X)$. The triple $(C(X),\circ,\mathbf{d})$ is called a \textit{function constellation}.

    \begin{theorem} \label{coro:adequate}
        Every semiconstellation is isomorphic to a (2)-subalgebra of a function constellation. Conversely, every (2)-subalgebra of a function constellation is a semiconstellation.
    
        \begin{proof}
            Let $S$ be a semiconstellation and fix an adequate relation $R$ on $S$. For instance, take $R = R_M$. By Proposition \ref{prop:adequate}, the function $\iota \colon S \to S^R$ is an injective strong (2)-morphism. Since $(S^R,\ast)$ is a reduced semiconstellation, we obtain from Proposition \ref{prop:reduced} that $(S^R,\ast)$ is also a constellation. Therefore, \cite[Proposition 2.6]{gould2009restriction} shows that the function $\lambda \colon S^R \to C(S^R)$, given by
                $$ x \mapsto \lambda_x \colon (S^R)^x \to S^R, \quad \lambda_x(y) = xy, \quad \forall y \in (S^R)^x, $$
            is a strong (2,1)-morphism. We claim that, in this case, the function $\lambda$ is also injective. In fact, if $x,y \in S^R$ are such that $\lambda_x = \lambda_y$, then $x^\ast \in dom(\lambda_x)$. Hence,
                $$ x \overset{\eqref{r1}}{=} x \circ x^\ast = \lambda_x(x^\ast) = \lambda_y(x^\ast) = y \circ x^\ast \overset{\ref{prop:char-reduced}(2)}{=} y. $$
            Since both $\iota \colon S \to S^R$ and $\lambda \colon S^R \to C(S^R)$ are injective strong (2)-morphisms, the composition $\lambda \circ \iota \colon S \to C(S^R)$ is an injective strong (2)-morphism. From Proposition \ref{prop:universal-2}, we conclude that $S$ is isomorphic to a (2)-subalgebra of $C(S^R)$. The converse is trivial.
        \end{proof}
    \end{theorem}

    Next, we show that the construction in Proposition \ref{prop:adequate} extends to a free construction functor from a category of pairs $(S,R)$, where $S$ is a semiconstellation and $R$ is an adequate relation on $S$, to the category of restriction semiconstellations. We begin by defining the categories involved. \\

    Denote by $(\mathcal{S} \times \mathcal{R})_0$ the family of pairs $(S,R)$, where $S$ is a semiconstellation and $R$ is any relation on $S$. Given two objects $(S,R), (S',R') \in (\mathcal{S} \times \mathcal{R})_0$, we say that a function $\varphi \colon S \to S'$ is a \textit{$(2,R)$-morphism} if it is a $(2)$-morphism such that $(\varphi(s),\varphi(t)) \in R'$, for all $(s,t) \in R$, and in this case we denote $\varphi \colon (S,R) \to (S',R')$. Thus, a $(2,R)$-morphism preserves both the semiconstellation structure and the prescribed relations. Let $(\mathcal{S} \times \mathcal{R})_1$ denote the family of all (2,$R$)-morphisms between objects of $(\mathcal{S} \times \mathcal{R})_0$. Identity maps are $(2,R)$-morphisms, and the composition of two $(2,R)$-morphisms is again a $(2,R)$-morphism. Therefore, the quintuple
        $$ \mathcal{S} \times \mathcal{R} = ((\mathcal{S} \times \mathcal{R})_0, (\mathcal{S} \times \mathcal{R})_1, \mathbf{d}, \mathbf{r}, \circ) $$
    is a category, where $\mathbf{d}$, $\mathbf{r}$ and $\circ$ are the usual domain, range, and composition of functions.\\
    
    On the other hand, restriction semiconstellations form a full subcategory $rSCON$ of the category $Alg(2,1)$. We define a functor $U \colon rSCON \to \mathcal{S} \times \mathcal{R}$ on objects by
        $$ U(S,\ast) = (S,R_\ast), $$
    where $R_\ast = \{ (s,t) \in S \times S \colon s^\ast = t^\ast \}$, and on morphisms by
        $$U(\varphi \colon (S,\ast) \to (S',\ast)) = \varphi \colon (S,R_\ast) \to (S',R'_\ast) $$
    To see that $U$ is well defined, let $\varphi \colon (S,\ast) \to (S',\ast)$ be a $(2,1)$-morphism. Clearly, $\varphi$ is a $(2)$-morphism. Moreover, if $(s,t) \in R_\ast$, then $s^\ast=t^\ast$, and hence  $\varphi(s)^\ast = \varphi(s^\ast) = \varphi(t^\ast) = \varphi(t)^\ast$. Thus, $(\varphi(s),\varphi(t)) \in R'_\ast$, so that $U$ associates a (2,1)-morphism $\varphi \colon (S,\ast) \to (S',\ast)$ to a (2,$R$)-morphism $U\varphi \colon U(S,\ast) \to U(S',\ast)$. Clearly $U$ preserves identities and composition, hence it is a functor.\\

    We can now reinterpret Proposition \ref{prop:adequate} in this setting. Let $(S,R)$ be an object of $\mathcal{S} \times \mathcal{R}$ with $R$ adequate. By Proposition \ref{prop:adequate}, $(S^R,\ast)$ is a restriction semiconstellation containing $S$ and the inclusion $\iota \colon S \to S^R$ is an injective strong (2)-morphism such that
        $$ R = \{ (\iota(s),\iota(t)) \colon (s,t) \in R \} = R_{\iota^\ast} \subseteq \{ (x,y) \in S^R \times S^R \colon x^\ast = y^\ast \}. $$
    Consequently, $\iota \colon (S,R) \to U(S^R,\ast)$ is a (2,$R$)-morphism in $\mathcal{S} \times \mathcal{R}$. The following theorem shows that this construction is universal.

    \begin{theorem} \label{teo:adequate}
        Let $R$ be an adequate relation on $S$. Then $(\iota,(S^R,\ast))$ is free over $(S,R)$.
    
        \begin{proof}
            Let $(T,\ast)$ be a restriction semiconstellation and let $\varphi \colon (S,R) \to U(T,\ast)$ be a $(2,R)$-morphism. Suppose that $\Phi \colon (S^R,\ast) \to (T,\ast)$ is a $(2,1)$-morphism satisfying $U\Phi \circ \iota = \varphi$. Then, for all $s, \overline{s} \in S^R$, we have
                $$ \Phi(s) = \Phi(\iota(s)) = \varphi(s) \quad\text{and}\quad \Phi(\overline{s}) = \Phi(\iota(s)^\ast) = \Phi(\iota(s))^\ast = \varphi(s)^\ast. $$
            That is, a (2,1)-morphism satisfying $U\Phi \circ \iota = \varphi$, if it exists, is uniquely determined. It remains to prove its existence. Define $\Phi \colon S^R \to T$ by $$\Phi(s) = \varphi(s) \quad \text{and} \quad \Phi(\overline{s}) = \varphi(s)^\ast.$$  We first show that $\Phi$ is well defined.
    
            If $\overline{s}=\overline{t}$, then $(s,t)\in R$. Since $\varphi \colon (S,R) \to U(T,\ast)$ is a (2,$R$)-morphism, we have that $\varphi(s)^\ast = \varphi(t)^\ast$. Therefore, $\Phi(\overline{s}) = \varphi(s)^\ast$ is independent of the choice of representative. Since $\overline{s} \in (S^R)^\ast$, it follows from Lemma \ref{lema:restriction}(2) that
                $$ \Phi(s^\ast) = \Phi(\overline{s}^\ast) = \Phi(\overline{s}) = \varphi(s)^\ast = \Phi(s)^\ast = \Phi(\overline{s})^\ast. $$
            To show that $\Phi$ preserves composition, we split the proof in four cases:
            \begin{itemize}
                \item If $s \circ t$ is defined, then $st$ is defined. Hence,
                    $$ \Phi(st) = \varphi(st) = \varphi(s)\varphi(t) = \Phi(s)\Phi(t). $$
    
                \item If $s \circ \overline{t}$ is defined, then $(s,t) \in R$. Thus, $\Phi(s)^\ast = \varphi(s)^\ast = \varphi(t)^\ast = \Phi(\overline{t})$ and
                    $$ \Phi(s \circ \overline{t}) = \Phi(s) = \varphi(s) \overset{\eqref{r1}}{=} \varphi(s)\varphi(s)^\ast = \Phi(s)\Phi(\overline{t}). $$
    
                \item If $\overline{s} \circ t$ is defined, then $st$ is defined. Since $(t,st) \in R$, we have $\overline{t} = \overline{st}$. Therefore,
                \begin{align*}
                    \Phi(\overline{s} \circ t) &= \Phi(t) \overset{\eqref{r1}}{=} \Phi(t)\Phi(t)^\ast = \Phi(t) \Phi(\overline{t}) = \Phi(t) \Phi(\overline{st}) \\
                    &= \varphi(t) \varphi(st)^\ast = \varphi(t) (\varphi(s)\varphi(t))^\ast \overset{\eqref{r4}}{=} \varphi(s)^\ast \varphi(t) = \Phi(\overline{s}) \Phi(t).
                \end{align*}
    
                \item If $\overline{s} \circ \overline{t}$ is defined, then $\overline{s} = \overline{t}$, and in this case
                    $$ \Phi(\overline{s} \circ \overline{s}) = \Phi(\overline{s}) = \varphi(s)^\ast \overset{\ref{lema:restriction}(1)}{=} \varphi(s)^\ast \varphi(s)^\ast = \Phi(\overline{s}) \Phi(\overline{s}). $$
            \end{itemize}
            This concludes that $\Phi$ is the unique (2,1)-morphism satisfying $U\Phi \circ \iota = \varphi$.
        \end{proof}
    \end{theorem}

    Denote by $\mathcal{S} \times \mathcal{A}$ the full subcategory of $\mathcal{S} \times \mathcal{R}$ consisting of pairs $(S,R)$ such that $R$ is an adequate relation on $S$. It follows from Theorem \ref{teo:adequate} and Proposition \ref{prop:livre} that there is a free construction functor $\mathcal{S} \times \mathcal{A} \to rSCON$, which associates each pair $(S,R)$ to the restriction semiconstellation $(S^R,\ast)$.

    \subsection{Examples} We characterize the adequate relations on a semiconstellation $S$ in several particular cases: when $S$ is a restriction semiconstellation, a right zero semigroup, a free semigroup, or the semiconstellation $S_m^n$ introduced in Example \ref{exe:semiconst}. \\
    
    Firstly, we show that a restriction semiconstellation admits a unique adequate relation.
    
    \begin{exe} \label{exe:adequate-1}
        Assume that $S$ admits a restriction semiconstellation structure $(S,\ast)$. We show that $R_m = R_M$ and, consequently, $R_M = \{(s,t) \colon S^s = S^t\}$ is the unique adequate relation on $S$. By definition, we have $R_m \subseteq R_M$. Conversely, let $(s,t) \in R_M$. From \eqref{r1}, we obtain
            $$ (s^\ast,s) = (s^\ast,ss^\ast) \in R_m \quad\text{and}\quad (t^\ast,t) = (t^\ast, tt^\ast) \in R_m. $$
        Since $(s,t) \in R_M$, it follows from Remark \ref{obs:restriction} that $S^{s^\ast} = S^s = S^t = S^{t^\ast}$. Therefore, $s^\ast t^\ast$ and $t^\ast s^\ast$ are defined, from where $(t^\ast, s^\ast t^\ast) \in R_m$ and $(s^\ast, t^\ast s^\ast) \in R_m$. Finally, since $R_m$ is an equivalence relation, we conclude that
            $$ (s^\ast, s), (s^\ast, s^\ast t^\ast) \overset{\eqref{r2}}{=} (s^\ast, t^\ast s^\ast), (t^\ast, s^\ast t^\ast), (t^\ast, t) \in R_m \implies (s,t) \in R_m. $$
        This shows that $R_M \subseteq R_m$ and, hence, $R_m = R_M$ is the unique adequate relation on $S$.
    \end{exe}

    To illustrate the previous example, suppose that $(S,+)$ is a restriction semigroupoid such that ${^sS} \neq \emptyset$, for every $s \in S$, and let $R = R_M$ be the unique adequate relation on $S$. Note that $R$ contains the relation $R'$ from Remark \ref{obs:adequate}, and \cite[Proposition 2.15]{haag2026c} shows that every restriction semigroupoid is categorical. Therefore, $(S^R,\ast)$ is also a semigroupoid. Moreover, it has a restriction category structure containing $(S,+)$ as a $(2,1)$-subalgebra.

    In fact, let $x \in S^R$. Since ${^sS} \neq \emptyset$, we may choose $t \in S$ such that $tx$ is defined. Define $\mathbf{d}(x) = x^\ast$ and $\mathbf{r}(x) = t^\ast$. Then $(S^R,\mathbf{d},\mathbf{r},\circ)$ is a category. Finally, redefine the restriction operation on the elements added to $S$ by setting $\overline{s}^\ast = \overline{s}$, for every $s \in S$. Thus, the quintuple $(S^R,\mathbf{d},\mathbf{r},\circ,\ast)$ is a restriction category containing $(S,+)$ as a $(2,1)$-subalgebra.

    This shows that the unique adequate relation on a restriction semigroupoid is precisely the relation that allows us to extend it to a restriction category by adjoining identities at the vertices of its underlying directed graph.\\

    For the second example, recall that a right zero semigroup is a semigroup $S$ whose operation is defined by $st = t$, for all $s,t \in S$. In contrast with the previous example, every equivalence relation on a right zero semigroup is an adequate relation.
    
    \begin{exe} \label{exe:adequate-2}
        Let $S$ be a right zero semigroup. We show that $R_m = \{(s,s) \colon s \in S\}$ and $R_M = S \times S$. In fact, since $S$ is a semigroup, we have $S^s = S^t = S$, for all $s,t \in S$. Therefore, $R_M = S \times S$. On the other hand, since $S$ is a right zero semigroup, we have $(t,st) = (t,t)$, for all $s,t \in S$. Therefore,
            $$ R_m = \mathcal{E}\{ (t,st) \colon s,t \in S \} = \mathcal{E}\{(t,t) \colon t \in S\} = \{(t,t) \colon t \in S\}. $$
        Since every equivalence relation contains $\{(s,s) \colon s \in S\} = R_m$ and is contained in $S \times S = R_M$, it follow that every equivalence relation on $S$ is an adequate relation.
    \end{exe}
    
    In the third example, we prove that adequate relations on a free semigroup $\langle X \rangle$ are in correspondence with equivalence relations on the generator set $X$. Recall that, given a non-empty set $X$, the free semigroup $\langle X \rangle$ is the family of finite sequences $x_n \dots x_1$, where $x_i \in X$, with operation given by concatenation of sequences.  
    
    \begin{exe} \label{exe:adequate-3}
        Let $S = \langle X \rangle$ be a free semigroup. Since $S$ is a semigroup, we have $R_M = S \times S$ as in Example \ref{exe:adequate-2}. Next, we show that for each $x \in S$, there is a unique $x_i \in X$ such that $(x_i,x) \in R_m$. In fact, if $x = x_1$, then $(x_1,x) = (x,x) \in R_m$ by reflexivity. Assume that $x = x_n \dots x_1$, where $n \geq 2$. Then setting $s = s_n \dots s_2$ and $t = x_1$, we obtain that
            $$ (x_1,x) = (x_1,x_n \dots x_2x_1) = (t,st) \in R_m. $$
        This shows that each $x \in S$ is related to some $x_i \in X$. For the uniqueness, suppose that $(x_i,x_j) \in R_m$, where $x_i,x_j \in X$. Since $R_m$ is the smallest equivalence relation containing the relation $L = \{(t,st) \colon s,t \in S\}$, there are $t^1,\dots,t^n \in S$ such that $x_i = t^1$, $t^n = x_j$ and
            $$ (t^i,t^{i+1}) \in L \quad\text{or}\quad (t^{i+1},t^i) \in L, \quad \forall i=1,\dots,n-1. $$
        But if $(t^i,t^{i+1}) \in L$, then there is $s^i \in S$ such that $(t^i,t^{i+1}) = (t^i,s^it^i)$. Writing $t^i = t^i_p \dots t^i_1$ and $s^i = s^i_q \dots s^i_1$, we obtain that $t^{i+1} = s^i_q \dots t^i_1$. Thus, $t^i$ and $t^{i+1}$ have the same first coordinate, namely, $t^i_1$. Analogously, if $(t^{i+1},t^i) \in L$, then the elements $t^{i+1}$ and $t^i$ have the same first coordinate. Hence,
            $$ x_i = t^1_1 = \dots = t^n_1 = x_j. $$
        This shows that each equivalence class of $R_m$ contains at most on element $x_i \in X$. Thus, each $x \in S$ is related to a unique $x_i \in X$ and $S/R_m \simeq X$ as sets. Consequently, $R$ is an adequate relation on $S$ if and only if $R$ is an equivalence relation containing $R_m$. Such equivalence relations are in correspondence with the equivalence relations on $S/R_m \simeq X$. 
    \end{exe}

    As noted in the previous examples, if $S$ is a semigroup, then the maximum adequate relation on $S$ is $R_M = S \times S$. In this case, the reduced semiconstellation $(S^{R_M},\ast)$ coincides with the monoid $S^1 = S \cup \{1\}$ obtained by appending an identity element to $S$, and the restriction structure is given by $s^\ast = 1$, for all $s \in S^1$.
    
    On the other hand, if $S$ is a semigroup and $R \neq S \times S$ is an adequate relation on $S$, then $(S^R,\ast)$ is a reduced semiconstellation that is not a semigroupoid. In fact, it follows from Remark \ref{obs:adequate} that if $S^R$ is a semigroupoid, then $R$ must contain the relation
        $$ R' = \mathcal{E}(\{(s,t) \colon S^s = S^t \neq \emptyset\} \cup \{(t,st) \colon \exists st\}) = S \times S. $$
    Thus, semiconstellations can be constructed even from semigroups.\\
    
    Lastly, we characterizes adequate relations on the semiconstellations $S_m^n$ from Example \ref{exe:semiconst}.
    
    \begin{exe} \label{exe:adequate-4}
        Let $S_m^n = \{ s_{i}^{j} \colon (i,j) \in I_{m+1} \times I_n \} \cup \{ d_i \colon i \in I_m \}$. From Example \ref{exe:semiconst}, we have $S^{d_i} = \{s_{i}^{j} \colon j \in I_n\}$, for every $i \in I_m$, and $S^{s_{i}^{j}} = \emptyset$, for every $(i,j) \in I_{m+1} \times I_n$. Therefore,
            $$ R_M = \{ (s_i^j, s_p^q) \colon i,p \in I_{m+1}, j,q \in I_n\} \cup \{ (d_i,d_i) \colon i \in I_m \}. $$
        That is, an adequate relation on $S_m^n$ must isolate the elements $d_1,\dots,d_m$. On the other hand, $st$ is defined if and only if $s = d_i$ and $t = s_i^j$, and in this case $d_is_i^j = s_{i+1}^j$. Thus,
            $$ R_m = \mathcal{E}\{ (s_i^j, s_{i+1}^j) \colon (i,j) \in I_m \times I_n \} = \{ (s_p^j,s_q^j) \colon p,q \in I_{m+1}, j \in I_n \} \cup \{ (d_i,d_i) \colon i \in I_m \}. $$
        Therefore, an adequate relation on $S_m^n$ is any equivalence relation that relates the elements $s_1^j,\dots,s_{m+1}^j$, for each $j=1,\dots,n$, and isolate the elements $d_1,\dots,d_m$. Since
            $$ S_m^n/R_m \simeq \{ s_1^1, \dots, s_1^n, d_1, \dots, d_m \}, $$
        we conclude that adequate relations on $S_m^n$ correspond to equivalence relations on $\{s_1^1,\dots,s_1^n\}$.
    \end{exe}

    Note that the reduced semiconstellations $(T_m^n,\ast)$ from Example \ref{exe:restriction} are precisely the restriction semiconstellations $((S_m^n)^{R_M},\ast)$, where $R_M$ is the maximum adequate relation on $S_m^n$. Since, by the previous example, adequate relations on $S_m^n$ are in correspondence with equivalence relations on the set $\{ s_1^1,\dots,s_1^n \}$, it follows that $S_m^n$ has a unique adequate relation if and only if $n = 1$. This is the case for the semiconstellation $S = S_2^1$ from Example \ref{exe:S21}.

    \section{The globalization problem for actions of \texorpdfstring{$S$}{}} \label{sec:6}

    \subsection{Actions of semiconstellations} We introduce partial and $R$-compatible actions of $S$ on sets, where $R$ is an adequate relation on $S$. We shows that relative $S$-subalgebras of $R$-compatible actions are partial actions and formulate an analogue to the globalization problem for actions of semiconstellations.\\
    
    In this section, $S$ is a semiconstellation, $R$ is an adequate relation on $S$ and $X$ is a set. We denote by $\iota \colon S \to S^R$ the inclusion of $S$ into the reduced semiconstellation $(S^R,\ast)$ constructed in Proposition \ref{prop:adequate}.
    
    \begin{defi} \label{defi:Sactions}
        A \textit{partial action} of $S$ is a pair $(X;\alpha)$, where $\alpha = \{\alpha_s\}_{s \in S}$ is a family of partial unary operations on $X$ satisfying the following conditions:
        \begin{enumerate} \Not{P}
            \item If $st$ is defined, then $\alpha_t^{-1}({_sX} \cap X_t) = {_{st}X} \cap {_tX}$.
            \item If $st$ is defined, then $\alpha_s(\alpha_t(x)) = \alpha_{st}(x)$, for all $x \in {_{st}X} \cap {_tX}$.
        \end{enumerate}
        A pair $(X;\alpha)$ as above is called an \textit{$R$-compatible action} if it is a partial action such that
        \begin{align*}
            {_sX} = {_tX}, \ \forall (s,t) \in R. \tag{C} \label{Rc}
        \end{align*}
    \end{defi}
    
    We provide an analogue of Proposition \ref{prop:actions} for partial actions of semiconstellations.
    
    \begin{defi}
        Let $(T,\ast)$ be a restriction semiconstellation. A function $\varphi \colon S \to T$ is called:
        \begin{itemize}
            \item A \textit{(2,$R$)-morphism} if it is a (2)-morphism such that $\varphi(s)^\ast = \varphi(t)^\ast$, for all $(s,t) \in R$.
            \item A \textit{semi-premorphism} if it satisfies \eqref{pm1}.
        \end{itemize}
    \end{defi}

    \begin{obs} \label{obs:2Rmorphism}
        In the previous section, we used the term (2,$R$)-morphism to nominate the arrows of the category $\mathcal{S} \times \mathcal{R}$. More precisely, given $(S,R), (S',R') \in \mathcal{S} \times \mathcal{R}$, we called a function $\varphi \colon S \to S'$ a (2,$R$)-morphism if it is a (2)-morphism such that
            $$ (\varphi(s),\varphi(t)) \in R', \quad \forall (s,t) \in R. $$
        We observe that if $(T,\ast)$ is a restriction semiconstellation, then a function $\varphi \colon S \to T$ is a (2,$R$)-morphism, as defined above, if and only if it is a (2,$R$)-morphism $\varphi \colon (S,R) \to U(T,\ast)$ in the category $\mathcal{S} \times \mathcal{R}$, where $U(T,\ast) = (T,R_\ast)$ and $R_\ast = \{ (s,t) \in T \times T \colon s^\ast = t^\ast \}$. Therefore, the terminology above is consistent with that introduced in the previous section.
    \end{obs}
    
    \begin{prop} \label{prop:actions-2}
        Let $(X;\alpha)$ be a $S$-algebra and denote by $\alpha \colon S \to PT(X)$ the function given by $\alpha(s) = \alpha_s$, for all $s \in S$. Then:
        \begin{enumerate}
            \item $(X;\alpha)$ is a partial action if and only if $\alpha \colon S \to PT(X)$ is a semi-premorphism.
            \item $(X;\alpha)$ is an $R$-compatible action if and only if $\alpha \colon S \to PT(X)$ is a (2,$R$)-morphism.
        \end{enumerate}
    
        \begin{proof}
            The proof of (1) is identical to Proposition \ref{prop:actions}(1). To show (2), suppose that $(X;\alpha)$ is an $R$-compatible action. If $(s,t) \in R$, then
                $$ \alpha_s^\ast = id_{_sX} \overset{\eqref{Rc}}{=} id_{_tX} = \alpha_t^\ast. $$
            On the other hand, since an $R$-compatible action is, in particular, a partial action, it follows from (1) that $\alpha$ is a semi-premorphism. Therefore, if $st$ is defined, then $(t,st) \in R$ and, hence
            \begin{align*}
                \alpha_s \star \alpha_t \overset{\eqref{pm1}}{=} \alpha_{st} \star \alpha_t^\ast \overset{\eqref{Rc}}{=} \alpha_{st} \star \alpha_{st}^\ast \overset{\eqref{r1}}{=} \alpha_{st}.
            \end{align*}
            This shows that $\alpha \colon S \to PT(X)$ is a (2,$R$)-morphism. Conversely, suppose that $\alpha$ is a (2,$R$)-morphism. The function $\alpha$ is a semi-premorphism since, if $st$ is defined, then
                $$ \alpha_s \star \alpha_t \overset{\eqref{r1}}{=} \alpha_s \star \alpha_t \star \alpha_t^\ast = \alpha_{st} \star \alpha_t^\ast. $$
            By (1), the pair $(X;\alpha)$ is a partial action of $S$. Furthermore,
                $$(s,t) \in R \implies id_{_sX} = \alpha_s^\ast = \alpha_t^\ast = id_{_tX} \implies {_sX} = {_tX}. $$
            That is, condition \eqref{Rc} holds, completing the proof.
        \end{proof}
    \end{prop}
    
    For a restriction semiconstellation $(T,\ast)$, we denote by $\mathcal{A}_p(T,\ast)$ and $\mathcal{A}(T,\ast)$ the categories of partial and global actions of $(T,\ast)$, respectively. In the sequel, we denote by $\mathcal{A}_p(S)$ and $\mathcal{A}_R(S)$ full subcategories of $Alg(S)$ whose objects are the partial and $R$-compatible actions of $S$, respectively. We denote the inclusion functors by
        $$ I_{(T,\ast)} \colon \mathcal{A}(T,\ast) \to \mathcal{A}_p(T,\ast) \quad\text{and}\quad I_S \colon \mathcal{A}_R(S) \to \mathcal{A}_p(S). $$
    
    By Proposition \ref{prop:universal}, an isomorphism between (2,1,1)-algebras is a bijective strong (2,1,1)-morphism. Regarding a category $\mathcal{C}$ as the (2,1,1)-algebra $(\mathcal{C}_1,\mathbf{d},\mathbf{r})$, it follows that an isomorphism $\mathcal{C} \to \mathcal{D}$ is precisely a fully faithful functor $F \colon \mathcal{C} \to \mathcal{D}$ that is bijective on objects.
    
    \begin{prop} \label{prop:actions-3}
        The inclusion $\iota \colon S \to S^R$ induces a functor $(-)|_S \colon \mathcal{A}_p(S^R,\ast) \to \mathcal{A}_p(S)$ whose restriction to $\mathcal{A}(S^R,\ast)$ yields an isomorphism of categories $\mathcal{A}(S^R,\ast) \to \mathcal{A}_R(S)$.
    
        \begin{proof}
            Since $\iota \colon S \to S^R$ in the inclusion of $S$ in $S^R = S \cup S/R$, for every function $\varphi \colon S^R \to T$, the composition $\varphi \circ \iota$ coincides with the restriction $\varphi|_S \colon S \to T$. Thus,
            \begin{align*}
                (X;\alpha) \in \mathcal{A}_p(S^R,\ast) &\overset{\ref{prop:actions}(1)}{\iff} \alpha \colon (S^R,\ast) \to PT(X) \text{ is a premorphism} \\
                &\implies \alpha \circ \iota \colon S \to PT(X) \text{ is a semi-premorphism} \\
                &\overset{\ref{prop:actions-2}(1)}{\iff} (X;\alpha \circ \iota) \in \mathcal{A}_p(S).
            \end{align*}
            On the other hand, let $\varphi \colon (X;\alpha) \to (Y;\beta)$ be a $S^R$-morphism. Since $S \subseteq S^R$, we obtain that $\varphi \colon (X;\alpha|_S) \to (Y;\beta|_S)$ is a $S$-morphism. Define $(-)|_S \colon \mathcal{A}_p(S^R,\ast) \to \mathcal{A}_p(S)$ by
                $$ (X;\alpha)|_S = (X;\alpha|_S) \quad\text{and}\quad \varphi|_S = \varphi. $$
            From $id_{(X;\alpha)}|_S = id_X = id_{(X;\alpha|_S)}$ and $(\varphi \circ \psi)|_S = \varphi \circ \psi = \varphi|_S \circ \psi|_S$, whenever the composition is defined, we conclude that $(-)|_S$ is a functor. Next, we show that the restriction of the functor $(-)|_S$ to the category $\mathcal{A}(S^R,\ast)$ is an isomorphism $\mathcal{A}(S^R,\ast) \to \mathcal{A}_R(S)$. Note that
            \begin{align*}
                (X;\alpha) \in \mathcal{A}(S^R,\ast) &\overset{\ref{prop:actions}(2)}{\iff} \alpha \colon (S^R,\ast) \to PT(X) \text{ is a (2,1)-morphism},
            \end{align*}
            and
            \begin{align*}
                (X;\alpha) \in \mathcal{A}_R(S) &\overset{\ref{prop:actions-2}(2)}{\iff} \alpha \colon S \to PT(X) \text{ is a (2,$R$)-morphism} \\
                &\overset{\ref{obs:2Rmorphism}}{\iff} \alpha \colon (S,R) \to U(PT(X)) \in \mathcal{S} \times \mathcal{R}.
            \end{align*}
            By Theorem \ref{teo:adequate}, the pair $(\iota,(S^R,\ast))$ is free over $(S,R)$. Hence, for each $(T,\ast) \in rSCON$, the association
                $$ \varphi \mapsto U\varphi \circ \iota = \varphi|_S, \quad \forall \varphi \colon (S^R,\ast) \to (T,\ast) \in rSCON, $$
            is a bijective correspondence between (2,1)-morphisms $(S^R,\ast) \to (T,\ast)$ and (2,$R$)-morphisms $(S,R) \to U(T,\ast)$ in $\mathcal{S} \times \mathcal{R}$. From Remark \ref{obs:2Rmorphism} we conclude that $\alpha \mapsto \alpha|_S$ is a bijection between global actions of $(S^R,\ast)$ on $X$ and $R$-compatible actions of $S$ on $X$. That is, $(-)|_S \colon \mathcal{A}(S^R,\ast) \to \mathcal{A}_R(S)$ is a well defined functor that is bijective in objects.
            
            Lastly, let $(X;\alpha), (Y;\beta) \in \mathcal{A}(S^R,\ast)$. Since $(-)|_S \colon \mathcal{A}(S^R,\ast) \to \mathcal{A}_R(S)$ is a functor, if $\varphi \colon (X;\alpha) \to (Y;\beta)$ is a $S^R$-morphism, then $\varphi|_S = \varphi \colon (X;\alpha|_S) \to (Y;\beta|_S)$ is a $S$-morphism. Conversely, suppose that $\varphi \colon (X;\alpha|_S) \to (Y;\beta|_S)$ is a $S$-morphism. Then
                $$ \forall s \in S,\ [x \in {_sX} \implies \varphi(x) \in {_sY} \text{ and } \varphi(\alpha_s(x)) = \beta_s(\varphi(x))]. $$
            On the other hand, let $e \in S^R \setminus S = S/R = S^\ast$. Then there is $s \in S$ such that $e = s^\ast$. Hence,
                $$ x \in {_eX} \overset{\eqref{P5}}{=} {_sX} \implies \varphi(x) \in {_sY} \overset{\eqref{P5}}{=} {_eY} \quad\text{and}\quad \varphi(\alpha_e(x)) \overset{\eqref{P3}}{=} \varphi(x) \overset{\eqref{P3}}{=} \beta_e(\varphi(x)). $$
            This shows that $\varphi \colon (X;\alpha) \to (Y;\beta)$ is a $S^R$-morphism. Therefore, $(-)|_S$ induces a bijection between $S^R$-morphisms $(X;\alpha) \to (Y;\beta)$ and $S$-morphisms $(X;\alpha|_S) \to (Y;\beta|_S)$, concluding that $(-)|_S$ is a fully faithful functor and, hence, an isomorphism of categories.
        \end{proof}
    \end{prop}
    
    Similarly to actions of restriction semiconstellations, we say that the relative $S$-subalgebra $(Y;\beta)|_X$ of an $R$-compatible action of $S$ is the \textit{restriction} of $(Y;\beta)$ to the subset $X$.
    
    \begin{corollary} \label{coro:Rrestriction}
        The restriction of an $R$-compatible action of $S$ is a partial action of $S$.
    
        \begin{proof}
            Let $(Y;\beta)$ be an $R$-compatible action of $S$ and $X$ be a subset of $Y$. Then $(Y;\beta)|_S^{-1}$ is a global action of $(S^R,\ast)$. By Proposition \ref{prop:relative}, the relative $S^R$-subalgebra $(X;\alpha)$ of $(Y;\beta)|_S^{-1}$ is a partial action of $(S^R,\ast)$. Hence, $(X;\alpha)|_S$ is a partial action of $S$. But $(X;\alpha)|_S$ is precisely the relative $S$-algebra of $(Y;\beta)$ obtained by restricting $(Y;\beta)$ to $X$.
        \end{proof}
    \end{corollary}

    Since restrictions of $R$-compatible actions are partial actions, it is natural to question: \textbf{whether a partial action can be obtained by restricting an $R$-compatible action?} In the remainder of this section, we will show that the answer to this question lies within the globalization problem for partial actions of reduced semiconstellations.

    \subsection{Reducing the problem} We extend the functor $(-)|_S^{-1} \colon \mathcal{A}_R(S) \to \mathcal{A}(S^R,\ast)$ to a right inverse for the functor $(-)|_S \colon \mathcal{A}_p(S^R,\ast) \to \mathcal{A}_p(S)$, concluding that a partial action of $S$ is isomorphic to a restriction of an $R$-compatible action if and only if the corresponding partial action of $(S^R,\ast)$ is globalizable.\\
    
    In Proposition \ref{prop:actions-3}, the functor $(-)|_S \colon \mathcal{A}(S^R,\ast) \to \mathcal{A}_R(S)$ is an isomorphism of categories because, by Theorem \ref{teo:adequate}, every $R$-compatible action of $S$ extends uniquely to a global action of $(S^R,\ast)$ by setting $\alpha_{\overline{s} } =id_{_sX}$, for every $\overline{s}\in S/R$. The next result shows that every partial action of $S$ also admits an extension to a partial action of $(S^R,\ast)$, although such an extension need not be unique.
    
    \begin{lemma} \label{lema:actions}
        A partial action of $S$ extends to a partial action of $(S^R,\ast)$ if and only if 
            $$ \forall \overline{s} \in S/R, \quad \alpha_{\overline{s}} = id_{_{\overline{s}}X}, \quad\text{where}\quad \left( \bigcup_{t \in \overline{s}} {_tX} \right) \cup \left( \bigcup_{t \in S^s} X_t \right) \subseteq {_{\overline{s}}X}. $$
        In particular, the functor $(-)|_S \colon \mathcal{A}_p(S^R,\ast) \to \mathcal{A}_p(S)$ is surjective on objects.
    
        \begin{proof}
            Let $(X;\alpha)$ be a partial action of $S$. Note that, for $t \in S \subseteq S^R$, we have $t^\ast = \overline{s}$ if and only if $t \in \overline{s}$. Therefore, if $(X;\alpha)$ extends to a partial action $(X;\beta)$ of $(S^R,\ast)$, then
                $$ \alpha_{\overline{s}} \overset{\eqref{P3}}{=} id_{_{\overline{s}}X} \quad\text{and}\quad {_tX} \overset{\eqref{P4}}{\subseteq} {_{t^\ast}X} = {_{\overline{s}}X}, \ \forall t \in \overline{s}. $$
            On the other hand, $\overline{s}t$ is defined if and only if $st$ is defined, and in this case $\overline{s}t = t$. Hence,
                $$ X_t = \beta_t({_tX}) = \beta_t({_{\overline{s}t}X} \cap {_tX}) \overset{\eqref{P1}}{=} \beta_t(\beta_t^{-1}(X_t \cap {_{\overline{s}}X})) = X_t \cap {_{\overline{s}}X} \implies X_t \subseteq {_{\overline{s}}X}, \ \forall t \in S^s. $$
            This shows that a partial action of $(S^R,\ast)$ extending $(X;\alpha)$ must satisfy $\alpha_{\overline{s}} = id_{_{\overline{s}}X}$, where ${_{\overline{s}}X}$ contains the subsets ${_tX}$, for all $t \in \overline{s}$, and $X_t$, for all $t \in S^s$.
                
            Conversely, for each $\overline{s} \in S/R$, choose a subset ${_{\overline{s}}X}$ of $X$ containing ${_tX}$, for all $t \in \overline{s}$, and $X_t$, for all $t \in S^s$. Let $\alpha_{\overline{s}} = id_{_{\overline{s}}X}$. Then $(X;\{\alpha_a\}_{a \in S^R})$ is a partial action of $(S^R,\ast)$. In fact, conditions \eqref{P3} and \eqref{P4} hold by definition, while conditions \eqref{P1} and \eqref{P2} are satisfied whenever $s,t \in S$ and $st$ is defined. It is straightforward to verify conditions \eqref{P1} and \eqref{P2} when $\overline{s}t$, $s\overline{t}$ or $\overline{s}\overline{t}$ are defined.
    
            Note that every partial action $(X;\beta)$ of $(S^R,\ast)$ obtained above satisfies $(X;\beta|_S) = (X;\alpha)$. Therefore, the functor $(-)|_S \colon \mathcal{A}_p(S^R,\ast) \to \mathcal{A}_p(S)$ is surjective in objects. However, if there is $\overline{s} \in S/R$ such that
                $$ \left(\bigcup_{t \in \overline{s}} {_tX} \right) \cup \left(\bigcup_{t \in S^s} X_t \right) \neq X, $$
            then the partial action $(X;\beta)$ of $(S^R,\ast)$ satisfying $(X;\beta|_S) = (X;\alpha)$ is not unique.
        \end{proof}
    \end{lemma}
    
    \begin{prop} \label{prop:actions-4}
        There is a fully faithful functor $G_R \colon \mathcal{A}_p(S) \to \mathcal{A}_p(S^R,\ast)$ such that:
        \begin{enumerate}
            \item $G_R$ is a right inverse for $(-)|_S$. That is, $(-)|_S \circ G_R = id_{\mathcal{A}_p(S)}$.
            \item $G_R$ coincides with $(-)|_S^{-1}$ in $\mathcal{A}_R(S)$. That is, $G_R \circ I_S = I_{(S^R,\ast)} \circ (-)|_S^{-1}$.
        \end{enumerate}
    
        \begin{proof}
            Given a partial action $(X;\alpha)$ of $S$, we define
                $$ \alpha_{\overline{s}} = id_{_{\overline{s}}X}, \quad\text{where}\quad {_{\overline{s}}X} = \left( \bigcup_{t \in \overline{s}} {_tX} \right) \cup \left( \bigcup_{t \in S^s} X_t \right), $$
            for all $\overline{s} \in S/R$. By Lemma \ref{lema:actions}, the pair $G_R(X;\alpha) = (X;\{\alpha_a\}_{a \in S^R})$ is a partial action of $(S^R,\ast)$ such that $G_R(X;\alpha)|_S = (X;\alpha)$. Furthermore, if $(Y;\beta)$ is a partial action of $S$, $\varphi \colon (X;\alpha) \to (Y;\beta)$ is a $S$-morphism and $\overline{s} \in S/R$, then
                $$ x \in {_{\overline{s}}X} \implies [\exists t \in \overline{s} \colon x \in {_tX}] \text{ or } [\exists t \in S^s \colon x \in X_t]. $$
            In the first case, we have $\varphi(x) \in {_tY} \subseteq {_{\overline{s}}Y}$. In the second case, we can find $y \in {_tX}$ such that $x = \alpha_t(y)$ and, hence,
                $$ \varphi(x) = \varphi(\alpha_t(y)) = \beta_t(\varphi(y)) \in Y_t \subseteq {_{\overline{s}}Y}. $$
            In both cases, we have $\varphi(\alpha_{\overline{s}}(x)) = \varphi(x) = \beta_{\overline{s}}(\varphi(x))$. Thus, $\varphi \colon G_R(X;\alpha) \to G_R(Y;\beta)$ is a $S^R$-morphism. It is straightforward to see that $\varphi \mapsto G_R\varphi = \varphi$ preserves identities and compositions. Therefore, we have a functor $G_R \colon \mathcal{A}_p(S) \to \mathcal{A}_p(S^R,\ast)$, given by
                $$ (X;\alpha) \mapsto G_R(X;\alpha) \quad\text{and}\quad \varphi \colon (X;\alpha) \to (Y;\beta) \mapsto \varphi \colon G_R(X;\alpha) \to G_R(Y;\beta). $$
                
            By construction, the functor $G_R$ satisfies $(-)|_S \circ G_R = id_{\mathcal{A}_p(S)}$, proving (1). To see that $G_R$ is a fully faithful functor, note that $G_R\varphi = \varphi$ as functions. Hence, $G_R\varphi = G_R\psi$ implies $\varphi = \psi$. On the other hand, let $\varphi \colon G_R(X;\alpha) \to G_R(Y;\beta)$ be a $S^R$-morphism. Then $\varphi|_S \colon G_R(X;\alpha)|_S \to G_R(Y;\beta)$ is a $S$-morphism. From (1), we obtain $G_R(X;\alpha)|_S = (X;\alpha)$ and $G_R(Y;\beta)|_S = (Y;\beta)$. Therefore, $G_R(\varphi|_S) \colon G_R(X;\alpha) \to G_R(Y;\beta)$ is a $S^R$-morphism. Since $G_R(\varphi|_S) = \varphi$, every $S^R$-morphism $G_R(X;\alpha) \to (Y;\beta)$ is the image of a unique $S$-morphism $(X;\alpha) \to (Y;\beta)$ via the association $\varphi \mapsto G_R\varphi$.
            
            Finally, if $(X;\alpha)$ is an $R$-compatible action of $S$, then the unique global action $(X;\beta)$ of $(S^R,\ast)$ such that $(X;\beta)|_S = (X;\alpha)$ coincides with $G_R(X;\alpha)$. In fact, the functor $(-)|_S^{-1}$ extends $(X;\alpha)$ by defining $\alpha_{_{\overline{s}}X'} = id_{_{\overline{s}}X'}$, where ${_{\overline{s}}X'} = {_sX}$. Since $(X;\beta|_S) = (X;\alpha)$ and $s \in \overline{s}$, it follows from Lemma \ref{lema:actions} that
                $$ {_{\overline{s}}X'} = {_sX} \subseteq {_{\overline{s}}X} \subseteq {_{\overline{s}}X'}. $$
            Hence, ${_{\overline{s}}X'} = {_{\overline{s}}X}$ and $(X;\alpha)|_S^{-1} = G_R(X;\alpha)$, for all $(X;\alpha) \in \mathcal{A}_R(S)$, proving (2).
        \end{proof}
    \end{prop}
    
        Since $(S^R,\ast)$ is a reduced semiconstellation, it follows from Theorem \ref{teo:construction} that every partial action $(X;\alpha)$ of $(S^R,\ast)$ has a reflector $(\iota_X,(FX;\beta))$ in $\mathcal{A}(S^R,\ast)$. By Proposition \ref{prop:livre}, we obtain a free construction functor $F \colon \mathcal{A}_p(S^R,\ast) \to \mathcal{A}(S^R,\ast)$, which associates each partial action $(X;\alpha)$ of $(S^R,\ast)$ to the global action $F(X;\alpha) = (FX;\beta)$. Hence, we have a functor $F_R \colon \mathcal{A}_p(S) \to \mathcal{A}_R(S)$, defined by the composition
    \begin{center}
        \begin{tikzpicture}
            \tikzstyle{every path}=[draw,->];
    
            \node (partS) at (0,-1.5) {$\mathcal{A}_p(S)$};
            \node (partR) at (0,0) {$\mathcal{A}_p(S^R,\ast)$};
            \node (globR) at (5,0) {$\mathcal{A}(S^R,\ast)$};
            \node (globS) at (5,-1.5) {$\mathcal{A}_R(S)$};
    
            \path (partS) to node[left]{$G_R$} (partR);
            \path (partR) to node[above]{$F$} (globR);
            \path (globR) to node[right]{$(-)|_S$} (globS);
    
            \path[dashed] (partS) to node[above]{$F_R = (-)|_S \circ F \circ G_R$} (globS);
        \end{tikzpicture}
    \end{center}
    In particular, given $(X;\alpha) \in \mathcal{A}_p(S)$, we obtain a $S^R$-morphism $\iota_X \colon G_R(X;\alpha) \to FG_R(X;\alpha)$. Since $(-)|_S \circ G_R = id_{\mathcal{A}_p(S)}$, the function $\iota_X \colon (X;\alpha) \to F_R(X;\alpha)$ is a $S$-morphism.
    
    \begin{theorem} \label{teo:actions}
        Let $(X;\alpha)$ be a partial action of $S$. Then:
        \begin{enumerate}
            \item The pair $(\iota_X,F_R(X;\alpha))$ is a reflector for $(X;\alpha)$ in $\mathcal{A}_R(S)$.
    
            \item The partial action $(X;\alpha)$ is isomorphic to a restriction of an $R$-compatible action if and only if the corresponding partial action $G_R(X;\alpha)$ is globalizable.
        \end{enumerate}
    
        \begin{proof}
            Let $(\iota_X,FG_R(X;\alpha))$ be the reflector for $G_R(X;\alpha)$ in $\mathcal{A}(S^R,\ast)$. Then
            \begin{align*}
                (I_{(S^R,\ast)} \circ F \circ G_R)(X;\alpha) &= (I_{(S^R,\ast)} \circ id_{\mathcal{A}(S^R,\ast)} \circ F \circ G_R)(X;\alpha) \\
                &= (I_{(S^R,\ast)} \circ (-)|_S^{-1} \circ (-)|_S \circ F \circ G_R)(X;\alpha) & \ref{prop:actions-3} \\
                &= (G_R \circ I_S \circ (-)|_S \circ F \circ G_R)(X;\alpha) & \ref{prop:actions-4}(2) \\
                &= (G_R \circ I_S \circ F_R)(X;\alpha).
            \end{align*}
            That is, up to category inclusions, we have $FG_R(X;\alpha) = G_RF_R(X;\alpha)$. On the other hand, let $\varphi \colon (X;\alpha) \to (Y;\gamma)$ be a $S$-morphism, where $(Y;\gamma) \in \mathcal{A}_R(S)$. Since $G_R$ is a fully faithful functor and $FG_R(X;\alpha) = G_RF_R(X;\alpha)$, we have a bijective correspondence between $S$-morphisms $F_R(X;\alpha) \to (Y;\gamma)$ and $S^R$-morphisms $FG_R(X;\alpha) \to G_R(Y;\gamma)$. Moreover, we have a bijective correspondence between commutative diagrams
            \begin{center}
                \begin{tikzpicture}[yscale=3/4]
                    \tikzstyle{every path}=[draw,->];
    
                    \node (X) at (0,0) {$(X;\alpha)$};
                    \node (Y) at (0,-2) {$(Y;\gamma)$};
                    \node (FX) at (3,0) {$F_R(X;\alpha)$};
                    \path (X) to node[above]{$\iota_X$} (FX);
                    \path (X) to node[left]{$\varphi$} (Y);
                    \path (FX) to (Y);
    
                    \node at (5,-1) {and};
    
                    \node (X2) at (7,0) {$G_R(X;\alpha)$};
                    \node (Y2) at (7,-2) {$G_R(Y;\gamma)$};
                    \node (FX2) at (10,0) {$FG_R(X;\alpha)$};
                    \path (X2) to node[above]{$\iota_X$} (FX2);
                    \path (X2) to node[left]{$\varphi$} (Y2);
                    \path (FX2) to (Y2);
                \end{tikzpicture}
            \end{center}
            Since $G_R(Y;\gamma) = (Y;\gamma)|_S^{-1} \in \mathcal{A}(S^R,\ast)$ and $(\iota_X,FG_R(X;\alpha))$ is a reflector for $G_R(X;\alpha)$ in $\mathcal{A}(S^R,\ast)$, there is a unique $S^R$-morphism $\Phi \colon FG_R(X;\alpha) \to G_R(Y;\gamma)$ satisfying $\Phi \circ \iota_X = \varphi$, which corresponds to a unique $S$-morphism $\Phi \colon F_R(X;\alpha) \to (Y;\gamma)$ satisfying $\Phi \circ \iota_X = \varphi$. This shows that $(\iota_X,F_R(X;\alpha))$ is a reflector for $(X;\alpha)$ in $\mathcal{A}_R(S)$, proving (1).
            
            Now, it follows from (1) and Proposition \ref{prop:universal-2} that $(X;\alpha)$ is isomorphic to a restriction of an $R$-compatible action if and only if $\iota_X \colon (X;\alpha) \to F_R(X;\alpha)$ is an injective full $S$-morphism. On the other hand, by Corollary \ref{coro:reduced}, $G_R(X;\alpha)$ is globalizable if and only if $\iota_X \colon G_R(X;\alpha) \to FG_R(X;\alpha)$ is an injective full $S^R$-morphism. Clearly, $\iota_X \colon (X;\alpha) \to F_R(X;\alpha)$ is injective if and only if $\iota_X \colon G_R(X;\alpha) \to FG_R(X;\alpha)$ is injective.
            
            For the rest of the proof, we assume that $\iota_X$ is injective and denote $FG_R(X;\alpha) = (FX;\beta)$, the $S^R$-algebra constructed in Section \ref{sec:4}. Hence, if $\iota_X$ is a full $S^R$-morphism, then
                $$ \forall a \in S^R,\ [\iota_X(x) \in {_aFX} \text{ and } \beta_a(\iota_X(x)) \in im(\iota_X) \implies x \in {_aX}]. $$
            This property holds, in particular, for every $s \in S \subseteq S^R$. Therefore, $\iota_X$ a full $S$-morphism. Conversely, suppose that $\iota_X$ is a full $S$-morphism. Let $\overline{s} \in S/R$ and $x \in X$ be such that
                $$ \iota_X(x) \in {_{\overline{s}}FX} \quad\text{and}\quad \beta_{\overline{s}}(\iota_X(x)) \in im(\iota_X). $$
            Since $\beta_{\overline{s}} = id_{_{\overline{s}}FX}$, the second conditions is superfluous. On the other hand, $\iota_X(x) \in {_{\overline{s}}FX}$ means that $(1,x) \simeq (b,y)$, for some $(b,y) \in {_{\overline{s}}\overline{X}}$. By Lemma \ref{lema:technical}, we can choose $b \in (S^R)^{\overline{s}}$. If $b \in S/R$, then it must be $b = \overline{s}$, and in this case $(b,y) \in \overline{X}$ if and only if $y \in {_{b^\ast}X} = {_{\overline{s}}X}$. Thus,
                $$ \iota_X(x) = \overline{s} \otimes y \overset{(\simeq_\ast)}{=} 1 \otimes \alpha_{\overline{s}}(y) \overset{\eqref{P3}}{=} 1 \otimes y = \iota_X(y). $$
            But $\iota_X$ is injective, from where we obtain that $x = y \in {_{\overline{s}}X}$. Lastly, if $b \in S$, then
                $$ \iota_X(y) = 1 \otimes y \in {_bFX} \quad\text{and}\quad \beta_b(\iota_X(y)) = b \otimes y = 1 \otimes x = \iota_X(x) \in im(\iota(X)). $$
            Since $\iota_X$ is an injective full $S$-morphism, we obtain that $y \in {_bX}$, and in this case
                $$ \iota_X(x) = \beta_b(\iota_X(y)) = \iota_X(\alpha_b(y)) \implies x = \alpha_b(y) \in X_b. $$
            From the construction of $G_R(X;\alpha)$, we have $X_b \subseteq {_{\overline{s}}X}$, whenever $b \in S^s$. Therefore, $x \in {_{\overline{s}}X}$, concluding that $\iota_X$ is a full $S^R$-morphism.
        \end{proof}
    \end{theorem}
    
    From Theorem \ref{teo:actions}, we conclude that not every partial action of a semiconstellation is isomorphic to a restriction of an $R$-compatible action. In fact, recall that the reduced semiconstellation $(T_m^n,\ast)$ from Example \ref{exe:restriction} equals to $((S_m^n)^{R_M},\ast)$, where $S_m^n$ is the semiconstellation from Example \ref{exe:semiconst}, and note that the partial actions from Examples \ref{exe:reduced-2} and \ref{exe:unipotet-3} lie in the image of the functor $G_{R_M} \colon \mathcal{A}_p(S_2^1) \to \mathcal{A}_p(T_2^1,\ast)$. Since such partial actions are not globalizable, the corresponding partial actions of $S_2^1$ are not isomorphic to restrictions of $R_M$-compatible actions of $S_2^1$. Note also that, by Example \ref{exe:adequate-4}, $R_M$ is the unique adequate relation on $S_2^1$.
    
    The next result shows that, if $S$ is a semigroupoid, then every partial action of $S$ is isomorphic to the restriction of an $R$-compatible action of $S$, for any choice of adequate relation $R$ on $S$.
    
    \begin{corollary} \label{coro:Sactions}
        Let $S$ be a semigroupoid, $R$ be an adequate relation on $S$ and $(X;\alpha)$ be a partial action of $S$. Then $(X;\alpha)$ is isomorphic to a restriction of an $R$-compatible action.
    
        \begin{proof}
            We show that, if $S$ is a semigroupoid, then $\mathfrak{D}(S^R) \subseteq (S^R)^\ast$. By Theorem \ref{teo:suficient}, we obtain that $G_R(X;\alpha)$ is globalizable. Hence, by Theorem \ref{teo:actions}(2), $(X;\alpha)$ is isomorphic to a restriction of an $R$-compatible action of $S$.
    
            In fact, suppose that $d \in \mathfrak{D}(S^R) \cap S$. Then there exists $a \in S^R$ such that $da$ is defined and ${^{da}(S^R)} \neq {^{d}(S^R)}$. If $a \in S/R$, then $da = d$ and, thus, ${^{da}(S^R)} = {^{d}(S^R)}$, a contradiction. If $a \in S$, then $d,da \in S$, from where it follows that
                $$ {^{da}(S^R)} = {^{da}S} \cup \{ \overline{t} \colon t \in {^{da}S} \} \overset{\ref{prop:semiconst}}{=} {^{d}S} \cup \{ \overline{t} \colon t \in {^{d}S} \} = {^{d}(S^R)}, $$
            a contradiction. Thus, the initial assumption that $d \in \mathfrak{D}(S^R) \cap S$ is an absurd. Therefore,
                $$ \mathfrak{D}(S^R) \cap S = \emptyset \implies \mathfrak{D}(S^R) \subseteq S^R \setminus S = S/R = (S^R)^\ast. $$
        \end{proof}
    \end{corollary}
    
    Corollary \ref{coro:Sactions} generalizes the globalization theorem for partial actions of semigroupoid on sets \cite[Theorem 3.7]{haag2026a}. In this context, global actions of $S$ are precisely the $R'$-compatible actions, where $R'$ is the adequate relation from Remark \ref{obs:adequate}. In fact, \cite[Lemma 3.2]{haag2026a} shows that $R'$ is an adequate relation and every global action of $S$ is $R'$-compatible. The converse is straightforward. Therefore, \cite[Theorem 3.7]{haag2026a} coincides with Corollary \ref{coro:Sactions} for $R = R'$.

    \section{The particular case of semigroups} \label{sec:7}

    We now specialize the previous constructions to semigroups. This case deserves a separate discussion. Although every semigroup is a semiconstellation, the globalization result obtained in the previous section does not directly recover the globalization theorem for partial semigroup actions on sets given in \cite[Theorem 3.5]{kudryavtseva2023}. As we shall see, the difference lies in the way a partial action of a semigroup $S$ is extended to an action of the monoid $S^1$. 
    
    The aim of this section is to understand how the globalization theorem for partial semigroup actions fits into the constructions developed above. In particular, we shall compare the two resulting reflectors and show that, although they need not coincide, they arise from essentially the same construction.
    
    Throughout this section, $S$ denotes a semigroup, $R = S \times S = R_M$ denotes the maximum adequate relation on $S$ and $S^1 = S \cup \{1\} = S^{R}$ denotes the monoid obtained by adjoining an identity to $S$. We first recall the notions of partial, global and $R$-adequate actions in this setting.\\

    A partial action of a semigroup $S$, regarded as a semiconstellation, is a pair $(X;\alpha)$ consisting of a set $X$ and a family of partial functions $\alpha = \{\alpha_s\}_{s \in S}$ on $X$ such that
        $$ \forall s,t \in S, \quad \alpha_t^{-1}({_sX} \cap X_t) = {_{st}X} \cap {_tX} \quad\text{and}\quad \alpha_s(\alpha_t(x)) = \alpha_{st}(x), \quad \forall x \in {_{st}X} \cap {_tX}. $$
    This is equivalent to the \textit{strong partial semigroup actions} from \cite[Definition 2.3]{kudryavtseva2023}. The terminology \textit{partial action} in \cite{kudryavtseva2023} refers to a more general notion. However, \cite[Proposition 2.9]{kudryavtseva2023} shows that if a partial action can be globalized, then it is a strong partial action. Therefore, we refer to strong partial semigroup actions simply as partial actions.\\

    A \textit{global semigroup action} is a partial action $(X;\alpha)$ such that ${_sX} = X$, for all $s \in S$. This is equivalent to the function $\alpha \colon S \to PT(X)$ being a semigroup morphism whose image lies in the submonoid $T(X) = \{ f \in PT(X) \colon dom(f) = X \}$ of $PT(X)$. Clearly, such a function can be uniquely extended to a monoid morphism $S^1 \to T(X)$ by defining $\alpha(1) = id_X$, and every monoid morphism $S^1 \to T(X)$ restricts to a semigroup morphism $S \to T(X)$. Therefore, global semigroup actions correspond to monoid morphisms $S^1 \to T(X) \subseteq PT(X)$.
    
    On the other hand, an $R$-compatible action of $S$ is a partial action $(X;\alpha)$ such that ${_sX} = {_tX}$, for all $(s,t) \in R = S \times S$. Thus, it is a partial action where ${_sX}$ is constant. By Proposition \ref{prop:actions-3}, the $R$-compatible actions of $S$ are in bijection with global actions of the restriction semigroup $(S^1,\ast)$, endowed with the restriction structure given by $s^\ast = 1$, for all $s \in S^1$. The latter correspond to $(2,1)$-morphisms $(S^1,\ast) \to PT(X)$.

    It is straightforward that every monoid morphism $S^1 \to T(X) \subseteq PT(X)$ can be regarded as a $(2,1)$-morphism $(S^1,\ast) \to PT(X)$. Consequently, every global semigroup action is an $R$-compatible action, but the converse is not true. This raises the question:\\

    \noindent\textbf{Which $R$-compatible actions correspond to global semigroup actions?}\\

    We say that a partial action $(X;\alpha)$ is \textit{non-degenerate} if every element of $X$ lies in the domain or image of some function $\alpha_s$, that is, if it satisfies $\bigcup_{s \in S} ({_sX} \cup {X_s}) = X$. Every global semigroup action is a non-degenerate $R$-compatible action since ${_sX} = X$, for all $s \in S$. The converse is true. In fact, if $(X;\alpha)$ is an $R$-compatible action, then
        $$ {_sX} = {_tX}, \quad \forall (s,t) \in R = S \times S. $$
    From \eqref{P1} and ${_sX} = {_{ts}X}$, for any $t \in S$, we obtain that $X_s \subseteq {_tX}$. In particular, we have that ${_sX} \cup X_s = {_sX}$, for all $s \in S$. Therefore, if $(X;\alpha)$ is non-degenerate, then
        $$ X = \bigcup_{t \in S} ({_tX} \cup X_t) = \bigcup_{t \in S} {_tX} = {_sX}, \quad \forall s \in S. $$
    This shows that a non-degenerate $R$-compatible action $(X;\alpha)$ is a global semigroup action. Hence, global semigroup action are precisely the non-degenerate $R$-compatible actions of $S$. Before continuing the discussion, we fix and recall some notations.\\
    
    We are denoting by $\mathcal{A}_p(S)$ and $\mathcal{A}_R(S)$ the categories of partial and $R$-compatible actions of $S$, respectively. Let us denote by $\mathcal{A}(S)$ the category of global semigroup actions of $S$. Then, by the previous discussion, we have that
         $$ \mathcal{A}(S) \subseteq \mathcal{A}_R(S) \subseteq \mathcal{A}_p(S) $$
    as full subcategories. Note that we need not to define another category of partial semigroup actions, since this is already the category $\mathcal{A}_p(S)$. With this notation, Corollary \ref{coro:Sactions} provides a functor $F_R \colon \mathcal{A}_p(S) \to \mathcal{A}_R(S)$, whereas the globalization theorem from \cite{kudryavtseva2023} provides a functor $T \colon \mathcal{A}_p(S) \to \mathcal{A}(S)$. Both functors associate with a partial action $(X;\alpha)$ a reflector in the corresponding target category. Since $\mathcal{A}(S)$ is a proper subcategory of $\mathcal{A}_R(S)$ in general, these two reflectors need not coincide.

    To see this, suppose first that $(X;\alpha)$ is an $R$-compatible action. Since $(X;\alpha)$ already belongs to $\mathcal{A}_R(S)$, we have $F_R(X;\alpha) \simeq (X;\alpha)$. If $(X;\alpha)$ is degenerate, then $T(X;\alpha) \not\simeq (X;\alpha)$, since $T(X;\alpha)$ is a global semigroup action and hence non-degenerate. Therefore, $F_R(X;\alpha) \not\simeq T(X;\alpha).$ On the other hand, if $(X;\alpha)$ is non-degenerate, then, by the previous discussion, it is a global semigroup action. Hence, $F_R(X;\alpha) \simeq (X;\alpha)$ and $T(X;\alpha) \simeq (X;\alpha)$. 
    
    Thus, for an $R$-compatible action $(X;\alpha)$, we have $F_R(X;\alpha) \simeq T(X;\alpha)$ if and only if $(X;\alpha)$ is non-degenerate. This leads to the question:\\

    \noindent\textbf{For an arbitrary partial action $(X;\alpha)$, when do the actions $F_R(X;\alpha)$ and $T(X;\alpha)$ coincide?}\\

    The answer is the same: the actions $F_R(X;\alpha)$ and $T(X;\alpha)$ are isomorphic if and only if $(X;\alpha)$ is a non-degenerate partial action. This is a consequence of \cite[Proposition 3.8]{haag2026a}. More precisely, given a partial action $(X;\alpha)$, we denote
        $$ X_1 = \bigcup_{s \in S} ({_sX} \cup X_s) \quad\text{and}\quad X_0 = X \setminus X_1. $$
    That is, $X_1$ and $X_0$ are the non-degenerate and completely degenerate parts of $(X;\alpha)$, respectively. Let $F_R(X;\alpha) = (FX;\beta)$ be the globalization of $(X;\alpha)$ and $\iota \colon (X;\alpha) \to (FX;\beta)$ be the corresponding embedding. Then, the function $\iota$ restricts to a bijection $\iota \colon X_0 \to FX_0$. Consequently, $F_R(X;\alpha)$ is non-degenerate if and only if $X_0 \simeq FX_0 = \emptyset$ if and only if $(X;\alpha)$ is non-degenerate. This is the statement of \cite[Proposition 3.8]{haag2026a}, from where we obtain that
    \begin{align*}
        F_R(X;\alpha) \simeq T(X;\alpha) \iff F_R(X;\alpha) \text{ is non-degenerate} \iff (X;\alpha) \text{ is non degenerate}.
    \end{align*}
    
    We are now in a position to address the main question of this section:\\

    \noindent\textbf{How can the globalization theorem for partial semigroup actions be recovered from the constructions developed above?}\\

    Although Theorem \ref{teo:actions} does not recover this result directly, the construction behind it does. The key point is that the two globalization procedures differ in the way a partial action of $S$ is extended to an action of $S^1$.

    To make this precise, we first recall some notation. If $T$ is a monoid with identity $1_T$, a partial monoid action of $T$ is a partial semigroup action $(X;\alpha)$ satisfying $\alpha_{1_T}=id_X.$ A global monoid action is a global semigroup action satisfying the same condition. We denote the corresponding categories by $\mathcal{A}_p^1(T)$ and $\mathcal{A}^1(T)$, respectively.

    We now consider two different pairs of categories associated with $S^1$. Regarding $S^1$ as the restriction semigroup $(S^1,\ast)$, we have $\mathcal{A}(S^1,\ast) \subseteq \mathcal{A}_p(S^1,\ast)$. On the other hand, regarding $S^1$ simply as a monoid, we have $\mathcal{A}^1(S^1) \subseteq \mathcal{A}_p^1(S^1)$. Moreover, $\mathcal{A}^1(S^1) \subseteq \mathcal{A}(S^1,\ast)$ and $\mathcal{A}_p^1(S^1) \subseteq \mathcal{A}_p(S^1,\ast).$ These inclusions are summarized in the following diagram:

    \begin{center}
        \begin{tikzpicture}[xscale=1.2]
            \tikzstyle{every path}=[draw,->];

            \node (A1) at (0,0) {$\mathcal{A}^1(S^1)$};
            \node (A) at (-2,1) {$\mathcal{A}(S^1,\ast)$};
            \node (Ap1) at (2,1) {$\mathcal{A}_p^1(S^1)$};
            \node (Ap) at (0,2) {$\mathcal{A}_p(S^1,\ast)$};

            \path (A1) to (A); \path (A1) to (Ap1);
            \path (A) to (Ap); \path (Ap1) to (Ap);
        \end{tikzpicture}
    \end{center}

    By Corollary \ref{coro:suficient}, every partial restriction semigroup action is globalizable, and this globalization is pointed out by a functor $F \colon \mathcal{A}_p(S^1,\ast) \to \mathcal{A}(S^1,\ast)$. In the previous section, we constructed functors $(-)|_S \colon \mathcal{A}(S^1,\ast) \to \mathcal{A}_R(S)$, which associates a global action of $(S^1,\ast)$ to an $R$-adequate action by forgetting the action of the element $1$, and $G_R \colon \mathcal{A}_p(S) \to \mathcal{A}_p(S^1,\ast)$, which extends a partial action of $S$ to a partial action of $(S^1,\ast)$ by defining
        $$ {_1X} = \bigcup_{s \in S} ({_sX} \cup X_s) \quad\text{and}\quad \alpha_1 = id_{_1X}. $$
    By Lemma \ref{lema:actions}, the set ${_1X}$ is the smallest domain possible for $\alpha_1$ so the partial action $(X;\alpha)$ can be extended to a partial action of $(S^1,\ast)$. The condition $\alpha_1 = id_{_1X}$ is mandatory by \eqref{P3}. The composition
        $$ F_R := (-)_S \circ F \circ G_R \colon \mathcal{A}_p(S) \to \mathcal{A}_R(S) $$
    is the functor that associates a partial action $(X;\alpha)$ of $S$ to an $R$-adequate action whose restriction to a certain subset is isomorphic to $(X;\alpha)$.\\
    
    Similarly, it is proved in \cite[Theorem 3.5]{kudryavtseva2023} that every \textit{firm} partial semigroup action is globalizable, and every partial monoid action is firm. Consequently, there is a functor $t \colon \mathcal{A}_p^1(S^1) \to \mathcal{A}^1(S^1)$ which associates each partial monoid action to its globalization. Since $\mathcal{A}^1(S^1) \subseteq \mathcal{A}(S^1,\ast)$, we can restrict the functor $(-)|_S$ to $\mathcal{A}^1(S^1)$. The restriction
        $$ (-)_S \colon \mathcal{A}^1(S^1) \to \mathcal{A}(S) $$
    is an isomorphism of categories. In fact, it is enough to note that ${_sX} = X$, for all $s \in S$, implies that the $R$-compatible action $(X;\alpha)|_S$ is non-degenerate. Lastly, we claim that there is a functor $G_1 \colon \mathcal{A}_p(S) \to \mathcal{A}_p^1(S^1)$, which extends a partial action of $S$ to a partial monoid action of $S^1$ by defining
        $$ {_1X} = X \quad\text{and}\quad \alpha_1 = id_X. $$
    Furthermore, since a partial monoid actions must satisfy $\alpha_1 = id_X$, this is the unique way to extend $(X;\alpha)$ to a partial monoid action of $S^1$. Consequently, the functor $G_1$ is an isomorphism of categories. Therefore, the composition
        $$ T := (-)|_S \circ t \circ G_1 \colon \mathcal{A}_p(S) \to \mathcal{A}(S) $$
    is the functor that associate a partial action of $S$ to its globalization, in the sense of \cite{kudryavtseva2023}.\\
    
    The following diagram illustrate the globalization functors $F_R$ and $T$ as the composition of the remaining functors. The not labeled arrows are the category inclusions $\mathcal{A}^1(S^1) \subseteq \mathcal{A}(S^1,\ast)$ and $\mathcal{A}_p^1(S^1) \subseteq \mathcal{A}_p(S^1,\ast)$.

    \begin{center}
        \begin{tikzpicture}[xscale=1.2]
            \tikzstyle{every path}=[draw,->];

            \node (A1) at (0,0) {$\mathcal{A}^1(S^1)$};
            \node (A) at (-2,1) {$\mathcal{A}(S^1,\ast)$};
            \node (Ap1) at (2,1) {$\mathcal{A}_p^1(S^1)$};
            \node (Ap) at (0,2) {$\mathcal{A}_p(S^1,\ast)$};

            \path (A1) to (A);
            \path (Ap1) to (Ap);

            \node (Sp1) at (-3,2) {$\mathcal{A}_p(S)$};
            \node (SR1) at (-5,1) {$\mathcal{A}_R(S)$};

            \path (Sp1) to node[above]{$G_R$} (Ap);
            \path (A) to node[below]{$(-)|_S$} (SR1);
            \path (Ap) to node[below right]{$F$} (A);
            \path[dashed] (Sp1) to node[above left]{$F_R$} (SR1);

            \node (Sp2) at (5,1) {$\mathcal{A}_p(S)$};
            \node (SR2) at (3,0) {$\mathcal{A}(S)$};

            \path (Sp2) to node[above]{$G_1$} (Ap1);
            \path (A1) to node[below]{$(-)|_S$} (SR2);
            \path (Ap1) to node[above left]{$t$} (A1);
            \path[dashed] (Sp2) to node[below right]{$T$} (SR2);
        \end{tikzpicture}
    \end{center}

    Finally, we note that the functor $t$ is the restriction of the functor $F$ to $\mathcal{A}_p^1(S^1)$. In fact, it is enough to see that, if $(X;\alpha)$ is a partial action of $(S^1,\ast)$ such that $\alpha_1 = id_X$, then its globalization $F(X;\alpha) = (FX;\beta)$ satisfies $\beta_1 = id_{FX}$. This can be directly verified in the construction at the beginning of Section \ref{sec:4}: if $X = {_1X} = {_{s^\ast}X}$ for all $s \in S$, then $\overline{X} = S^1 \times X$, and thus ${_1\overline{X}} = \overline{X}$. Therefore,
        $$ {_1FX} = \{ a \otimes x \in FX \colon (a,x) \simeq (b,y) \text{ for some } (b,y) \in {_1\overline{X}} \} = FX, $$
    and the action $\beta_1 \colon FX \to FX$ is given by $\beta_1(a \otimes x) = a \otimes x$. Therefore, $\beta_1 = id_{FX}$. But then
        $$ F_R = (-)|_S \circ F \circ G_R \quad\text{and}\quad T = (-)|_S \circ F \circ G_1. $$
    That is, the difference between the globalizations obtained in \cite{kudryavtseva2023} and from Corollary \ref{coro:Sactions} is the replacement of the functor $G_R$ by the functor $G_1$, which extends a partial action of $S$ to a partial action of $S^1$ by defining ${_1X} = X$ instead of ${_1X} = \bigcup_{s \in S} ({_sX} \cup X_s)$.\\

    This gives another interpretation of \cite[Theorem 3.5]{kudryavtseva2023}: every partial action of a semigroup $S$ is isomorphic to the restriction of a non-degenerate $R$-compatible action of $S$. Furthermore, such non-degenerate $R$-compatible action is a reflector for the partial action in the category of non-degenerate $R$-compatible actions of $S$.
    
    We note that, for any semiconstellation $S$ and any adequate relation $R$ on $S$, Lemma \ref{lema:actions} allows us to extend a partial action $(X;\alpha)$ of $S$ to a partial action of $(S^R,\ast)$ by defining ${_{\overline{s}}X} = X$, for all $\overline{s} \in S/R$. It is therefore natural to ask whether the same idea can always be used to define a functor $G_1$ and, in this way, obtain a reflector in the category of non-degenerate $R$-compatible actions.

    In general, however, the answer is negative. As shown in \cite[Example 4.5]{haag2026a}, the statement that every partial action of a semigroupoid $S$ admits a non-degenerate globalization that is a reflector in the category of non-degenerate global actions of $S$ is false. Since the category $\mathcal{A}(S)$ of global semigroupoid actions coincides with the category $\mathcal{A}_{R'}(S)$ of $R'$-compatible actions, where $R'$ is the adequate relation from Remark \ref{obs:adequate}, the corresponding statement in terms of $R$-compatible actions is also false. Thus, although the globalization procedures for semigroups and semigroupoids are closely related, non-degeneracy plays a special role in the semigroup case.

    \section{Further directions and open problems} \label{sec:8}

    Some of the results obtained in this paper suggest natural extensions, while others lead to questions that we have not pursued here. In this final section, we discuss a selection of these questions.
    
    \subsection*{Reflectors and almost global actions}

    In Definition \ref{defi:almost}, we introduced the category $\mathcal{A}_a(S,\ast)$ of almost global actions of a restriction semiconstellation $(S,\ast)$, generalizing global actions of $(S,\ast)$ by weakening the partial action axioms \eqref{P1} and \eqref{P3}. By Theorem \ref{teo:construction}, every partial action $(X;\alpha)$ of $(S,\ast)$ has a reflector $(\iota,(FX;\beta))$ in the category $\mathcal{A}_a(S,\ast)$.

    By Theorems \ref{teo:suficient} and \ref{teo:reduced}, if $(S,\ast)$ is either a restriction semigroupoid or a reduced restriction semiconstellation, then the almost global action $(FX;\beta)$ is in fact global. This was enough to conclude that every partial action of a semigroupoid is globalizable, but the same does not hold for partial actions of reduced semiconstellations, and hence of restriction semiconstellations. This leads to the following question:\\

    \noindent\textbf{Is the almost global action $(FX;\beta)$ always a global action?}\\
    
    The main difficulty to answer this question positively is the lack of the property
        $$ [(s,x) \simeq (t,y) \text{ and } s,t \in S^p] \implies [(s,x) \in {_p\overline{X}} \iff (t,y) \in {_p\overline{X}}]. $$
    Under the additional assumption that disruptive elements act globally, Lemma \ref{lema:suficient} shows that this property holds for the relation $\simeq^{(1)}$, which in this case coincides with $\simeq$. As shown in the proof of Theorem \ref{teo:suficient}, this property alone is enough to guarantee that $(FX;\beta)$ is a global action.

    If the answer to the question above is positive, Proposition \ref{prop:construction-2} implies that a partial action $(X;\alpha)$ is globalizable if and only if the $S$-morphism $\iota \colon (X;\alpha) \to (FX;\beta)$ is an embedding. The globalization problem would then reduce to understanding the relation $\simeq$ well enough to determine when $\iota$ is an embedding.

    \subsection*{Non-reduced restriction semiconstellations}

    Section \ref{sec:5} was dedicated to the construction of the restriction semiconstellations $(S^R,\ast)$, where $R$ is an adequate relation on a semiconstellation $S$. In this construction, we wanted $S^R$ to contain $S$ as a (2)-subalgebra, and to be generated by $S$ as a $(2,1)$-algebra. These properties were useful in the construction of the functors between $\mathcal{A}_p(S^R,\ast)$ and $\mathcal{A}_p(S)$.
    
    Knowing that $S \subseteq S^R$ is a (2)-subalgebra if and only if the inclusion map is an (injective) strong (2)-morphism, we deduced in Lemma \ref{lema:adequate} that $R$ must be contained in the relation $R_M$. Asking $R$ to also contain the relation $R_m$ was useful in proving that $S^R$ is a semiconstellation, but restricts the construction to produce reduced semiconstellations. Thus:\\

    \noindent\textbf{How can the construction of the reduced restriction semiconstellations $(S^R,\ast)$ be generalized to obtain restriction semiconstellations that are not necessarily reduced?}\\

    For such a construction, it may be necessary to impose further conditions on the pairs $(S,R)$. For instance, if $(S^R,\ast)$ is not reduced, then it has a non-trivial natural partial order which is inherited by $S$, and the set $(S^R)^\ast = S/R$ is a disjoint union of semilattices. The following example may provide some indication of the additional conditions that are needed.

    Let $S$ be a semigroupoid and $R = R_M$. We denote by $\mathcal{P}_f(S,R)$ the family of finite non-empty subsets of the equivalence classes of $R$. In this case,
        $$ A \in \mathcal{P}_f(S,R) \iff A = \{s_1,\dots,s_n\} \quad\text{where}\quad S^{s_1} = \dots = S^{s_n}. $$
    Note that $\mathcal{P}_f(S,R)$ is a disjoint union of semilattices, where the partial order is reverse set inclusion. The set $T = \{ (A,a) \colon a \in A \in \mathcal{P}_f(S,R) \}$, endowed with the operation given by
        $$ \exists (A,a) \cdot (B,b) \iff \exists ab, \quad\text{and in this case}\quad (A,a) \cdot (B,b) = (Ab \cup B, ab), $$
    is again a semigroupoid. This semigroupoid has a partial order, given by $(A,a) \leq (B,b)$ if and only if $a = b$ and $B \subseteq A$. We define a relation $R_T$ on $T$ by
        $$ R_T = \{ ((A,a),(B,b)) \in T \times T \colon A = B \}.$$
    This is an equivalence relation that need not be adequate. Nevertheless, the set $T \cup T/R_T$ can be endowed with a restriction semiconstellation structure with the following properties: it contains $T$ as a $(2)$-subalgebra; it is generated by $T$ as a $(2,1)$-algebra; the set of projections coincides with $T/R_T$; and its natural partial order restricts to the natural partial order on $T$.

    The similarity between the second part of this construction and the one in Section \ref{sec:5} suggests that there may be a more general method for constructing restriction semiconstellations that are not necessarily reduced.

    \subsection*{Representation of restriction semiconstellations}

    As a first application of Proposition \ref{prop:adequate}, we obtained a representation theorem for semiconstellations, showing that, up to isomorphism, they are precisely the $(2)$-subalgebras of function constellations $C(X)$.

    The proof uses the fact that the reduced restriction semiconstellations $(S^R,\ast)$ are constellations, which allows us to apply a result from \cite{gould2009restriction} and obtain an injective strong $(2,1)$-morphism $(S^R,\ast) \to C(S^R)$. Thus, every reduced restriction semiconstellation is isomorphic to a $(2,1)$-subalgebra of a function constellation. The converse, however, does not hold: $(2,1)$-subalgebras of function constellations need not be reduced restriction semiconstellations.
    
    On the other hand, \cite[Theorem 3.9]{haag2026c} characterizes restriction semigroupoids, up to isomorphism, as the $(2,1)$-subalgebras of categories of partial functions $PT(\mathcal X)$, where $\mathcal X$ is a family of pairwise disjoint sets. This suggests the following question:\\

    \noindent\textbf{Can restriction semiconstellations be characterized by a suitable representation theorem?}

    \subsection*{Reflectivity of \texorpdfstring{$R$}{}-compatible actions}

    Our last question concerns the relation between reflectors for partial actions of semiconstellations and those for partial actions of restriction semiconstellations. Given a relation $R$ on a semiconstellation $S$, we defined $\mathcal{A}_R(S)$ as the full subcategory of $\mathcal{A}_p(S)$ whose objects are the partial actions $(X;\alpha)$ satisfying ${_sX} = {_tX}$, whenever $(s,t) \in R$.
    
    When $R$ is an adequate relation on $S$, we proved that $\mathcal{A}_R(S)$ is a reflective subcategory of $\mathcal{A}_p(S)$. The proof proceeds by relating the inclusion $\mathcal{A}_R(S) \subseteq \mathcal{A}_p(S)$ to the inclusion $\mathcal{A}(S,\ast) \subseteq \mathcal{A}_p(S,\ast)$, and then applying the construction of reflectors for partial actions of reduced restriction semiconstellations from Theorem \ref{teo:reduced}. 

    It is natural to ask whether the assumption that $R$ is adequate can be weakened. This is particularly relevant when $S$ already carries the structure of a restriction semiconstellation $(S,\ast)$. In this case, the unique adequate relation on $S$ is $R_m = \{ (t,st) \in S \times S \colon \exists st \}$. On the other hand, every global action of $(S,\ast)$ belongs to $\mathcal{A}_{R_\ast}(S)$, where $R_\ast = \{(s,t) \in S \times S \colon s^\ast = t^\ast \}$. The relation $R_\ast$ is adequate precisely when $(S,\ast)$ is reduced. Thus, for a general restriction semiconstellation, the reflectivity result for adequate relations cannot be applied directly to $\mathcal{A}_{R_\ast}(S)$.

    If $\mathcal{A}{R_\ast}(S)$ were reflective in $\mathcal{A}_p(S)$ under more general conditions, then the inclusions $\mathcal{A}(S,\ast) \subseteq \mathcal{A}_{R_\ast}(S)$ and $\mathcal{A}_p(S,\ast) \subseteq \mathcal{A}_p(S)$ could provide a way to recover reflectors for partial actions of $(S,\ast)$ from reflectors for partial actions of $S$. This would, in a sense, reverse the strategy used in Section \ref{sec:6}. We are therefore led to the following question:\\

    \noindent\textbf{For which relations $R$ is the category $\mathcal{A}_R(S)$ a reflective subcategory of $\mathcal{A}_p(S)$?} \\

    Adequacy is not necessary for reflectivity. A simple example is the diagonal relation $R = \{(s,s) \colon s \in S\}$. In this case, $\mathcal{A}_R(S) = \mathcal{A}_p(S)$, so $\mathcal{A}_R(S)$ is trivially reflective in $\mathcal{A}_p(S)$. However, $R$ is adequate if and only if $st = t$ whenever $st$ is defined. Thus, the class of relations for which $\mathcal{A}_R(S)$ is reflective may be strictly larger than the class of adequate relations.


    \subsection*{Acknowledgment}
    T. Tamusiunas was partially supported by Conselho Nacional de Desenvolvimento Científico e Tecnológico (Brazil) through a Productivity Research Fellowship, Grant No. 303411/2025-2, and by Conselho Nacional de Desenvolvimento Científico e Tecnológico (Brazil) under the Universal Call, Grant No. 403606/2025-0.

    \bibliographystyle{abbrvnat}
    {\small \label{sec:ref}\bibliography{ref}}
    
\end{document}